\documentclass[twoside,11pt,preprint]{article}

\usepackage{jmlr2e}

\usepackage{amsmath}
\usepackage{bm}
\usepackage{bbm}

\usepackage{algorithm}
\usepackage{algpseudocode}

\newcommand\Algphase[2]{%
  \vspace*{-.5\baselineskip}%
  \Statex\hspace*{\dimexpr-18pt\relax}%
  \Statex\hspace*{-18pt}\textbf{Step #1:} #2}

\newcommand{\sgn}{\mathrm{sgn}}

\global\long\def\Reals{\mathbb{R}}

\global\long\def\intd{\mathrm{d}} 

\newcommand{\1}{\boldsymbol{1}}

\DeclareMathOperator{\sign}{sgn}

\newcommand{\EE}{\mathbb{E}}

\newcommand{\PP}{\mathbb{P}}

\usepackage{mathtools}
\RequirePackage{xcolor}

\providecommand\given{} 
\newcommand\SetSymbol[1][]{
	\nonscript\,#1:\nonscript\,\mathopen{}\allowbreak}
\DeclarePairedDelimiterX\Set[1]{\lbrace}{\rbrace}%
{ \renewcommand\given{\SetSymbol[]} #1 }
\DeclarePairedDelimiter{\Inner}{\langle}{\rangle}

\newcommand{\Block}{\mathfrak{B}}

\DeclarePairedDelimiter{\braces}{\{}{\}}
\DeclarePairedDelimiter{\ip}{\langle}{\rangle}
\DeclarePairedDelimiter{\brackets}{(}{)}
\DeclarePairedDelimiter{\sqbrackets}{[}{]}
\DeclarePairedDelimiter{\norm}{\lVert}{\rVert}
\DeclarePairedDelimiter{\normF}{\lVert}{\rVert_{\textnormal{F}}}
\DeclarePairedDelimiter{\abs}{\lvert}{\rvert}
\DeclarePairedDelimiter{\floor}{\lfloor}{\rfloor}
\DeclarePairedDelimiter{\ceil}{\lceil}{\rceil}

\newcommand{\RR}{\mathbb{R}}
\newcommand{\NN}{\mathbb{N}}
\newcommand{\II}{\mathbbm{1}}

\newcommand{\sg}{\gamma^*}

\newcommand{\psgap}{\gamma_{\mathrm{ps}}}
\newcommand{\SmoothClass}{\Theta_{\delta,\epsilon,\zeta}^{s_0,s_1}(R)}
\newcommand{\RegulClass}{\Sigma_{\sg}(L)}

\newcommand{\Father}{f^{\Phi_{Jk}}}
\newcommand{\Mother}{f^{\Psi_{jk}}}
\newcommand{\HatFather}{\hat{f}^{\Phi_{Jk}}}
\newcommand{\HatMother}{\hat{f}^{\Psi_{jk}}}

\newcommand{\A}{\alpha}
\newcommand{\B}{\beta}
\newcommand{\AMother}{\A^{\Psi_{jk}}}
\newcommand{\HatAMother}{\hat{\A}^{\Psi_{jk}}}
\newcommand{\BMother}{\B^{\Psi_{jk}}}
\newcommand{\HatBMother}{\hat{\B}^{\Psi_{jk}}}

\usepackage{soulutf8}
\usepackage[textsize=tiny]{todonotes}
\newcommand{\REV}{}
\newcommand{\ENDREV}{}

\usepackage{lastpage}

\firstpageno{1}

\begin{document}

\title{Frontiers to the learning of nonparametric hidden Markov models}

\author{\name Kweku Abraham \email lkwa2@cam.ac.uk \\
       \addr University of Cambridge, Statistical Laboratory\\
       Wilberforce Road\\
       Cambridge CB3 0WB, UK
         \AND
       \name Elisabeth Gassiat \email elisabeth.gassiat@universite-paris-saclay.fr \\
       \addr Universit\'e Paris-Saclay, CNRS\\
       Laboratoire de mathématiques d’Orsay\\
       91405, Orsay, France
              \AND
       \name Zacharie Naulet \email znaulet@inrae.fr \\
 \addr Université Paris-Saclay, INRAE\\
 MaIAGE\\
 78350, Jouy-en-Josas, France }

\maketitle

\begin{abstract}
  Hidden Markov models (HMMs) are flexible tools for clustering dependent data coming from unknown populations, allowing nonparametric modelling of the population densities. Identifiability fails when the data is in fact independent and identically distributed (i.i.d.), and we study the frontier between learnable and unlearnable two-state nonparametric HMMs. Learning the parameters of the HMM requires solving a nonlinear inverse problem whose difficulty depends not only on the smoothnesses of the populations but also on the distance to the i.i.d. boundary of the parameter set. The latter difficulty is mostly ignored in the literature in favour of assumptions precluding nearly independent data. This is the first work conducting a precise nonasymptotic, nonparametric analysis of the minimax risk taking into account all aspects of the hardness of the problem, in the case of two populations. Our analysis reveals an unexpected interplay between the distance to the i.i.d. boundary and the relative smoothnesses of the two populations: a surprising and intriguing transition occurs in the rate when the two densities have differing smoothnesses. We obtain upper and lower bounds revealing that, close to the i.i.d. boundary, it is possible to ``borrow strength'' from the estimator of the smoother density to improve the risk of the other.
\end{abstract}


\section{Introduction}
\label{sec:intro}

\subsection{Context and aim}
\label{sec:context}

Hidden Markov Models (HMMs) are a class of probabilistic models that play an important role in computer science and machine learning, particularly in the analysis of data sequences. They are widely used in various applications, including speech recognition and natural language processing, due to their ability to model hidden states that evolve over time. This makes them ideal for capturing the evolution of sequences from different populations, effectively functioning as time-varying mixture models. Mixture models used for i.i.d. data require modelling assumptions on the population distributions (also called emission distributions), for example that they come from a parametric distribution; an advantage of HMMs is that identification can be obtained without such prior modelling \citep{MR3889699}. Thus, HMMs can be viewed as nonparametric mixture models that allow for greater flexibility in the emission distributions, making them particularly valuable in machine learning for their adaptability and robustness \citep{CC00,Lef03,LWM03,SC09,MR2797735}.
Such flexibility has been discovered and studied in the recent years, see Section \ref{sec:related} for references and discussion. However, all theoretical results in this literature are  asymptotic in nature, that is with the length $n$ of the data sequence tending to infinity while model parameters are fixed. When the sequence of data is not far from being a sequence of i.i.d. observations, algorithms become unstable, making the output of the algorithms questionable \citep{Phys:2020}. 
This is due to the fact that nonparametric mixtures are highly nonidentifiable and that identification algorithms for nonparametric HMMs proposed in previous literature involve tuning parameters for which no clues are given to address this issue. 
\REV Indeed, in HMMs, the set of hidden Markov chain parameters and emission distributions can be divided into two subsets, the one for which the observations are not independent random variables (where identification is possible) and the one for which they form an i.i.d. sequence (where identification becomes impossible), and these two subsets share a boundary. Approaching the boundary makes learning more difficult.

The aim of our paper is to understand, in the possible learning properties of nonparametric HMMs,  the interplay between the closeness to this boundary and the number of observations. \ENDREV
The method we adopt for this purpose is to obtain nonasymptotic minimax rates in which the dependence to the i.i.d. frontier appears clearly together with the usual parameters such as the number of observations and the smoothness of probability emission densities. To obtain the upper bound, we propose a new estimation method which is straightforward to implement.

\subsection{Contributions}
\label{sec:contributions}

We consider a two-state HMM with real-valued emissions, in which we observe the first $n$ entries of a sequence $\bm{Y}=(Y_1,Y_2,\dots)\in [0,1]^\NN$ which, under a parameter $\theta=(p,q,f_0,f_1)$, satisfies 
\begin{equation}\label{eqn:model}
	\begin{split}
		\PP_\theta (Y_n\in A \mid \bm{X}) &= \int_A f_{X_n}(y)\intd y,\\
		\bm{X}=(X_n)_{n\in \NN} &\sim \operatorname{Markov}(\pi,Q_\theta),
	\end{split}
\end{equation}
with the $Y_n,~n\in\NN$ conditionally independent given $\bm{X}$.
The vector $\bm{X}$ of `hidden states', which we assume is started from its invariant distribution $X_1\sim \pi$, takes values in $\braces{0,1}^\NN$. The transition matrix of the
chain is given by
\begin{equation}
	Q=Q_\theta%
	\coloneqq%
	\begin{pmatrix}
		1 - p & p\\
		q & 1 - q
	\end{pmatrix}%
	,
\end{equation}
with the convention that for $j\geq 1$,
$\PP_{\theta}(X_{j+1} = 0 \mid X_j = 0) = 1-p<1$ and
$\PP_{\theta}(X_{j+1} = 0 \mid X_j = 1) = q>0$. The functions $f_0,f_1\in L^2([0,1])$ are density functions. Thus all $Y_k$, $k\geq 1$ follow the mixture distribution $\pi_0 f_0 + \pi_1 f_1$. 

The goal is to estimate the parameter $\theta$. This is a nonlinear inverse problem known to be solvable, up to a label-switching issue, even without any modelling assumptions on $f_0$ and $f_1$ \citep{MR3439359, MR3509896}: specifically, given that the highly non-identifiable i.i.d. nonparametric mixture  is a degenerate submodel of a HMM, under conditions which rule out independence. There are three ways in which the data $(Y_n)_{n\in\NN}$ can fail to exhibit dependence: when the hidden states themselves are in reality independently distributed; when the emission distributions are identical; or when only one population is observed. 
We adopt the minimax paradigm and we analyse the smallest maximum risk attainable over the following class of parameters. We define for some $\delta,\epsilon\in (0,1)$ and some $\zeta,s_0,s_1,R>0$ \begin{equation}\label{eqn:smoothclass}
	\SmoothClass%
	\coloneqq
	\braces{ \theta\;:  p,q \geq\delta,\;\abs{1-p - q} \geq \epsilon,\:\norm{f_{0}-f_{1}}_{L^2} \geq \zeta, \; \norm{f_i}_{B_{2,\infty}^{s_i}}\leq R}.
\end{equation}
Here  $\norm{\cdot}_{B_{2,\infty}^s}$ denotes a Besov norm whose precise definition as used in this paper is delayed to equation~\eqref{eq:def:besov} below. \REV The space $B_{2,\infty}^s$ can be thought to be similar to the subspace of $s$-times differentiable functions with continuous $s$-derivative that are square-integrable, but it allows for slightly more general functions with comparable smoothness. We refer to \cite{triebel1983} for a thorough introduction to Besov spaces and their history. \ENDREV The quantities $\delta,\epsilon$ and $\zeta$ lower bound the ``distance'' to the i.i.d.\ submodel. Indeed if $\delta=0$, we may be unable to estimate both $f_0$ and $f_1$ since we may see data from one of these alone; if $\zeta=0$ we may be unable to estimate $p$ and $q$; and if $\epsilon=0$ then we may be unable to identify the contributions of $f_0$ and $f_1$ to the mixture $\pi_0 f_0 + \pi_1f_1$. We use concentration inequalities for Markov chains \citep{Paulin2015} to build our estimators. 
This requires us to slightly shrink the set $\SmoothClass$ and restrict our attention to parameters that are also in
\begin{equation}
  \label{eqn:def:RegulClass}
  \RegulClass%
	\coloneqq \Set{\theta \given 1-\abs{1-p-q} \geq \sg,\ \max_{j=0,1}\norm{f_j}_{\infty} \leq L},
\end{equation}
i.e.\ parameters with uniformly bounded emission densities (here $\norm{\cdot}_{\infty}$ denotes the usual supremum norm) and having an absolute spectral gap. \REV The assumption that the Markov chain starts from its stationary distribution could be relaxed as explained in \cite[Section~3.3]{Paulin2015}, at the price of increasing the constants in the upper bounds, and longer proofs. \ENDREV We throughout use $\PP_{\theta}$ to denote the law of $(\bm{X},\bm{Y})$, and all induced marginal and conditional laws. 

We are mainly interested in the regimes where $\delta,\epsilon,\zeta$ can be eventually small, and how the minimax risks for $Q$ and $f_0,f_1$ over $\SmoothClass$ are affected in these regimes.%

The main message of our theorems may now be stated informally as follows (up to label switching and technical details relative to smoothnesses). The symbol $\asymp$ in the theorem means that expressions on the left and right side of $\asymp$ are proportional with a proportionality constant eventually depending on $R$, $L$ and the absolute spectral gap of the chain $\bm{X}$, but nothing else.
\begin{theorem}[Informal]
\label{th:informal}
  The minimax rate for estimating the transition matrix $Q$ satisfies, for any norm $\norm{\cdot}$,
  \begin{equation*}
    \inf_{\hat{Q}}\sup_{\theta\in \SmoothClass\cap \RegulClass}\EE_{\theta}\big(\|\hat{Q} - Q\|^2 \big)%
    \asymp \frac{\max(\delta,\epsilon\zeta)^2}{\delta^2\epsilon^4\zeta^6}\frac{1}{n}.
  \end{equation*}
  The minimax rates for estimating $f_0$ and $f_1$ when $s_0=s_1=s$ satisfy
  \begin{equation*}
    \inf_{\hat{f}_j}\sup_{\theta\in \SmoothClass \cap \RegulClass}\EE_{\theta}\big(\|\hat{f}_j - f_j\|_{L^2}^2 \big)%
    \asymp \Big(\frac{1}{\delta^2\epsilon^2\zeta^2n}
    \Big)^{2s/(2s+1)}%
    + \frac{1}{\delta^2\epsilon^4\zeta^4n},
  \end{equation*}
  while if $s_0 > s_1$ they satisfy
  \begin{align*}
    \inf_{\hat{f}_0}\sup_{\theta\in \SmoothClass \cap \RegulClass}\EE_{\theta}\big(\|\hat{f}_0 - f_0\|_{L^2}^2 \big)%
    &\asymp \Big(\frac{1}{\delta^2\epsilon^2\zeta^2n}
    \Big)^{2s_{0}/(2s_{0}+1)}%
      + \frac{1}{\delta^2\epsilon^4\zeta^4n}\\
    \inf_{\hat{f}_1}\sup_{\theta\in \SmoothClass \cap \RegulClass}\EE_{\theta}\big(\|\hat{f}_1 - f_1\|_{L^2}^2 \big)%
    &\asymp \Big(\frac{1}{\delta^2 n}
    \Big)^{2s_{1}/(2s_{1}+1)}%
      + \frac{1}{\delta^2\epsilon^4\zeta^4n},
  \end{align*}
  and correspondingly if $s_0 < s_1$.
\end{theorem}

 For a formal and rigorous statement of the minimax lower and upper bounds, we refer to Theorems~\ref{theo:lower:param}, \ref{thm:2} (lower bounds), and to Theorems~\ref{thm:concent-pq}, \ref{thm:1}, \ref{thm:rough}, and their Corollaries~\ref{cor:smooth-ub}, \ref{cor:rough-ub} (upper bounds). The precise theorems are stated in a nonasymptotic manner. The asymptotic leading terms given in the above main results are in the case where the ``distance'' to frontier is large compared to $n^{-a}$ for some (precisely defined) $a$. In this regime, the transition between the situation where emission densities have similar or different smoothnesses can be described as  ``$s_0 = s_1$'' or ``$s_0 > s_1$'', but the transition appears in a more intricate manner when taking a nonasymptotic point of view. However, the main message is that some transition in the minimax rate occurs depending on the relative smoothnesses of the emission densities.

The transition in the rates arises due to a simple but unexpected phenomenon we call ``sharing estimation strength'', that can be described informally as follows. 
It is possible to estimate the combination $\psi_1=\pi_0f_0 + \pi_1 f_1$ at a good rate because it is simply the invariant density of $Y_n$. Hence a reasonable density estimator can estimate $\psi_1$ at rate $n^{-s/(1+2s)}$ where $s$ is the smoothness of $\psi_1$, with no dependence on $\epsilon,\delta,\zeta$. In the case where $f_0$ is much smoother than $f_1$, it may be more efficient to estimate $f_0$ and $\psi_1$, and estimate $f_1$ by plug in, rather than directly estimating $f_1$. This is reflected both in the upper bounds (see Theorem \ref{thm:1} and Theorem \ref{thm:rough}) and the lower bounds (see Theorem \ref{thm:2}). The precise analysis of how one can ``borrow'' strength from the estimator of the smoother emission density to improve on the estimation rate for the rougher emission density is more involved, but this is the inspiration behind it.

\subsection{Related work}
\label{sec:related}

It has been proved in \citep{MR3439359, MR3509896} that 
once  i.i.d.\ submodels are excluded, consistent estimation is possible for nonparametric HMMs without prior modelling assumptions of the emission distributions. Moreover, no cost is incurred relative to the case where the underlying labels are observed. For $s$-smooth probability densities, the minimax rate $n^{-s/(1+2s)}$ is achieved using tensor methods in \citep{dCGlC17}  and using penalized least-squares estimation in \citep{MR3543517}. This rate can be achieved adaptively in a ``state-by-state'' manner: up to a label-switching issue, one can achieve the rate $n^{-s_j/(1+2s_j)}$ if $f_j$ has smoothness $s_j$, without knowledge of $(s_j,~j=0,1)$, see \citep{MR3862446}. See also \citep{lecestre} for robust estimation of the law of the observations in finite state space HMMs.  

\REV

Earlier works do not consider the tradeoff between the required sample size and the required ``distance'' from independence, and it is this tradeoff that forms the focus of the current work, continuing from the previous article \citep{AGNparamhmm} in which we considered the model \eqref{eqn:model} but with $f_0,f_1$ densities with respect to counting measure on $\braces{1,\dots,K}$ with known $K$. Discrete modeling is restrictive  and extending the study to continuous densities with nonparametric modeling is important for applications. Some of the results in the continuous case mirror their discrete counterparts. For instance the minimax rate for estimating $Q$ remains unchanged, though this is less trivial than it appears. While this might look obvious because for any function $h : [0,1]\to \Set{1,\dots,K}$, the pairs $((X_n,h(Y_n))_{n\geq 0}$ form a hidden Markov model with the same transition matrix $Q$. Finding a $h$ for which $Q$ is still identifiable from $(h(Y_n))_{n\geq 0}$ is however not straightforward, and it turns out that estimating $Q$ requires first to solve a nonparametric problem (see Section~\ref{sec:estim-separ-hyperpl-1}). Moreover the nonparametric setting exhibits striking \emph{qualitative}, as well as \emph{quantitative}, differences relative to the discrete case. The rates for $f_0$ and $f_1$ in the nonparametric setting arise from delicate interplay between the smoothnesses $s_0,s_1$ and the parameters $\delta,\epsilon,\zeta$. Also, the dependence of these rates in $\delta,\epsilon,\zeta$ differ between the discrete and the continuous case. A detailed comparison between this work and \cite{AGNparamhmm} can be found in Section~\ref{sec:previous-works}.

One additional novelty relative to other HMM papers in the nonparametric setting is that we use a wavelet block thresholding estimator. This allows us to adapt to the smoothnesses $s_0$ and $s_1$ without needing to use Lepski's method, and is thus, at least in principle, more computationally feasible.

\ENDREV

\subsection{Organisation of the paper}


In Section~\ref{sec:lower-bounds} we give the lower bounds on the minimax risk for estimating $Q$ and the densities $f_0$ and $f_1$. In Section~\ref{sec:upper-bounds} we derive the matching upper bounds. It is worth noting that the upper bounds are obtained via construction of estimators that are explicit and can be computed efficiently. Section~\ref{sec:discuss} is devoted to the discussion of questions left open in our work. Proofs are relegated to the appendices.


\section{Lower bounds}
\label{sec:lower-bounds}

We give a lower bound for each component $p$ and $q$ of $Q$ separately, which implies a bound for estimating (a permutation of) $Q$ in any norm (since $Q$ is a $2\times 2$ matrix). The proof of Theorem 
\ref{theo:lower:param} can be found in Section \ref{proof:lower:param}. In the theorem $\epsilon_0>0$ is a constant whose precise value can in principle be computed.

\begin{theorem}
\label{theo:lower:param}
Assume $n\delta^2\epsilon^4\zeta^6 \geq 1$, $\zeta\leq 1/(4\sqrt{3})$,  $\epsilon\leq \epsilon_0$, $\delta\leq 1/6$,  $R\geq 5/4 + 1/(8\sqrt{3})$ and $L\geq 5/8$. 
Then there exists a constant $c>0$  such that
$$
 \inf_{\hat{p}} \sup_{\theta\in \SmoothClass \cap \RegulClass}\EE_{\theta}\big(|\hat{p} - p|^2 \big)
    \geq
     \frac{c \max(\delta^2,\epsilon^2\zeta^2)}{\delta^2\epsilon^4\zeta^6}\frac{1}{n}
$$
where the  infimum is over all estimators $\hat{p}$ based on $Y_{1},\ldots,Y_{n}$. The same lower bound holds for the estimation of $q$.
\end{theorem}

We now consider the lower bounds for the estimation risk of the emission densities. Note that the lower bounds do not follow from standard density estimation (as in \citep{tsybakov:2009}) because density estimation is not a submodel of HMM when one excludes the i.i.d. boundary of the parameter set. Surprisingly this fact appears to have been overlooked until the recent work of \citep{MR4576679} where the first rigorous minimax lower bounds for estimating the densities have been established (see Section~C of therein). The arguments therein rely on reducing to the simpler model where $\bm{X}$ is observed (so that the problem reduces to standard density estimation with two independent samples); 
	this reduction is too severe to characterise the precise dependence of the minimax risk on $\delta$, $\epsilon$ and $\zeta$. To bypass the reduction to density estimation requires understanding the Kullback--Leibler divergence between arbitrary HMM distributions, which is challenging because of dependency. We establish the rates with the correct constants in the next theorem, whose proof can be found in Section \ref{proof:thm:2}.
\begin{theorem}
  \label{thm:2}
  Assume $n\delta^2\epsilon^2\zeta^4 \geq 1$, $\zeta\leq 1/(4\sqrt{3})$, $\epsilon\leq \epsilon_0$ for a suitable 
$\epsilon_0>0$, $\delta\leq 1/6$,  $R\geq 5/4 + 1/(8\sqrt{3})$ and $L\geq 5/8$. 
Then there exists a constant $c>0$  such that
\begin{equation}
\label{lower:rate:rough}
 \inf_{\check{f}_{0}}   \sup_{\theta\in \SmoothClass \cap \RegulClass}\EE_{\theta}\Big( \norm{\check{f}_{0} - f_{0}}_{L^2}^2 \Big)%
    \geq c \left\{\frac{1}{\delta^2\epsilon^4\zeta^4 n} + \left(\frac{1}{\delta^2 n}\right)^{2s_{0}/(2s_{0}+1)}\right\}.
\end{equation}
If moreover it holds $(n\delta^2\epsilon^2\zeta^4)^{-s_0/(1+2s_0)}\leq c_0 \zeta$ and $\delta^{2s_{1}+1}(n\epsilon^2\zeta^2)^{(s_{1}-s_0)}\leq c_1$ for suitable constants $c_0$ and $c_1$, then there exists a constant $c>0$  such that
\begin{equation}
\label{lower:rate:smooth}
 \inf_{\check{f}_{0}}   \sup_{\theta\in \SmoothClass \cap \RegulClass}\EE_{\theta}\Big( \norm{\check{f}_{0} - f_{0}}_{L^2}^2 \Big)%
    \geq c \left\{\frac{1}{\delta^2\epsilon^4\zeta^4 n} + \left(\frac{1}{\delta^2\epsilon^2\zeta^2 n}\right)^{2s_{0}/(2s_{0}+1)}\right\}.
\end{equation}
The infima are over all estimators $\check{f}_{0}$ based on $Y_{1},\ldots,Y_{n}$. The same lower bounds hold for the estimation of $f_1$ by exchanging the role of $s_0$ and $s_1$ in the conditions and in the bounds. 
\end{theorem}

Note that if $\bm{X}$ was observed, then we would on average see $n\pi_0\gtrsim n\delta$ i.i.d. samples from $f_0$, hence we would be able to estimate $f_0$ with maximum risk $\lesssim (n\delta)^{-2s_0/(2s_0 + 1)}$ which is faster than the rates derived in Theorem~\ref{thm:2} by at least a factor of $\delta^{-2s_0/(2s_0+1)}$. This shows that the inverse problem is fundamentally harder than standard density estimation.

This theorem calls for a  number of comments. The first part of the theorem states that for the estimation of the emission densities, the minimax risk is lower bounded by a parametric term, and a nonparametric term with the usual rate $n^{-2s_{0}/(2s_{0}+1)}$ corrected with $\delta^2$, that is with an effective sample size $\delta^2 n$ replacing $n$. The second part of the theorem is more involved. It states that, if one of the emission density is smooth enough compared to the other one and relative to ``frontier'' parameters, the lower bound can be made larger, reducing the effective sample size to $\delta^2 \epsilon^2\zeta^2 n$. 
 If $s_0 > s_1$, this will eventually occur under the asymptotic regime where $\delta,\epsilon,\zeta$ do not decay too quickly to zero. Thus, the smoother emission density has a smaller effective sample size when getting close to the frontier (though still has a faster estimation rate overall).

\section{Upper bounds}
\label{sec:upper-bounds}

In this section we construct estimators whose maximum risk over $\SmoothClass \cap \RegulClass$ match those established in the lower bounds of Theorems~\ref{theo:lower:param} and~\ref{thm:2} in most cases.

\subsection{The estimation procedure}
\label{sec:estimation-procedure}

Here we describe the heuristic we use to build a near minimax optimal estimator of $\theta = (p,q,f_0,f_1)$. As noted previously \citep{MR3439359}, understanding the law of three consecutive observations is key to recovering the model parameters. A reparametrisation simplifies the expression for said law, and allows the dependence on the parameters $\delta$, $\epsilon$ and $\zeta$ to appear more naturally. Set
\begin{gather}
	\label{eqn:def:phi-psi}
	\phi(\theta)%
	=%
	\brackets[\big]{\begin{matrix}
			\frac{q-p}{p+q}\;,%
			&1 - p - q\;,%
			&\norm{ f_{0} - f_{1}}_{L^2}
			\\
	\end{matrix}},
	\quad
	\psi(\theta)%
	=%
	\brackets[\big]{\begin{matrix}
			\frac{q f_{0} + p f_{1}}{p + q}\;,%
			&\frac{f_0-f_1}{\norm{ f_0-f_1 }_{L^2}}\\
	\end{matrix}}
	. 
\end{gather}



For $m\geq 1$, let $P_{\phi,\psi}^{(m)}$ denote the law of $(Y_1,\dots,Y_m)$ under parameter 
$(\phi,\psi)$, and let $p_{\phi,\psi}^{(m)}$ denote the corresponding density with respect to Lebesgue measure on $[0,1]^m$. In the parametrisation \eqref{eqn:def:phi-psi}, defining for $\phi=(\phi_1 , \phi_2, \phi_3)$
\begin{equation}
  \label{eqn:def:r}
  r(\phi)=\tfrac{1}{4} (1-\phi_1^2)\phi_2\phi_3^2,
\end{equation}
one computes, with $f\otimes g$ defined by $(f\otimes g)(x,y) = f(x)g(y)$,
\begin{multline}
  \label{eq:17}
  p^{(3)}_{\phi,\psi} =\psi_1\otimes \psi_1 \otimes \psi_1 + r(\phi)\brackets[\big]{\psi_2\otimes \psi_2 \otimes \psi_1 +\psi_1 \otimes \psi_2 \otimes \psi_2} \\
  +  \phi_2 r(\phi) \psi_2\otimes \psi_1 \otimes \psi_2 - \phi_1\phi_2\phi_3 r(\phi) \psi_2\otimes \psi_2\otimes \psi_2.
\end{multline}

The parametrisation $\theta\mapsto (\phi,\psi)$ is invertible and has a simple inversion formula, 
\begin{eqnarray}
	&&	p= \tfrac{1}{2} (1-\phi_2)(1-\phi_1), \;
		q = \tfrac{1}{2}(1-\phi_2)(1+\phi_1),\;
       \label{inv1}\\
&&		f_0 =\psi_1-\tfrac{1}{2}\phi_1\phi_3\psi_2+\tfrac{1}{2}\phi_3\psi_2,\;
		f_1 =\psi_1-\tfrac{1}{2}\phi_1\phi_3\psi_2-\tfrac{1}{2}\phi_3\psi_2.\label{inv2}
\end{eqnarray}
It is also possible to invert the map $(\phi,\psi)\mapsto p^{(3)}_{\phi,\psi}$ up to label switching issues. We now illustrate how this can be done to recover $\phi,\psi$ from $p^{(3)}_{\phi,\psi}$; we only describe $\phi_2$ since it is the simplest to invert, but the same idea is applied to recover $\phi_1$ (and consequently $p,q$) and the wavelet coefficients of $f_0$ and $f_1$ in Sections~\ref{sec:parametric-part}, ~\ref{sec:nonparametric-part:1} and~\ref{sec:nonparametric-part:2}. From formula \eqref{eq:17}, noting that $\Inner{\psi_2,1} = 0$ and $\Inner{\psi_1,1} = 1$, it is seen that for any bounded function $h$ on $[0,1]$
\begin{align*}
  r(\phi)\Inner{\psi_2,h}^2%
  &= \EE_{\theta}(h\otimes h) - \EE_{\theta}(h)^2,\\
  r(\phi)\phi_2\Inner{\psi_2,h}^2%
  &= \EE_{\theta}(h\otimes 1 \otimes h) - \EE_{\theta}(h)^2.
\end{align*}
Provided $\Inner{\psi_2,h} \ne 0$ the previous formula can be inverted to express $\phi_2$ as a function of the ``moments'' $\EE_{\theta}(h)$, $\EE_{\theta}(h\otimes h)$ and $\EE_{\theta}(h\otimes 1 \otimes h)$:
\begin{equation}
  \label{eq:inverse-phi2}
  \phi_2%
  = \frac{\EE_{\theta}(h\otimes 1 \otimes h) - \EE_{\theta}(h)^2}{\EE_{\theta}(h\otimes h) -\EE_{\theta}(h)^2}.
\end{equation}
Analogous formulas show that $(\phi,\psi) \mapsto p_{\phi,\psi}^{(3)}$ can be inverted (up to label-switching) upon computating of suitable moments of $p_{\phi,\psi}^{(3)}$, see Lemmas~\ref{lem:compute-m} and \ref{lem:invformulam} for the other parameters. Then $(p,q,f_0,f_1)$ is retrieved by using \eqref{inv1} and \eqref{inv2}. 

We propose to estimate $(p,q)$ and the wavelet coefficients of $f_0$ and $f_1$ using the method of moments. In the inversion procedure described above we replace the moments by their empirical versions computed using
\begin{equation}
  \label{eq:empirical-dist}
  \PP_n^{(s)}(H) \coloneqq \frac{1}{n-s+1} \sum_{i=1}^{n-s+1} H(Y_i,\dots, Y_{i+s-1}),\quad H:[0,1]^s \to \RR,\ s\geq 1.
\end{equation}

As suggested by equation~\eqref{eq:inverse-phi2}, the formula for computing $(\phi,\psi)$ given the moments is unstable if the function $h$ is chosen poorly, so that the estimates may be far from the true values if $\Inner{\psi_2,h}$ is too small even if empirical moments are close to their means. No fixed choice of $h$ works uniformly over the parameter space: given $h$, there exists a parameter $(\phi,\psi)$ such that $\Inner{\psi_2,h}$ is arbitrarily small, resulting in an arbitrarily large maximum risk over $\SmoothClass$. To avoid a deteriorated maximum risk, it is therefore necessary to estimate $h$ from the data. The oracle choice for $h$ would maximize $h \mapsto \frac{|\Inner{\psi_2,h}|}{\|h\|}$ and hence be given by $h = \psi_2$. Thus, a crucial step in our estimation procedure is to provide an initial (crude) estimator $\tilde{\psi}_2$ of $\psi_2$ such that $\|\tilde{\psi}_2\|_{L^2} = 1$ and such that
\begin{equation}
  \label{eq:defItilde}
  \tilde{\mathcal{I}}\coloneqq \Inner{\psi_2,\tilde{\psi}_2}
\end{equation}
is sufficiently bounded away from zero with high probability under each parameter $(\phi,\psi)\in \SmoothClass \cap \RegulClass$. For this reason we describe $\tilde{\psi}_2$ as a separating function: since $\psi_2 = (f_0 - f_1) / \norm{f_0 - f_1}_{L^2}$, finding $\tilde{\psi}_2$ is tantamount to finding an hyperplane in $L^2[0,1]$ which separates $f_0$ and $f_1$ sufficiently well. The estimator $\tilde{\psi}_2$ is built in Section~\ref{sec:estim-separ-hyperpl-1}.

Algorithm \ref{alg:steps} summarizes the complete estimation procedure. \REV A full version of the estimation algorithm with discussion of its computational complexity is deferred to Section~\ref{sec:summary-algorithm}. \ENDREV Computing our estimator involves only elementary operations, namely: (i) determining the leading eigenvector of a relatively small matrix, (ii) calculating empirical averages, and (iii) performing straightforward algebraic manipulations. This makes our estimator both practical to implement and computationally efficient. Notably, unlike certain alternative estimators -- such as the least squares estimator \citep{MR3543517} -- our approach does not require solving a nonconvex optimization problem, ensuring that the estimator can always be reliably computed. Also, our procedure exploits the appealing adaption properties of wavelet estimators, avoiding to use Lepski's method to achieve rate adaptation \citep{MR3862446}.

\begin{algorithm}[t]
  \caption{Estimation procedure}
\label{alg:steps}
  \begin{algorithmic}[1]
  \Require An observed chain $(Y_{1},\ldots,Y_{n})$.
  \Ensure Estimators $\hat{p}$, $\hat{q}$, $\hat{f}_0$ and $\hat{f}_1$.
  \State Estimation of a good separating function $\tilde{\psi}_2$ (see Section~\ref{sec:estim-separ-hyperpl-1})
  \State Estimation of  $(\phi_1,\phi_2)$ and then $(p,q)$ (see Section~\ref{sec:parametric-part})
  \State Estimation of $(f_0, f_1)$ using block thresholding with estimators of the wavelet coefficients (see Section~\ref{sec:nonparametric-part:1} for the case $s_0=s_1$, or Section~\ref{sec:nonparametric-part:2} otherwise).
\end{algorithmic}
\end{algorithm}


Before entering the details of the estimation procedure, we recall some classical results about wavelets and Besov spaces in Section~\ref{sec:prel-wavel-besov}.

\subsection{Preliminaries on wavelets and the Besov norm we use}
\label{sec:prel-wavel-besov}

Throughout the paper we use the $S$-regular boundary-corrected wavelet basis of \citep{MR1256527}, see also e.g.\ \citep[Section 4.3.5]{GN16}, denoted $\Set{\Set{\Phi_{Jk} \given k=0,\dots, 2^{J-1}}, \Set{\Psi_{jk} \given j \geq J,\, k=0,\dots,2^{j}-1} }$, with initial resolution level $J$ chosen as in the latter reference. As is common, we will refer to the $(\Phi_{Jk})$ as father wavelets and to the $(\Psi_{jk})$ as mother wavelets. Any $f \in L^2[0,1]$ has the series expansion
\begin{equation*}
  f = \sum_{k=0}^{2^{J}-1}\Inner{\Phi_{Jk},f}\Phi_{Jk} + \sum_{j = J}^{\infty}\sum_{k=0}^{2^j-1}\Inner{\Psi_{jk},f}\Psi_{jk}
\end{equation*}
with convergence of the series in $L^2[0,1]$. In fact, as our densities will be assumed regular enough, wavelet series expansions for $f_0$ and $f_1$ will also converge uniformly \citep[e.g.][eq.~(4.71)]{GN16}). Furthermore, it is well-known that the Besov space $B_{2,\infty}^s$ can be characterised via the wavelet coefficients. Indeed the norm for $B_{2,\infty}^s$ that we will use (see e.g.\ \citep[Equation~(4.166)]{GN16}) is given by 
\begin{equation}
  \label{eq:def:besov}
  \norm{f}_{B_{2,\infty}^s}^2 \coloneqq%
  \sum_{k=0}^{2^J-1}\Inner{\Phi_{Jk},f}^2%
  +\sup_{j \geq J} 2^{2js}\sum_{k=0}^{2^j-1}\Inner{\Psi_{jk},f}^2.
\end{equation}

\subsection{Estimation of a separating hyperplane}
\label{sec:estim-separ-hyperpl-1}

\REV
As explained in Section~\ref{sec:estimation-procedure}, our estimation procedure is based on computing empirical averages of the type $\PP_n^{(2)}(\tilde{\psi}_2\otimes f) = \frac{1}{n-1}\sum_{i=1}^{n-1}\tilde{\psi}_2(Y_i)f(Y_{i+1})$ where $\tilde{\psi}_2$ is a crude estimator of $\psi_2$. If $\tilde{\psi}_2$ is also estimated from $(Y_1,\dots,Y_n)$, it is not clear at all that these empirical average approach $\EE_{\theta}(\psi_2(Y_1)f(Y_2))$, as they are sum of somewhat complex dependent random variables, each term of which depends on the whole sample $(Y_1,\dots,Y_n)$. A classical trick is to estimate $\tilde{\psi}_2$ using a sample $(\tilde{Y}_1,\dots,\tilde{Y}_n)$ that is independent from the sample $(Y_1,\dots,Y_n)$ used to compute the average $\frac{1}{n-1}\sum_{i=1}^{n-1}\tilde{\psi}_2(Y_i)f(Y_{i+1})$. In the context of HMM, however, the sample cannot be split into two independent parts. Fortunately this is not too worrisome. As explained in Section~\ref{sec:disc-about-assumpt}, it is possible to split the sample $(Y_1,\dots,Y_n)$ into three parts, and then use the first third to estimate $\tilde{\psi}_2$ and the last third for empirical averages. This way, deviation inequalities for  the empirical averages of functions involving $\tilde{\psi}_2$ can be achieved as if $\tilde{\psi}_2$ were independent of the observations used in the empirical averages, up to a term $Ce^{-c\gamma^*n}$ for $C$ and $c$ universal constants. Thus, to facilitate reading, we will throughout assume that $\tilde{\psi}_2$ is estimated using a sample $(\tilde{Y}_1,\dots,\tilde{Y}_n) \sim P_{\phi,\psi}^{(n)}$ independent of $(Y_1,\dots,Y_n)$.

\ENDREV

For notational convenience, we define the set of wavelet indices $$\Lambda(M) \coloneqq \Set{0,1,\dots,2^{J-1}} \cup \Set{(j,k) \given j=J,\dots,M,\ k=0,\dots,2^j-1}$$ including all father indices and mother indices of levels $J \leq j  \leq M$, and for all $\lambda \in \Lambda(M)$ we set $e_{\lambda} = \Phi_{Jk}$ if $\lambda = k$ and $e_{\lambda} = \Psi_{jk}$ if $\lambda=(j,k)$.

For $M$ large enough (see Theorem~\ref{thm:psitilde2} below) compute the $2^M \times 2^M$ matrix $\tilde{\mathcal{G}}$ with entries
\begin{equation*}
  \tilde{\mathcal{G}}_{\lambda,\lambda'}%
  = \frac{1}{2}\tilde{\PP}_n^{(2)}(e_{\lambda} \otimes e_{\lambda'} + e_{\lambda'} \otimes e_{\lambda})%
  - \tilde{\PP}_n^{(1)}(e_{\lambda})\tilde{\PP}_n^{(1)}(e_{\lambda'}).
\end{equation*}
The matrix $\tilde{\mathcal{G}}$ is an estimator of the matrix $\mathcal{G}$ with entries
\begin{align*}
  \mathcal{G}_{\lambda,\lambda'}%
  &= \frac{1}{2}\EE_{\theta}(e_{\lambda} \otimes e_{\lambda'} + e_{\lambda'} \otimes e_{\lambda})%
    - \EE_{\theta}(e_{\lambda})\EE_{\theta}(e_{\lambda})
  = r(\phi) \Inner{\psi_2,e_{\lambda}}\Inner{\psi_2,e_{\lambda'}}
\end{align*}
where the second equality follows from equation~\eqref{eq:17}. Hence, $\mathcal{G}$ is proportional to the Gram matrix of the vector $V_{\theta} \propto (\Inner{\psi_2,e_{\lambda}}\,:\, \lambda \in \Lambda(M))$. The matrices $\tilde{\mathcal{G}}$ and $\mathcal{G}$ are real symmetric, and thus by the spectral theorem are always diagonalizable. By concentration arguments, we expect that $\tilde{\mathcal{G}}$ will have an eigenvalue approximately equal to $r(\phi)$ (which can be positive or negative) and the rest of eigenvalues will be smaller in absolute value. The eigenvector $\tilde{V}$ (chosen such that $\norm{\tilde{V}} = 1$) corresponding to the leading eigenvalue is then an estimator of $\pm V_{\theta} / \norm{V_{\theta}}$. We suggest to set
\begin{equation*}
  \tilde{\psi}_2(x)%
  \coloneqq%
  \frac{\max\big(- \tau,\, \min\big(\tau, \sum_{\lambda\in \Lambda(M)}\tilde{V}_{\lambda} e_{\lambda}(x) \big) \big) }%
  {\Big(\int_0^1\max\big(- \tau,\, \min\big(\tau, \sum_{\lambda\in \Lambda(M)}\tilde{V}_{\lambda} e_{\lambda}(y) \big) \big)^2dy\Big)^{1/2} }
\end{equation*}
where the truncation $\tau\geq 1$ is intended to prevent technicalities within the proofs. The next theorem shows that $\tilde{\psi}_2$ is well aligned with $\psi_2$ with high probability under $\theta \in \SmoothClass \cap \RegulClass$. The proof of Theorem \ref{thm:psitilde2} can be found in Section \ref{sec:proof-theorem-psitilde2}.
 
\begin{theorem}
  \label{thm:psitilde2}
  Suppose for some $L\geq 1$, $\zeta >0$, $R > 0$, $s_{*} > 0$, $M \geq J$ we have
  \begin{equation*}
    \tau \geq \frac{L}{\zeta},\qquad%
    2^{-Ms_{*}} \leq  \frac{\zeta\sqrt{2^{2s_{*}}-1} }{4R}.
  \end{equation*}
  There exists a constant $C > 0$ such that for all $S\geq s_0,s_1 \geq s_{*}$ and all $\sg > 0$
  \begin{equation*}
    \sup_{\theta \in \SmoothClass \cap \RegulClass}\tilde{\PP}_{\theta}\Big( |\Inner{\tilde{\psi}_2,\psi_2}| \leq \frac{7}{8} \Big)%
    \leq 2 \cdot 24^{2^M}\exp\brackets[\Bigg]{- \frac{Cn\sg \delta^2\epsilon^2\zeta^4}{L^3 + 2^M\sqrt{L}\delta\epsilon\zeta^2 } }.
  \end{equation*}%
  
\end{theorem}

\subsection{Parametric part}
\label{sec:parametric-part}

Define $m(\phi)=(m(\phi)_1 ,m(\phi)_2 , m(\phi)_3 )$ by 
\begin{align*} m(\phi)_1 &\coloneqq \EE_{\theta}[\tilde{\psi}_2(Y_1)\tilde{\psi}_2(Y_2) \mid \tilde{\psi}_2]- \EE_{\theta}[\tilde{\psi}_2(Y_1)^2 \mid \tilde{\psi}_2], \\ m(\phi)_2 &\coloneqq \EE_{\theta}[\tilde{\psi}_2(Y_1)  \tilde{\psi}_2(Y_3) \mid \tilde{\psi}_2] - \EE_{\theta}[\tilde{\psi}_2(Y_1)^2\mid \tilde{\psi}_2], \\ m(\phi)_3 &\coloneqq - \EE_{\theta}[\tilde{\psi}_2(Y_1) \tilde{\psi}_2(Y_2) \tilde{\psi}_2(Y_3 )\mid \tilde{\psi}_2]%
    + \EE_{\theta}[\tilde{\psi}_2(Y_1)^3\mid \tilde{\psi}_2]%
    + \big(2m(\phi)_1 + m(\phi)_2 \big)\EE_{\theta}[\tilde{\psi}_2(Y_1)\mid \tilde{\psi}_2].\end{align*} 
    This can be estimated by the following empirical quantities:
\begin{align*}
  \hat{m}_1%
  &\coloneqq \PP_n^{(2)}(\tilde{\psi}_2 \otimes \tilde{\psi}_2) - \PP_n^{(1)}(\tilde{\psi}_2)^2,\\
  \hat{m}_2%
  &\coloneqq \PP_n^{(3)}(\tilde{\psi}_2\otimes 1 \otimes \tilde{\psi}_2) - \PP_n^{(1)}(\tilde{\psi}_2)^2,\\
  \hat{m}_3%
  &\coloneqq - \PP_n^{(3)}(\tilde{\psi}_2\otimes \tilde{\psi}_2 \otimes \tilde{\psi}_2)%
    + \PP_n^{(1)}(\tilde{\psi}_2)^3%
    + \big(2\hat{m}_1 + \hat{m}_2 \big)\PP_n^{(1)}(\tilde{\psi}_2).
\end{align*}
Easy computations lead to (recall $\tilde{\mathcal{I}}\coloneqq \Inner{\psi_2,\tilde{\psi}_2}$ in \eqref{eq:defItilde} and $r(\phi)=(1/4)(1-\phi_1^2)\phi_2\phi_3^2$ in \eqref{eqn:def:r}) 
\begin{equation}
  \label{def:m}
  m(\phi) \equiv%
  \big(r(\phi) \tilde{\mathcal{I}}^2,\, r(\phi)\phi_2\tilde{\mathcal{I}}^2,\, r(\phi)\phi_1\phi_2\phi_3\tilde{\mathcal{I}}^3 \big),
\end{equation}
see Lemma~\ref{lem:compute-m} in Appendix~\ref{sec:proofs}. The moments in the previous display can be inverted modulo label-switching. 
Namely, it is possible to express $\phi_1\sgn(\tilde{\mathcal{I}})$, $\phi_2$, and $\phi_3|\tilde{\mathcal{I}}|$ as functions of $m(\phi)$. 
The inversion formulas for $m$ are given in Lemma~\ref{lem:invformulam}.
By replacing $m(\phi)$ with the empirical estimates in the inversion formula we define
\begin{equation*}
  \hat{\phi}_1%
  \coloneqq%
  \frac{\hat{m}_3}{[4\hat{m}_1^2(\hat{m}_2)_+ + \hat{m}_3^2]^{1/2}},\qquad%
  \hat{\phi}_2%
  \coloneqq%
  \max\Big(-1,\, \min\Big(\frac{\hat{m}_2}{\hat{m}_1},\, 1\Big)\Big).
\end{equation*}
Notice that since  $m(\phi)_2\geq 0$, we replaced $\hat{m}_2$ by $(\hat{m}_2)_+= \max(\hat{m}_2,0)$.
We then build an estimator of $p$ and $q$ justified by \eqref{inv1} by letting 
\begin{align*}
  \hat{p}=\tfrac{1}{2}(1 - \hat{\phi}_1)(1-\hat{\phi}_2),\\
  \hat{q}= \tfrac{1}{2}(1 + \hat{\phi}_1)(1-\hat{\phi}_2).
\end{align*}
To account for label switching, write $Q^\sigma$ for the matrix with entries $(Q^\sigma)_{ij}=Q_{\sigma(i),\sigma(j)}$ for a permutation $\sigma$. We consider the loss relative to the Frobenius norm $\normF{\cdot} \coloneqq \sum_{i,j}(\cdot)_{i,j}^2$. The proof of Theorem \ref{thm:concent-pq} can be found in Section \ref{sec:proof-prop-refpr}

\begin{theorem}
  \label{thm:concent-pq}
  Assume that $\zeta\leq 1$, that $\tau$ and $M$ are chosen as prescribed in Theorem~\ref{thm:psitilde2}, and that $n\gamma^*\geq \tau^6/L^3$. Then there are universal constants $B,C>0$ such that
  \begin{multline*}
    \sup_{\theta\in \SmoothClass \cap \RegulClass}\inf_\sigma \EE_{\theta}\brackets[\Big]{\normF{{Q}_{\hat{\theta}}^\sigma-Q_{\theta}}^2 }\leq 2 \cdot 24^{2^M}\exp\brackets[\Bigg]{- \frac{Cn\sg \delta^2\epsilon^2\zeta^4}{L^3 + 2^M\sqrt{L}\delta\epsilon\zeta^2 } }\\
	+ B\exp\brackets[\Bigg]{- \frac{C n \sg \delta^2\epsilon^4\zeta^6 }{L^3 + \max(\tau,\sqrt{L})^3 \delta \epsilon^2\zeta^3} }  
	+\frac{BL^3\max(\delta^2,\epsilon^2\zeta^2)}{\delta^2\epsilon^4\zeta^6}\frac{1}{n\sg}.
  \end{multline*}
\end{theorem}

In an asymptotic regime, the first terms in the bound in Theorem~\ref{thm:concent-pq} can be neglected and our estimator achieves the rate of convergence $\frac{L^3\max(\delta^2,\epsilon^2\zeta^2)}{\delta^2\epsilon^4\zeta^6}\frac{1}{n\sg}$, which is, up to constants, the minimax rate established in Theorem~\ref{theo:lower:param}.
We note that the parametric part ${Q}_{\hat{\theta}}$ achieves the same rate in the nonparametric setting as in the multinomial setting \citep{AGNparamhmm}; at first glance 
this seems  unsurprising in view of the fact that the pairs $((X_n,h(Y_n))_{n\geq 0}$ form a hidden Markov model with transition matrix $Q_{\theta}$ for any function $h$, so that for a suitable $h$ we can reduce to a parametric setting.
However, reducing to a parametric setting in which $Q_{\theta}$ is still identifiable is in fact a nonparametric problem (as alluded to in Section~\ref{sec:related}, or see Section~\ref{sec:estimation-procedure} for more details), so that getting the same minimax parametric rate is not a priori guaranteed. Indeed, to construct an estimator for the parametric part $(p,q)$, we must first solve the nonparametric problem of estimating $\psi_2$. This step does not harm the risk of our estimator and we are able to match the semiparametric rate given in Theorem~\ref{theo:lower:param}. This is because the estimator $\tilde{\psi}_2$ does not need to be a \textit{good} estimator of $\psi_2$ (it is not required even to be consistent), but must only guarantee that $\tilde{\mathcal{I}} = \Inner{\psi_2,\tilde{\psi}_2}$ does not get too small.

\subsection{Nonparametric part: case $s_0 =s_1$}
\label{sec:nonparametric-part:1}

Using the ideas like in Section~\ref{sec:estimation-procedure}, the wavelet coefficients of $f_0$ and $f_1$ can be extracted from $\Set{\EE_{\theta}(\tilde{\psi}_2 \otimes \Phi_{Jk})}$, $\Set{\EE_{\theta}(\tilde{\psi}_2 \otimes \Psi_{jk})}$, $\Set{\EE_{\theta}(\Phi_{Jk})}$, $\Set{\EE_{\theta}(\Psi_{jk})}$ and $\EE_{\theta}(\tilde{\psi}_2)$, and further estimated using their empirical relatives. Given these empirical wavelets coefficients, we construct estimators for $f_0$ and $f_1$ based on block-thresholding the coefficients.

For notational convenience, we write $\Father \coloneqq \Inner{\Phi_{Jk},f}$ and $\Mother \coloneqq \Inner{\Psi_{jk},f}$.  
First, using the inversion formulas for $m$ given in Lemma~\ref{lem:invformulam} and by
replacing $m(\phi)$ with the empirical estimates in the inversion formula we define an estimator of $g\coloneqq \phi_3\abs{\tilde{\mathcal{I}}}$ by
\begin{equation*}
  \hat{g} \coloneqq \frac{\sqrt{4\hat{m}_1^2(\hat{m}_2)_+ + \hat{m}_3^2}}{\hat{m}_2}\1_{\{ \hat{m}_2 > 0\} }.
\end{equation*}
Now, our goal is to find estimators $\Set{\brackets{\HatFather_0}_{k},\brackets{\HatMother_0}_{jk}}$ of $\Set{\brackets{\Father_0}_{k},\brackets{\Mother_0}_{jk}}$ (and similarly for $f_1$). 
We use \eqref{inv2} and we set
\begin{align*}
\hat{G}^{\Phi_{Jk}} &\coloneqq%
	\PP_n^{(2)}(\tilde{\psi}_2 \otimes \Phi_{Jk} ) - \PP_n^{(1)}(\tilde{\psi}_2)\PP_n^{(1)}(\Phi_{Jk}), \\
  \HatFather_{0}%
  &\coloneqq \PP_n^{(1)}(\Phi_{Jk}) + \frac{\hat{g}(1 - \hat{\phi}_1)}{2\hat{m}_1}\1_{ \{ \hat{m}_1 \ne 0 \} }\hat{G}^{\Phi_{Jk}}, \\
  \HatFather_{1}%
  &\coloneqq \PP_n^{(1)}(\Phi_{Jk}) - \frac{\hat{g}(1 + \hat{\phi}_1)}{2\hat{m}_1}\1_{ \{ \hat{m}_1 \ne 0 \} } \hat{G}^{\Phi_{Jk}}.%
\end{align*}
The same definition applies \textit{mutatis mutandis} to the estimators of the mother coefficients $\HatMother_{0}$, $\HatMother_{1}$,  and $\hat{G}^{\Psi_{jk}}$. It is customary that not all empirical coefficients be retained in the final estimator, and that small coefficients should be discarded to reduce the risk. It is also well-known \citep{cai2008information} that individual coefficient
thresholding is sub-optimal with respect to the $L^2$ loss, as opposed to block-thresholding procedures with carefully chosen blocks \citep{cai1999adaptive,chicken2005block}. Here, we build the blocks as follows.

Motivated by \citep{cai1999adaptive,chicken2005block} we wish to build blocks of consecutive wavelets with size approximately $\log(n)$, which is known to be the best compromise for global versus local adaptation. Since there may be fewer than $\log(n)$ wavelets at small resolution levels $j$, we will only threshold coefficients with $j$ large enough. We define
\begin{equation*}
  J_n \coloneqq \inf\Set*{j \geq J \given 2^j \geq  \log(n) }
\end{equation*}
where the infimum is over the integers. We then let $N \coloneqq 2^{J_n}$ so that each level with $j \geq J_n$ can be partitioned into an integer number of blocks of $N$ consecutive wavelets. More precisely,  for each level $j\geq J_n$, and each $\ell = 0,\dots, N^{-1}2^j - 1$ we define the blocks of indices
\begin{equation}
  \label{eq:BJL}
  \mathfrak{B}_{j\ell} \coloneqq \Set{k \in \Set{0,\dots,2^{j-1}} \given (\ell-1)N \leq k \leq \ell N - 1}.
\end{equation}
For a constant $\tau\geq 1$ we also define $\tilde{\jmath}_n$ as the largest integer such that $2^{\tilde{\jmath}_n} \leq \frac{n}{\log(n)\tau^2}$; we shall assume that  $J < J_n < \tilde{\jmath}_n$ which is always satisfied for $n$ large enough. We then let, for $i=0,1$,
\begin{align*}
  \hat{f}_{i}%
  &\coloneqq \sum_{k=0}^{2^J-1} \HatFather_{i} \Phi_{Jk}%
    + \sum_{j=J}^{J_n-1}\sum_{k=0}^{2^j-1} \HatMother_i \Psi_{jk}
    + \sum_{j=J_n}^{\tilde{\jmath}_n}\sum_{ \ell }\Bigg(\sum_{k\in \mathfrak{B}_{j\ell}}\HatMother_{i} \Psi_{jk} \Bigg)\1_{ \{ \norm{\hat{f}_{i}^{\mathfrak{B}_{j\ell}}} > \Gamma \hat{S}_n \} }
\end{align*}
where $\norm{\hat{f}_{i}^{\mathfrak{B}_{j\ell}}}^2 \coloneqq \sum_{k\in \mathfrak{B}_{j\ell}}( \HatMother_i )^2$, $\Gamma >0$ is a tuning parameter, and
\begin{equation*}
  \hat{S}_n
  \coloneqq \sqrt{\frac{\log(n)}{n}}\max\Big(1,\, \frac{\hat{g}}{|\hat{m}_1|}\Big)\1_{ \{\hat{m}_1 \ne 0 \} }.
\end{equation*}
The above estimators perform well in probability; to ensure good perfomance in expectation we truncate below at 0 and above at some $\check{T}$, defining for $i=0,1$
\begin{equation*}
  \check{f}_{i} \coloneqq \max\big(0,\, \min\big(\check{T},\, \hat{f}_{i}\big)\big).
\end{equation*}

\begin{theorem}
  \label{thm:1}
  Assume $\tau$ and $M$ are chosen as prescribed in Theorem~\ref{thm:psitilde2}. Suppose $n\sg \geq \max(\tau^3,\frac{\tau^2 \log(n)^2}{L})$,  $\tilde{\jmath}_n > J_n$, $L \leq n$, $\check{T} \geq L$, and $\zeta \leq 1$. Then there are universal constants $\beta > 0$, $B > 0$ and $C > 0$ such that for all $\Gamma \geq \beta L^{1/2}\max((L/\sg)^{1/2},1/\sg)$ and for $i=0,1$, provided $s_{*} \leq s_i \leq S$ with $S > 0$ the regularity of the wavelet basis,
  \begin{multline*}
    \sup_{\theta\in \SmoothClass \cap \RegulClass}\EE_{\theta} \min_{i'=0,1} \brackets[\Big]{ \norm{ \check{f}_{i'} - f_{i}}_{L^2}^2 }%
    \leq B\check{T}^2 24^{2^M}\exp\brackets[\Bigg]{- \frac{Cn\sg \delta^2\epsilon^2\zeta^4}{L^3 + 2^M\sqrt{L}\delta\epsilon\zeta^2 } }\\
    + B \check{T}^2 \exp\brackets[\Bigg]{- \frac{C n \sg \delta^2\epsilon^4\zeta^6 }{L^3 + \max(\tau,\sqrt{L})^3 \delta \epsilon^2\zeta^3} }
      + \frac{B L^2}{\delta^2\epsilon^2\zeta^2}\frac{\log(n)}{n\sg}%
      + \frac{BL^3}{\delta^2\epsilon^4\zeta^4}\frac{1}{n\sg}\\%
      + \frac{B\max(\tau,\sqrt{L})^6}{\delta^2\epsilon^4\zeta^4}\frac{1}{(n\sg)^2}
      + \frac{BR^2\max(1,\frac{L^2}{\Gamma^2 \sg})}{\min(1,s_{i})}\Big(\frac{\Gamma^2}{R^2\delta^2\epsilon^2\zeta^2 n} \Big)^{2s_{i}/(2s_{i} + 1)}\\
      + \frac{BR^2\max(1,\frac{L^2}{\Gamma^2 \sg})}{\min(1,s_{i})}\Big( \frac{\tau^2 \log(n)}{n} \Big)^{2s_{i}}.
  \end{multline*}
\end{theorem}
The proof of Theorem \ref{thm:1} is in Section \ref{sec:proof-thm:1}. Of particular interest is the boundary regime, where $\sg$, $R$, $L$, $\check{T}$ and $\tau$ are of constant order while $\delta$, $\gamma$ and $\zeta$ are small. The following corollary is intended to illustrate how the bound simplifies in such settings, provided $\delta$, $\gamma$ and $\zeta$ are not too small. The proof of Corollary~\ref{cor:smooth-ub} is given in Section~\ref{sec:proof-cor:smooth-ub}.

\begin{corollary}
  \label{cor:smooth-ub}
  Assume that $\sg$, $R$, $L$, $\check{T}$, and $\tau$ remain constant as $n\to \infty$ and $\delta \geq n^{-a}$, $\epsilon \geq n^{-b}$, $1 \geq \zeta \geq  n^{-c}$ for constants $a,b,c>0$ such that $1-2a-4b-6c>0$ and such that $2^M = o(n^{(1-a-b-2c)/2})$ (the penultimate requirement corresponds to where the bounds on the right vanish, so that parameters are proved to be learnable). Then the bound in the Theorem~\ref{thm:1} simplifies: for large enough $n$,
	\[  \sup_{\theta\in \SmoothClass \cap \RegulClass}\EE_{\theta} \min_{i'=0,1} \brackets[\Big]{ \norm{ \check{f}_{i'} - f_{i}}_{L^2}^2 }\leq C \left\{
\frac{1}{\delta^2 \epsilon^4\zeta^4n} + \brackets[\Big]{\frac{1}{\delta^2\epsilon^2\zeta^2n}}^{2s_i/(1+2s_i)}\right\},\] 
for a constant $C$ depending on $\gamma^*$, $L$, $R$, $\Gamma$, $B$, $\tau$, $\check{T}$, and $a,b,c$.
\end{corollary}

\subsection{Nonparametric part: case $s_0 <s_1$}
\label{sec:nonparametric-part:2}

In the particular situation where $s_0 = s_1$, the lower bound \eqref{lower:rate:smooth} holds for the estimation of both emission densities, and the estimators $\check{f}_0$ and $\check{f}_1$ are rate minimax adaptive, including to the parameters of interest $\delta, \epsilon, \zeta$. However, in the situation where $s_0 \ne s_1$, assuming without loss of generality assuming that $s_0 < s_1$, the estimator for the rougher density $f_0$ is not rate optimal in term of $\delta,\epsilon,\zeta$. We fill the gap by constructing another estimator for $f_0$ that attains the optimal rate. The construction of the estimator exploits the ``borrowing strength'' phenomenon described in the introduction, which we now make more formal. We focus only on estimating $f_0$ when $s_0 < s_1$; the estimation of $f_1$ when $s_0 > s_1$ is similar.

The starting point is to remark that
\begin{equation}
\label{eq:rough:rationale}
  f_{0}%
  = \frac{2 \psi_1}{1 +\phi_1}%
  - \Bigg( \frac{1 - \phi_1}{1 +\phi_1}\psi_1 - \frac{g(1 -\phi_1)}{2m_1}G \Bigg)
\end{equation}
with $G = r(\phi)\tilde{\mathcal{I}}\psi_2$, whose wavelet coefficients can be estimated using $\Set{\hat{G}^{\Phi_{Jk}},\hat{G}^{\Psi_{jk}} }$. Note that $2\psi_1/(1+\phi_1) = \pi_0^{-1}\psi_1$ and the other term involved in \eqref{eq:rough:rationale} equals $(1-\pi_0)\pi_0^{-1}f_1$. We recall the rationale of the borrowing strength phenomenon: $\psi_1$ is ``easy'' to estimate (estimating it is a direct problem, not an inverse problem) since it is the stationary distribution of $Y_n$; also $f_1$, being smoother than $f_0$, can be estimated at a better rate. We estimate the father wavelet coefficients of $f_0$ using the same estimators as before. Regarding the mother coefficients, however, we let $\alpha_0 \coloneqq \pi_0^{-1}\psi_1$ and $\beta_0 \coloneqq f_0 - \alpha_0$ and we estimate separately the coefficients of these two functions using
\begin{equation*}
  \HatAMother_{0}\coloneqq \frac{2 \hat{\psi}_1^{\Psi_{jk}}}{1 + \hat{\phi}_1}\1_{\{\hat{\phi}_1 \ne -1\}},\quad%
  \HatBMother_{0}\coloneqq%
  - \Bigg(\frac{1- \hat\phi_1}{1 + \hat\phi_1}\1_{\{\hat{\phi}_1 \ne -1\}}\hat{\psi}_1^{\Psi_{jk}}%
 - \frac{\hat{g}(1- \hat\phi_1)}{2\hat{m}_1}\1_{\{ \hat{m}_1 \ne 0 \}} \hat{G}^{\Psi_{jk}}\Bigg).
\end{equation*}
Then, what we shall call the `rough estimator' (since it only usefully estimates the rougher of the two functions $f_0,f_1$) is defined as:
\begin{multline}
  \label{eq:def:rough}
  \hat{f}_{0}^R%
 \coloneqq 
 \sum_{k=0}^{2^{J_n}-1} \HatFather_0 \Phi_{J_nk}%
 + \sum_{j=J}^{J_n-1}\sum_{k=0}^{2^j-1}\HatMother_0 \Psi_{jk}\\%
  + \sum_{j=J_n}^{\tilde{\jmath}_n}\sum_{\ell=0}^{2^j/N-1}\Big(\sum_{k\in \Block_{j\ell}}\HatAMother_{0} \Psi_{jk} \Big)\1_{ \{ \norm{\hat{\psi}_1^{\Block_{j\ell}}} > \Gamma\sqrt{\log(n)/n} \} }\\
 + \sum_{j=J_n}^{\tilde{\jmath}_n}\sum_{\ell=0}^{2^j/N-1}\Big(\sum_{k\in \Block_{j\ell}}\HatBMother_{0} \Psi_{jk} \Big)\1_{ \{  \norm{\hat{\B}_0^{\Block_{j\ell}}} > \Gamma\hat{T}_n \} },
\end{multline}
with $\HatFather_0$ and $\HatMother_0$ as previously and
\begin{equation*}
  \hat{T}_n \coloneqq%
  \sqrt{\frac{\log(n)}{n}}\max\Big(1,\, \frac{\hat{g}}{|\hat{m}_1|}\1_{\hat{m}_1\ne 0},\, \frac{1}{1 - \hat{\phi}_1^2}\1_{\hat{\phi}_1^2 \ne 1}\Big).
\end{equation*}
Note that in \eqref{eq:def:rough}, thresholding of the estimated coefficients of $\psi_1$ is done ``as usual'' for density estimation, whereas thresholding of the $\HatBMother_{0}$'s is done with another carefully chosen threshold. 

As previously, we also further require a truncation of the estimator to obtain control in expectation not just in probability, and for some $\check{T}>0$ we define
\begin{equation*}
  \check{f}_0^R \coloneqq \max\big(0,\, \min\big(\check{T},\, \hat{f}_0^R\big)\big).
\end{equation*}%

The following theorem gives an upper bound on the maximum risk of $\hat{f}_0^R$. The proof of Theorem \ref{thm:rough} is detailed in Section \ref{sec:proof:thm:rough}. 

\begin{theorem}
  \label{thm:rough}
  Assume $\tau$ and $M$ are chosen as prescribed in Theorem~\ref{thm:psitilde2}. Suppose $n\sg \geq \max(\tau^3,\frac{\tau^2 \log(n)^2}{L})$,  $\tilde{\jmath}_n > J_n$, $L \leq n$, $\check{T} \geq L$, $\zeta \leq 1$, and $s_{*} < s_0 \leq S$, with $S > 0$ the regularity of the wavelet basis. Then there are universal constants $\beta > 0$, $B > 0$ and $C > 0$ such that for all $\Gamma \geq \beta \max(\frac{L}{\sqrt{\sg}},\frac{\sqrt{L}}{\tau\sg} )$ 
    \begin{multline*}
      \sup_{\theta\in \SmoothClass \cap \RegulClass}\EE_{\theta}\brackets[\Big]{ \norm{ \check{f}_0^R - f_0}_{L^2}^2 }%
      \leq  B\check{T}^2 24^{2^M}\exp\brackets[\Bigg]{- \frac{Cn\sg \delta^2\epsilon^2\zeta^4}{L^3 + 2^M\sqrt{L}\delta\epsilon\zeta^2 } }\\
       + B \check{T}^2\exp\brackets[\Bigg]{- \frac{C n \sg \delta^2\epsilon^4\zeta^6 }{L^3 + \max(\tau,\sqrt{L})^3 \delta \epsilon^2\zeta^3} }
         + \frac{B L^2}{\delta^2\epsilon^2\zeta^2}\frac{\log(n)}{n\sg}%
   + \frac{BL^3}{\delta^2\epsilon^4\zeta^4}\frac{1}{n\sg}\\%
         +\frac{B\max(\tau,\sqrt{L})^6}{\delta^2\epsilon^4\zeta^4}\frac{1}{(n\sg)^2}
         + \frac{R^2}{\min(1,s_0)}\Big(\frac{\Gamma^2}{nR^2\delta^2} \Big)^{2s_0/(2s_0 + 1)}\\%
       +
         \frac{R^2}{\min(1, s_1)} \frac{1}{\delta^2} \Big( \frac{\Gamma^2 }{R^2 n \epsilon^2\zeta^2} \Big)^{2s_1/(2s_1 + 1)}%
         + \frac{BR^2}{\min(1,s_0)}\Big( \frac{\tau^2 \log(n)}{n} \Big)^{2s_0}.
    \end{multline*}
\end{theorem}

As with Theorem~\ref{thm:1} and its Corollary~\ref{cor:smooth-ub}, of particular interest is the boundary regime, where $\sg$, $R$, $L$, $\check{T}$ and $\tau$ are of constant order while $\delta$, $\gamma$ and $\zeta$ are small, but not too small. The following corollary is intended to illustrate how the bound simplifies in such setting. The proof of Corollary~\ref{cor:rough-ub} is given in Section~\ref{sec:proof-cor:rough-ub}.

\begin{corollary}
  \label{cor:rough-ub}
  Assume that $\sg$, $R$, $L$, $\check{T}$, and $\tau$ remain constant as $n\to \infty$ and $\delta \geq n^{-a}$, $\epsilon \geq n^{-b}$, $1 \geq \zeta \geq  n^{-c}$ for constants $a,b,c>0$ with $a,b,c = o(1)$ and $2^M = o(n^{(1-a-b-2c)/2})$ as $n \to \infty$. Then if $s_1 < s_0$ the bound in the Theorem~\ref{thm:rough} simplifies: for large enough $n$,
\[
  \sup_{\theta\in \SmoothClass \cap \RegulClass}\EE_{\theta}\Big( \norm{\check{f}_{0}^R - f_{0}}_{L^2}^2 \Big)%
    \leq C \left\{\frac{1}{\delta^2\epsilon^4\zeta^4 n} + \left(\frac{1}{\delta^2 n}\right)^{2s_{0}/(2s_{0}+1)}\right\},
\]
for a constant $C$ depending on $\gamma^*$, $L$, $R$, $\Gamma$, $B$, $\tau$, $\check{T}$.
\end{corollary}

In the regime of Corollary \ref{cor:rough-ub}, \textit{ie}. when $\delta,\epsilon,\zeta$ are small but not too small, the estimator $\check{f}_0$ achieves the lower bound established in Theorem~\ref{thm:2}. In settings where $\delta,\epsilon,\zeta$ are allowed to be smaller than a polynomial in $n$, a transition in the rate still occurs according to how $s_0$ and $s_1$ compare, but then it may be required to have $s_1$ much larger than $s_0$ (depending on $\delta,\epsilon,\zeta$) to get matching upper and lower bounds.

\REV
We conclude this section by mentioning that the ``borrowing strength phenomenon'' is not specific to the case where the $f_j$'s belong to different Besov bodies $\Set{f \given \|f\|_{B_{2,\infty}^{s_j}} \leq R}$ with $s_0 \ne s_1$. Indeed, the same phenomenon should occur as long as the $f_j$'s belong to classes $\mathcal{S}_j$ of different ``complexities'' (which can for instance be measured by the number of balls of finite radius needed to cover $\mathcal{S}_j$); or in other words, as soon as nonparametric estimation over $\mathcal{S}_1$ is easier than over $\mathcal{S}_0$ (or conversely). Thus, the phenomenon would take place if the Besov bodies are replaced by other types of smoothness classes (for instance Hölder balls of finite radius).
\ENDREV

\REV

\subsection{Summary of the algorithm}
\label{sec:summary-algorithm}

In this section we present our algorithm in full, self-contained manner, and discuss its computational complexity. To simplify the exposition of the algorithm, let us recall or introduce some notations.

We use the $S$-regular boundary corrected wavelet basis $\Set{\Set{\Phi_{Jk} \given k=0,\dots,2^{J-1}},\Set{\Psi_{jk} \given j\geq J,\ k=0,\dots,2^{j-1}}}$ constructed in \cite{MR1256527}. We use the notation $\Lambda(m) = \Set{0,\dots,2^{J-1}} \cup \Set{(j,k) \given j=J,\dots,m,\ k=0,\dots,2^{j-1}}$ for $m\geq J$. We also write $e_{\lambda} = \Phi_{J\lambda}$ if $\lambda\in \Set{0,\dots,2^{J-1}}$ or $e_{\lambda} = \Psi_{jk}$ if $\lambda=(j,k)$. We also define for real-valued function $f$ and reals $a < b$ the clipping operation $\mathrm{clip}(f, [a,b])$ defined such that $\mathrm{clip}(f, [a,b])(x) = \max( a, \min(f(x), b)) $.

Our complete estimation procedure is given in the Algorithm~\ref{alg:complete}. The Algorithm~\ref{alg:complete} computes the estimator of $Q$ defined in Section~\ref{sec:parametric-part} and the estimators of $f_0$ and $f_1$ defined in Section~\ref{sec:nonparametric-part:1}, where they are proven to be minimax optimal in the case where $s_0 = s_1$. In the case where $s_0 \ne s_1$ and information is available to identify the smoothest emission density, the previous algorithm can be complemented by an additional step to improve the estimator of the roughest density, corresponding to the estimator derived in Section~\ref{sec:nonparametric-part:2}. We summarize this additional step in the  Algorithm~\ref{alg:rough}, assuming without loss of generality that $s_0 < s_1$.

\begin{algorithm}[t]
  \caption{Full algorithm}
  \label{alg:complete}
  
\begin{algorithmic}[1]
\Require Data $(Y_1,\dots,Y_{3n})$ and hyperparameters $M \in \Set{J,J+1,\dots}$, $\tau > 0$, $\Gamma > 0$, $\check{T}>0$.
\Ensure Estimators $\hat{Q}$, $\check{f}_0$, and $\check{f}_1$.

\Algphase{1}{Sample splitting}
\State Let $(\tilde{Y}_1,\dots,\tilde{Y}_n) = (Y_{2n+1},\dots,Y_{3n})$.

\Algphase{2}{Estimation of the separating hyperplane}
\State Compute the $2^M\times 2^M$ matrix $\tilde{\mathcal{G}}$ with entries $\tilde{\mathcal{G}}_{\lambda,\lambda'} = \frac{1}{2(n-1)}\sum_{i=1}^{n-1}(e_{\lambda}(\tilde{Y}_i)e_{\lambda'}(\tilde{Y}_{i+1}) + e_{\lambda'}(\tilde{Y}_i)e_{\lambda}(\tilde{Y}_{i+1}) ) - \frac{1}{n}\sum_{i=1}^ne_{\lambda}(\tilde{Y}_i) \cdot\frac{1}{n}\sum_{i=1}^ne_{\lambda'}(\tilde{Y}_i)$ for every $\lambda,\lambda' \in \Lambda(M)$.
\State Compute leading eigenvector $v$ of $\tilde{\mathcal{G}}$.
\State Let $\tilde{\psi}_2 \propto \mathrm{clip}(\sum_{\lambda\in \Lambda(M)} v_\lambda e_\lambda,[-\tau,\tau] )$ with $\|\tilde{\psi}_2\| = 1$.

\Algphase{3}{Estimation of the transition matrix $Q$}
\State Compute
\begin{align*}
    \hat{m}_1 &= \textstyle\frac{1}{n-1}\sum_{i=1}^{n-1}\tilde{\psi}_2(Y_i)\tilde{\psi}_2(Y_{i+1})
    - \left(\frac{1}{n}\sum_{i=1}^n\tilde{\psi}_2(Y_i) \right)^2,\\
    \hat{m}_2 &=  \textstyle\frac{1}{n-2}\sum_{i=1}^{n-1}\tilde{\psi}_2(Y_i)\tilde{\psi}_2(Y_{i+2}) - \left(\frac{1}{n}\sum_{i=1}^n\tilde{\psi}_2(Y_i) \right)^2
    ,\\
    \hat{m}_3 &=  \textstyle\frac{1}{n-2}\sum_{i=1}^{n-2}\tilde{\psi}_2(Y_i)\tilde{\psi}_2(Y_{i+1})\tilde{\psi}_2(Y_{i+2}) + \left(\frac{1}{n}\sum_{i=1}^n\tilde{\psi}_2(Y_i) \right)^3 + (2\hat{m}_1 + \hat{m}_2)\frac{1}{n}\sum_{i=1}^n\tilde{\psi}_2(Y_i).
\end{align*}
\State Compute $\hat{\phi}_1 = \frac{\hat{m}_3}{[4\hat{m}_1^2(\hat{m}_2)_+ + \hat{m}_3^2]^{1/2}}$ and $\hat{\phi}_2 = \max\left(-1,\min\left(\frac{\hat{m}_2}{\hat{m}_1},1\right)\right)$.
\State Let $\hat{p} = \frac{1}{2}(1-\hat{\phi}_1)(1-\hat{\phi}_2)$, $\hat{q} = \frac{1}{2}(1+\hat{\phi}_1)(1-\hat{\phi}_2)$, and $\hat{Q} =
\left(\begin{smallmatrix}
  1- \hat{p} & \hat{p}\\\hat{q} &1 - \hat{q}\\
\end{smallmatrix}\right)$.

\Algphase{4}{Estimation of the emission densities}
\State Compute $\hat{g} = \frac{\sqrt{4\hat{m}_1^2(\hat{m}_2)_+ + \hat{m}_3^2}}{\hat{m}_2}\1_{\{\hat{m}_2 > 0\}}$.
\State Let $\tilde{\jmath}_n = \left\lfloor \log_2\left(\frac{n}{\log(n)\tau^2} \right) \right\rfloor$, $J_n = \lceil \log_2(n) \rceil$, $\hat{S}_n = \sqrt{\frac{\log(n)}{n}}\max\left(1,\, \frac{\hat{g}}{|\hat{m}_1|}\right)\1_{\{\hat{m}_1\ne 0\}}$.
\State For all $\lambda \in \Lambda(\tilde{\jmath}_n)$, compute the empirical wavelet coefficients $\hat{\psi}_1^{\lambda} = \frac{1}{n}\sum_{i=1}^n e_{\lambda}(Y_i)$ and $\hat{G}^{\lambda} = \frac{1}{n-1}\sum_{i=1}^{n-1}\tilde{\psi}_2(Y_i)e_{\lambda}(Y_{i+1}) - \frac{1}{n}\sum_{i=1}^n\tilde{\psi}_2(Y_i)\cdot\frac{1}{n}\sum_{i=1}^ne_{\lambda}(Y_i)$.
\For{m=0,1}
\State Compute $\hat{f}_m^{\lambda} = \hat{\psi}_1^{\lambda} + (-1)^m \frac{\hat{g}(1 + (-1)^{m+1}\hat{\phi}_1)}{2\hat{m}_1}\1_{\{\hat{m}_1\ne 0\}}\hat{G}^{\lambda}$ for all $\lambda \in \Lambda(\tilde{\jmath}_n)$.
\State Set $\hat{f}_m^{(j,k)} = 0$ for all coefficients $(j,k)$ belonging to blocks $\mathfrak{B}_{j\ell} = \Set{k \in \{0,\dots,2^{j-1}\} \given (\ell-1)2^{J_n} \leq k \leq \ell 2^{J_n}-1 }$ such that $\sum_{k\in \mathfrak{B}_{j\ell}}[\hat{f}_m^{(j,k)}]^2 \leq \Gamma^2 \hat{S}_n^2$ and $j\geq J_n$.
\State Let $\check{f}_m = \mathrm{clip}(\sum_{\lambda \in \Lambda(\tilde{\jmath}_n)}\hat{f}_m^{\lambda}e_{\lambda}, [0,\check{T}])$.
\EndFor
\end{algorithmic}
\end{algorithm}

\begin{algorithm}[t]
\caption{Improved estimator of $f_0$ when $s_0 < s_1$}
\label{alg:rough}
\begin{algorithmic}[1]

  \Require  $\hat{g}$, $\tilde{\jmath}_n$, $J_n$, $(\hat\psi_1^{\lambda})_{\lambda \in \Lambda(\tilde{\jmath}_n)}$, $(\hat{G}^{\lambda})_{\lambda\in \Lambda(\tilde{\jmath}_n)}$, $(\hat{f}_0^{\lambda})_{\lambda \in \Lambda(J_n)}$ as obtained in Step 4 of Algorithm~\ref{alg:complete}, $\check{T}>0$.
  \Ensure Estimator $\hat{f}_0^R$

  \State Let $\hat{T}_n = \sqrt{\frac{\log(n)}{n}}\max\left(1,\ \frac{\hat{g}}{|\hat{m}_1|}\1_{\{\hat{m}_1\ne 0\}},\, \frac{1}{1-\hat{\phi}_1^2}\1_{\{\hat{\phi}_1^2\ne 1 \}} \right)$
  \State  Compute $\hat{\alpha}_0^{\lambda} = \frac{2\hat{\psi}_1^{\lambda}}{1+\hat{\phi}_1}\1_{\{\hat{\phi}_1 \ne -1\}}$ and $\hat{\beta}_0^{\lambda} = -\left(\frac{1-\hat{\phi}_1}{1+\hat{\phi}_1}\1_{\{\hat{\phi}_1\ne -1\}}\hat{\psi}_1^{\lambda} - \frac{\hat{g}(1-\hat{\phi}_1)}{2\hat{m}_1}\1_{\{\hat{m}_1\ne 0\}}\hat{G}^{\lambda} \right)$ for all $\lambda \in \Lambda(\tilde{\jmath}_n) \backslash \Lambda(J_n)$.
  \State Set $\hat{\alpha}_0^{(j,k)} = 0$ for all coefficients $(j,k)$ belonging to blocks $\mathfrak{B}_{j\ell} = \Set{k \in \{0,\dots,2^{j-1}\} \given (\ell-1)2^{J_n} \leq k \leq \ell 2^{J_n}-1 }$ such that $\sum_{k\in \mathfrak{B}_{j\ell}}[\hat{f}_m^{(j,k)}]^2 \leq \Gamma^2 \log(n)/n$ and $j\geq J_n$.

  \State Set $\hat{\beta}_0^{(j,k)} = 0$ for all coefficients $(j,k)$ belonging to blocks $\mathfrak{B}_{j\ell} = \Set{k \in \{0,\dots,2^{j-1}\} \given (\ell-1)2^{J_n} \leq k \leq \ell 2^{J_n}-1 }$ such that $\sum_{k\in \mathfrak{B}_{j\ell}}[\hat{f}_m^{(j,k)}]^2 \leq \Gamma^2 \hat{T}_n^2$ and $j\geq J_n$.

  \State Let $\hat{f}_0^R = \mathrm{clip}(\sum_{\lambda\in \Lambda(J_n)}\hat{f}_0^{\lambda}e_{\lambda} + \sum_{\lambda\in \Lambda(\tilde{\jmath}_n)\backslash \Lambda(J_n)}(\hat{\alpha}_0^{\lambda} + \hat{\beta}_0^{\lambda})e_{\lambda} ,[0,\check{T}])$.
\end{algorithmic}
\end{algorithm}

We now discuss the computational complexity of our algorithm. As for the minimax rates, our interest is about the complexity of the algorithm as function of $n$, $\delta$, $\epsilon$, and $\zeta$. We do assume that pointwise evaluation of wavelets can be done in time $O(1)$. The complexity of step 2 of Algorithm~\ref{alg:complete} is dominated by the computation of the leading eigenvector of a $2^M\times 2^M$ matrix, which can be done in $O(2^{3M})$ time. The Theorem~\ref{thm:psitilde2}, however, prescribes that $2^M$ must be at least $\left(\frac{4R}{\zeta\sqrt{2^{2s_{*}}-1}} \right)^{1/s_{*}}$, so step~2 of Algorithm~\ref{alg:complete} is feasible in time $O(\zeta^{-3/s_{*}})$. The most demanding computation in step~3 of Algorithm~1 is to evaluate $\tilde{\psi}_2(Y_i)$ for all $i=1,\dots,n$. Since the wavelets are compactly supported, evaluating $\tilde{\psi}_2(Y_i)$ requires only summing $O(M)$ terms, and hence the step~3 can be achieved in time $O(n\cdot M) = O(n \log(1/\zeta))$. In the step~4 of Algorithm~\ref{alg:complete}, we do not need to reevaluate $(\tilde{\psi}_2(Y_i))_{i=1}^n$ since we can keep it in memory from the previous step. Exploiting the compactness of the support of the wavelets, we can compute $(\tilde{\psi}_1^{\lambda},\tilde{G}^{\lambda})_{\lambda \in \Lambda(\tilde{\jmath}_n}$ in time $O(n\cdot \tilde{\jmath}_n) = O(n\log(n/\tau^2)) = O(n\log(n/\zeta^2))$, again by Theorem~\ref{thm:psitilde2}. The thresholding of the coefficients can be trivially performed in time $O(2^{\tilde{\jmath}_n}) = O(\frac{n}{\log(n)\tau^2}) = O(\frac{n}{\log(n)\zeta^2})$ since there are $2^{\tilde{\jmath}_n}$ coefficients. Gathering all these estimates, it is seen that Algorithm~\ref{alg:complete} runs in time $O(\max(\zeta^{-3/s_{*}},\frac{n}{\log(n)\zeta^2},n\log(n/\zeta^2))$, which is typically dominated by $n\log(n)$. Furthermore, it is easily seen that running Algorithm~\ref{alg:rough} does not increase the computational complexity of the overall algorithm.

Our algorithm is thus simple and computationally efficient, avoiding any non-convex optimization step. It thus provides a promising alternative to existing methods. Further practical implementation may require additional work on tuning the hyperparameters, which is beyond the scope of this paper and a consideration for future research.

\subsection{Comparison with the case of discrete emissions}
\label{sec:previous-works}

To the best of our knowledge, the paper \cite{AGNparamhmm} is the only work that has considered the explicit dependence of the distance to the i.i.d frontier in the minimax rates of estimating HMM. In \cite{AGNparamhmm} we considered only the case of emissions on $\{1,\dots,K\}$ for known $K\geq 2$. The present work considers the more interesting (for applications) case of continuous emission densities. Although the results of both papers share some similarities, there are some aspects that are crucially different. The major difference between the discrete case and the present paper resides in the necessity of estimating the separating hyperplane described in Section~\ref{sec:estimation-procedure}. This step of the estimation procedure isn't needed for the discrete case, and was overlooked in the previous literature on nonparametric HMMs.

We note that the parametric part $\hat{Q}$ achieves the same rate in the nonparametric setting as in the multinomial setting (first inequality in Theorem \ref{th:informal}); at first glance 
this seems  unsurprising in view of the fact that the pairs $((X_n,h(Y_n))_{n\geq 0}$ form a hidden Markov model with transition matrix $Q$ for any function $h$, so that for a suitable $h$ we can reduce to a parametric setting.
This is the {\it{no bias}} phenomenon already used in \citep{MR3769193}  for multidimensional mixture models and in \citep{DanJu} for finite state space HMMs. 
Choosing $A_1,\dots,A_K$ partitioning $[0,1]$ and defining $h$ by $h(y)=k$ for $y\in A_k$, we may apply the results from the discrete setting to deduce that $Q$ can be estimated at the parametric rate given in \citep{AGNparamhmm}. However in said rate $\zeta$ must lower bound the euclidean distance between vectors $(\ip{f_0,\II_{A_k}}: k\leq K)$ and $(\ip{f_1,\II_{A_k}}: k\leq K)$. If the $A_k$ are not chosen carefully, this distance may be much smaller than $\norm{f_0-f_1}_{L^2}$, potentially even equal to 0. A suitable choice of $(A_k)_{k=1}^K$ depends on the direction $(f_0-f_1)/\norm{f_0-f_1}_{L^2} = \psi_2$, which is unknown and \emph{nonparametric}. 
This is tantamount to estimating the separating hyperplane.

Similarly, the no bias phenomenon could be exploited to build histogram estimators of $f_0$ and $f_1$ and thereby reducing the continuous case to the discrete case. Doing so, it is tempting to think that the minimax rates for the continuous case can be deduced from the results in \cite{AGNparamhmm}. Unfortunately, in \cite{AGNparamhmm} we did not explicit the dependence of the rates in the number of bins $K$, which do not enable for immediate obtentation of the rates for $f_0$ and $f_1$ since in the continuous case the number of bins must be a function of number of observations to ensure the adequate bias-variance tradeoff. Furthermore, the approach considered in this paper offers several advantages compared to the histogram approach: (i) histograms permit optimal estimation only in a very limited range of smoothness, \textit{ie.} $s_0,s_1 \in (0,2]$, compared to $(0,S]$ in this paper (where $S$ can be made large by choosing the suitable wavelet basis); (ii) making histogram estimators that are adaptive to smoothness requires some form of model selection to choose the optimal number of bins, which is avoided in this paper using thresholding; and  (iii) the estimator in \cite{AGNparamhmm} is a minimum distance estimator that requires solving a tricky non-convex optimization problem, while in the moment based estimator in the current paper is computable in almost linear time (see Section~\ref{sec:summary-algorithm}).

Finally, the continuous cases offers some curiosities in comparison with the discrete case. First, the minimax rate for estimating $f_0$ and $f_1$ in \cite{AGNparamhmm} was found to be of order $(\delta^2\epsilon^4\zeta^4n)^{-1}$. In the continuous case, although the minimax rate is also bounded by a term of order $(\delta^2\epsilon^4\zeta^4n)^{-1}$, in most regimes of interest\footnote{\textit{ie.} $\delta,\epsilon,\zeta$ small but not too small, as in Corollaries~\ref{cor:smooth-ub} and~\ref{cor:rough-ub}.} the dominating term in the rate is of order $(\delta^2\epsilon^2\zeta^2n)^{-2s_i/(2s_i+1)}$ for the smoothest density (see Corollary~\ref{cor:smooth-ub}) or $(\delta^2n)^{-2s_i/(2s_i+1)}$ (see Corollary~\ref{cor:rough-ub}). Thus, the constants $\delta,\epsilon,\zeta$ appear with different powers in the dominating term, which is a curiosity for which we do not have a clear intuition. Second, the ``borrowing estimation strength'' phenomenon described in Section~\ref{sec:nonparametric-part:2} came as a big surprise to us when writing this paper. We uncovered this phenomenon when trying to match the minimax upper and lower bounds, realizing that given one of the two densities, the other can be estimated in two ways, leading to different rates. We could't have guessed this phenomenon from our previous work \citep{AGNparamhmm} since its appear only in situations where $f_0$ and $f_1$ have different ``complexities'' -- here measured by smoothness $s_0,s_1$, in \cite{AGNparamhmm} measured by $K$ -- which we didn't considered earlier.

\ENDREV

\section{Conclusion and open questions}
\label{sec:discuss}

In this paper, we obtain precise behaviour of the minimax risk of all parameters in a nonparametric hidden Markov models, with exact constants regarding the distance to the i.i.d.\ frontier where the parameters become non-identifiable (we were not interested in the exact dependence of the constants with respect to $L$, $R$ and $\gamma^*$). In particular, we prove a surprising transition in the minimax rates depending on relative smoothnesses of the emission densities. 

Similarly to wavelet density estimation with i.i.d.\ data, the parameter $\Gamma$ used in the optimal threshold must be chosen depending on the upper $L$ for the supremum norms of $f_0,f_1$. In the i.i.d.\ case a simple workaround to adapt to $L$ is to obtain a consistent estimator of the density in $L^\infty$ norm, see \citep[Exercise 8.2.1]{GN16}, and plug into the threshold. In the HMM situation, it is not obvious how to obtain an asymptotically valid value for $L$ empirically. Our optimal threshold also depends on $\gamma^*$, which requires the preliminary step of the separation hyperplane estimation, itself requiring $L$. 
For the estimation of the separating hyperplane, we assume lower bounds on $\min \{s_0 , s_1\}$ and on $\zeta$.
 If neither $L$ nor $\sg$ is known, the interconnectedness of the parametric and nonparametric part causes us difficulty in fully adapting. 

 The main open question concerns full adaptation to get the right constants in the upper bound when a transition occurs due to different smoothnesses. 
From results herein one deduces the existence of pairs of estimators $(\check{f}_{0},\check{f}_{1})$, $(\check{f}^{R}_{0},\check{f}_{1})$, $(\check{f}_{0},\check{f}^{R}_{1})$, $(\check{f}^{R}_{0},\check{f}^{R}_{1})$ of which one pair is minimax optimal. When it is known which pair to use, we indeed get minimax optimal estimators. The question of the possibility or impossibility of choosing the correct pair without oracle guidance is of distinguished interest, yet challenging. It will be the subject of a future work.

\REV
Finally, we remark that we only investigated the minimax rates over Besov $B_{2,\infty}^{s_j}$ bodies. But our results can easily be extended to $B_{2,q_j}^{s_j}$ for any $1\leq q_j \leq\infty$. Indeed, it is trivial that $\|\cdot\|_{B_{2,\infty}^s} \leq \|\cdot\|_{B_{2,q}^s}$ for all $s>0$ and all $1\leq q \leq \infty$, from which it is deduced that $B_{2,\infty}^{s_j}$ balls are larger than $B_{2,q_j}^{s_j}$ balls, hence all our upper bounds remain valid if $B_{2,\infty}^{s_j}$ is replaced by $B_{2,q_j}^{s_j}$. On the other direction, we prove the lower bounds using a classical reduction to a multiple hypotheses testing problem, and it can be seen in our proofs (see for instance Section~\ref{proof:thm:2}) that the hypotheses we choose all belong to $\{(f_0,f_1)\;:\; \max_{i=0,1}\|f_i\|_{B_{2,1}^{s_i}} \leq R\}$. Hence our minimax lower bounds indeed hold over $B_{2,1}^{s_j}$ bodies, and thus extend trivially to $B_{2,q_j}^{s_j}$ bodies for any $1\leq q_j \leq \infty$, by the same embedding argument as before. A natural direction for the next would be to investigate the rates over $B_{p,q}^s$ bodies $1\leq p,q\leq \infty$, $s>0$, with loss measured in $L_r$ norm for $1\leq r \leq \infty$, as it is classical in nonparametric estimation (see for instance the seminal paper of \cite{donoho1996}). In this situation, we expect that the rates will exhibit the same ``elbow'' uncovered by \cite{donoho1996}, but it would be interesting to figure out the interplay between $(\delta,\epsilon,\zeta)$ and $(p,q,s,r)$, which is beyond the scope of the present paper.

\ENDREV


\acks{Kweku Abraham is supported bythe EPSRC Programme Grant on the Mathematics of Deep Learning, under the project: EP/V026259/1. {\'E}lisabeth Gassiat is supported by Institut Universitaire de France. {\'E}lisabeth Gassiat and Zacharie Naulet are support by the ANR under projects ANR-21-CE23-0035-02 and ANR-23-CE40-0018-02.}


\newpage

\appendix

\section{About the assumption of two independent samples}
\label{sec:disc-about-assumpt}

We assumed in the paper that we first get $\tilde{\psi}_{2}$ based on an independent sample of the HMM. 
Suppose we are given a single stationary HMM of length $3n$ with distribution $\PP_{\theta}$ such that the hidden Markov chain has absolute spectral gap $\gamma^*$. Let $Y'=(Y_1,\dots,Y_n)$, $\tilde{Y}'= (Y_{2n+1},\dots,Y_{3n})$, and denote $\PP_{(Y',\tilde{Y}')}$ the distribution of $(Y',\tilde{Y}')$. Denote also $\PP_{Y'}$ the distribution of $Y'$ (which is the same as the distribution of $\tilde{Y}'$ by stationarity).
 For $j=1,\ldots,4$ let $\hat{\theta}_{j}$ denote our estimator of $\theta_{j}$. Notice that $\hat{\theta}_{j}$ (resp.\ $\theta_{j}$) is non-negative and bounded by $2$ (resp.\ 1) for $j=1,2$ and $\check{T}$ (resp. $L$) for $j=3,4$, so that, denoting $M$ (resp.\ $\tilde{M}$) the upper bound, we have
 $\norm{\hat{\theta}_{j}-\theta_{j}}\leq M \vee \tilde{M}$, $\norm{\cdot}$ being the euclidean norm for $j=1,2$ and the $L^2[0,1]$-norm for $j=3,4$. Then,
 \begin{align*}
   &\EE_{\PP_{(Y',\tilde{Y}')}}\left(\norm{ \hat{\theta}_{j}-\theta_{j}}^2\right)\\
   =&\int_{0}^{M\vee \tilde{M}}\PP_{(Y',\tilde{Y}')}\left(\norm{ \hat{\theta}_{j}-\theta_{j}}^2\geq t \right)dt\\
=&~\EE_{\PP_{Y'}^{\otimes 2}}\left(\norm{ \hat{\theta}_{j}-\theta_{j}}^2\right)+\int_{0}^{M \vee \tilde{M}}\left[\PP_{(Y',\tilde{Y}')}\left(\norm{ \hat{\theta}_{j}-\theta_{j}}^2\geq t \right)-\PP_{Y'}^{\otimes 2}\left(\norm{ \hat{\theta}_{j}-\theta_{j}}^2\geq t \right)\right]dt\\
\leq&~\EE_{\PP_{Y'}^{\otimes 2}}\left(\norm{ \hat{\theta}_{j}-\theta_{j}}^2\right) + \big(M \vee \tilde{M}\big)\norm{\PP_{(Y',\tilde{Y}')}
	-\PP_{Y'}^{\otimes 2}}_{\textnormal{TV}},
 \end{align*}
 where $\norm{\cdot }_{\textnormal{TV}}$ denotes the total variation norm. Using Proposition \ref{prop:tv} below, we deduce that the first term on the right side of the last display dominates the second, hence the only cost of using one sample for the whole procedure is a multiplicative constant factor.

	\begin{proposition}
	\label{prop:tv}
There exist universal constants $C$ and $c$ such that	
	\[\norm{\PP_{(Y',\tilde{Y}')}
	-\PP_{Y'}^{\otimes 2}}_{\textnormal{TV}} \leq C e^{-c\gamma^* n}.\]
	\end{proposition}

\begin{proof}
Denote $Z_i = (X_i, Y_i )$, $i=1,\ldots,3n$, where $(X_{1},\cdots,X_{n})$ is the hidden Markov chain. Using similar notations, we have
$$
\norm{\PP_{(Y',\tilde{Y}')}
	-\PP_{Y'}^{\otimes 2}}_{\textnormal{TV}} \leq \
	\norm{\PP_{(Z',\tilde{Z}')}
	-\PP_{Z'}^{\otimes 2}}_{\textnormal{TV}}.
$$
Now, for any $(x_{1},\ldots,x_{n},x_{2n+1},\ldots,x_{3n})$, the distribution of $(Y_{1},\ldots,Y_{n},Y_{2n+1},\ldots,Y_{3n})$ conditional on $(X_{1},\ldots,X_{n},X_{2n+1},\ldots,X_{3n})=(x_{1},\ldots,x_{n},x_{2n+1},\ldots,x_{3n})$ is the same under $\PP_{(Y',\tilde{Y}')}$ and $\PP_{Y'}^{\otimes 2}$, so that
$$\norm{\PP_{(Z',\tilde{Z}')}
	-\PP_{Z'}^{\otimes 2}}_{\textnormal{TV}}\leq 2\norm{\PP_{(X',\tilde{X}')}
	-\PP_{X'}^{\otimes 2}}_{\textnormal{TV}}$$
	and the result follows from the uniform geometric ergodicity of the binary chain.
\end{proof}

\section{Proofs for the lower bounds}
\label{proof:lower}
For proving our lower bounds, we shall follow the usual path, in which we need at some point upper bounds for distances between joint distributions $P^{(n)}_{\theta}$ for different values of $\theta$. We shall use the same trick as the one used in \citep{AGNparamhmm}, that is an upper bound on the Kullback-Leibler divergence using a pseudo-distance $\rho$ between parameters, see  the end of Section III in  \citep{AGNparamhmm} for heuristics explaining the importance of $\rho$ interpreted as a fundamental statistical distance in HMM learning.

The following result is Proposition 2 in  \citep{AGNparamhmm}, for which a close look at the proof shows that it still holds with emission densities on $[0,1]$ instead of probability mass functions.
\begin{proposition}
 \label{prop:UpperBoundOnKullback}
	Assume there exists $c>0$ such  that $\min(f_0,f_1,\tilde{f}_0,\tilde{f}_1)\geq c$ uniformly on $[0,1]$.
	Then
	\begin{equation} K(P^{(n)}_\theta,P^{(n)}_{\tilde{\theta}}) \leq Cn\rho(\phi(\theta),\psi(\theta);\phi(\tilde{\theta}),\psi(\tilde{\theta}))^2,\end{equation}
	where, as in \citep{AGNparamhmm}, we have defined
	\begin{multline}
          \rho(\phi,\psi;\tilde{\phi},\tilde{\psi}) = \max\braces[\big]{ \abs{r(\phi)-r(\tilde{\phi})},\abs{\phi_2r(\phi)-\tilde{\phi}_2r(\tilde{\phi})},\\
            \abs{\phi_1\phi_2\phi_3r(\phi)-\sgn\brackets[\big]{\ip{\psi_2,\tilde{\psi}_2}}\tilde{\phi}_1\tilde{\phi}_2\tilde{\phi}_3r(\tilde{\phi})},\\
		\norm{\psi_1-\tilde{\psi}_1}_{L^2},\max(\abs{r(\phi)},\abs{r(\tilde{\phi})})\norm{\psi_2-\sgn\brackets[\big]{\ip{\psi_2,\tilde{\psi}_2}}\tilde{\psi}_2}_{L^2}}.
\end{multline}
[Recall $r(\phi)=(1/4)(1-\phi_1^2)\phi_2\phi_3^2$.]
\end{proposition}

\subsection{Proof of Theorem \ref{theo:lower:param}}
\label{proof:lower:param}
To prove Theorem \ref{theo:lower:param}, we shall use a standard two-points argument using Le Cam's method (\cite{cam:1986}, see also \cite{MR1462963}  for a review of lower bound ideas): if $\theta$ and $\tilde{\theta}$ in  $\SmoothClass \cap \RegulClass$ are such that  $|p-\tilde{p}|^2 \geq R_n$ and 
$K(P^{(n)}_\theta,P^{(n)}_{\tilde{\theta}}) \leq \alpha <1$, then
$$
 \inf_{\hat{\theta}} \sup_{\theta\in \SmoothClass \cap \RegulClass}\EE_{\theta}\Big(|\hat{p} - p|^2 \Big)
    \geq \frac{R_{n}}{4}\left(1-\sqrt{\alpha}\right).
$$
We follow the method in  the multinomial case \citep[see][]{AGNparamhmm} used to choose the two points in proving Theorems 1 and 3 therein, except that rather than defining $\psi$ according to Lemma 3 therein we choose $\psi_1=1$ and $\psi_2(x)=\sqrt{3}(2x-1)$. This choice of $\tilde{\psi}=\psi$ leads to lower bounded $f_0$ and $f_1$ (so that we can apply Proposition \ref{prop:UpperBoundOnKullback}) when $\|f_{0}-f_{1}\|_{L^2}= \zeta \leq 1/(4\sqrt{3})$, $\|f_{i}\|_{\infty}\leq 5/8$ and $\|f_{i}\|_{B^{s_{i}}_{2,\infty}}\leq 5/4 + 1/(8\sqrt{3})$, $i=0,1$, as a consequence of the inversion formulae (Lemma \ref{lem:invert-param}).
Under the assumption that for a suitable 
$\epsilon_0>0$ we have $\zeta\leq 1/(4\sqrt{3})$, $\sg\leq 1/3$, $\epsilon\leq \epsilon_0$, $\delta\leq 1/6$,  the proof of the lower bounds for $\phi$ in Theorem 3 and the lower bound for $p$ in Theorem 1 in  \citep{AGNparamhmm} goes through to get the result. 
That is: \\
When $\delta > \epsilon \zeta$, we choose $\phi=(1-3\delta,\epsilon,\zeta(1+S)^{1/2}$ with $S=(2-6\delta-\sqrt{R}_n)\sqrt{R}_n/(6\delta-9\delta)$ and $R_n =c/(n\epsilon^4 \zeta^6)$, and we choose $\tilde{\phi}=(1-3\delta-\sqrt{R}_n,\epsilon,\zeta)$, so that $r(\phi)=r(\tilde{\phi})$, $ \rho(\phi,\psi;\tilde{\phi},\tilde{\psi}) \leq 6cn^{-1/2}$ and  $|p-\tilde{p}|^2 \geq c/(n\epsilon^4 \zeta^6)$. \\
When now $\delta \leq \epsilon \zeta$, we choose $\phi=(1-3\delta,\epsilon,\zeta(1+\sqrt{R}_n/\epsilon)^{1/2}$ with $R_n =c/(n\epsilon^2 \delta^2 \zeta^4)$ and $\tilde{\phi}=(1-3\delta,\epsilon+\sqrt{R}_n,\zeta)$, so that again $r(\phi)=r(\tilde{\phi})$, $ \rho(\phi,\psi;\tilde{\phi},\tilde{\psi}) \leq cCn^{-1/2}$ for some constant $C$, and  $|p-\tilde{p}|^2 \geq c/(n\epsilon^2 \delta^2 \zeta^4)$. The theorem follows by setting $c$ small enough.

\subsection{Proof of Theorem \ref{thm:2}}
 \label{proof:thm:2}
 For the parametric term in the lower bound, we are again able to copy the proof of \citep{AGNparamhmm} Theorems 1 and 3 up to the choice of $\psi$. Under the assumption that for a suitable 
$\epsilon_0>0$ we have $\zeta\leq 1/(4\sqrt{3})$, $\sg\leq 1/3$, $\epsilon\leq \epsilon_0$, $\delta\leq 1/6$, as with proving Theorem~\ref{theo:lower:param} we choose $\psi_1=1$, $\psi_2(x)=\sqrt{3}(2x-1)$, $\tilde{\psi}=\psi$ and the proof of the lower bound for $f_0$ in \citep[Theorem 1]{AGNparamhmm} goes through. 
That is we choose $\phi=(1-3\delta,\epsilon,\zeta(1+S)^{1/2}$ with $S=(2-6\delta-\sqrt{R}_n)\sqrt{R}_n/(6\delta-9\delta)$ and $R_n =c/(n\epsilon^4 \zeta^6)$,  and we choose $\tilde{\phi}=(1-3\delta-\sqrt{R}_n,\epsilon,\zeta)$. Again $ \rho(\phi,\psi;\tilde{\phi},\tilde{\psi}) \leq 6cn^{-1/2}$ and now $\|f_0-\tilde{f}_0\|_{L^2}^2 \geq c/(n\delta^2\epsilon^4 \zeta^4)$. \\

We now prove the lower bound given in the second part of the theorem
\[ R_{\textnormal{smooth}}=(n\delta^2\epsilon^2\zeta^2)^{-s_0/(2s_0+1)}\] 

We proceed via a usual reduction to multiple testing, see for instance \citep{tsybakov:2009}.
For a suitable $c,\alpha$, it suffices to construct function 
$f_{0,m}\in \{f\;:\; \|f\|_{B_{2,\infty}^{s_0}} \leq R\},f_{1,m}\in \{f\;:\; \|f\|_{B_{2,\infty}^{s_1}} \leq R\}$, $0\leq m\leq M=\ceil{2^{c2^j}}$, for some $j$, such that
\begin{equation}
	\label{eqn:KullbackSmallf0Far}
	K\brackets[\Big]{P_{m}^{(n)}, P_{0}^{(n)}} \leq c\alpha 2^j, \quad \norm{f_{0,m}-f_{0,m'}}_{L^2} \geq cR_{\textnormal{smooth}},
\end{equation}
where $P_m^{(n)}$ denotes the law of $(Y_1,\dots,Y_n)$ under parameter $\theta_m=(p_m,q_m,f_{0,m},f_{1,m})$ (for suitable choices of the parameters $p_m,q_m$).
Indeed, given such functions, we note that
\[ \frac{1}{M\log M} \sum_{m=1}^M K\brackets[\big]{P_m^{(n)},P_0^{(n)}}\leq \alpha,\] so that applying \citep[Theorem 6.3.2]{GN16} yields the claim (for example $\alpha=1/16$ suffices). We closely follow the proof of \citep[Theorem 6.3.9]{GN16} to construct $f_{0,m}$, and use ideas inspired by \citep{AGNparamhmm} to choose the remaining parameters of $\theta_m$.

Define \begin{align*} f_{0,0}=1, \quad  f_{1,0}=f_{0,0}+\zeta \psi_{2,0}, \\
	\psi_{2,0}(x) = \sqrt{3}(2x-1).
\end{align*}
Note that $f_{0,0},f_{1,0}\geq 3/4$ pointwise (recall we assumed $\zeta\leq (4\sqrt{3})^{-1}$) and hence any small perturbations of these will remain bounded away from zero.

We choose perturbations $f_{0,m}$ of $f_0$ to satisfy the second condition of equation~\eqref{eqn:KullbackSmallf0Far}, and we choose the remaining parameters $f_{1,m}$, $p_m,q_m$ to ensure the Kullback--Leibler condition holds. Proposition~\ref{prop:UpperBoundOnKullback}, which upper bounds the KL divergence by a `` distance'' $\rho$ will be of help for the latter. 

Define the parameters $\theta_m=(p_m,q_m,f_{0,m},f_{1,m})$ as follows:
First, choose $\phi_{1,m}= - 1+c\delta$ and $\phi_{2,m}=\epsilon$ for all $m\geq 0$ and define $p_m,q_m$ according to the inversion formulae in Lemma~\ref{lem:invert-param}. Next, for $m\geq 1$, for $g_m$ to be chosen define 
\[ f_{0,m} = f_{0,0} + g_m,\quad f_{1,m} = f_{1,0}- \frac{1+\phi_1}{1-\phi_1} g_m.\]  Writing $\psi_{1,m},\psi_{2,m},\phi_{3,m}$ for the corresponding alternative parametrisation as in Section~\ref{sec:estimation-procedure}, the above choice ensures that $\psi_{1,m}=\psi_{1,0}$ regardless of the choice of $g_m$. We will choose $g_m$ (depending on $n$) such that $\norm{\psi_{2,m}-\psi_{2,0}}_{L^2}\to 0$ (uniformly in $m$) as $n\to\infty$ so that in particular it is less than $2$ eventually, hence
\[ \ip{\psi_{2,m},\psi_{2,0}}= 1- \tfrac{1}{2} \norm{\psi_{2,m}-\psi_{2,0}}_{L^2}^2 \geq 0.\]
Under the condition that $\phi_{3,m}\asymp \zeta$, one sees that
\[\rho\brackets[\big]{(\phi,\psi)(\theta_m);(\phi,\psi)(\theta_0))}=C\max \braces[\Big]{\delta\epsilon\zeta \abs{\phi_{3,m}-\phi_{3,0}},\delta\epsilon\zeta^2\norm{\psi_{2,m}-\psi_{2,0}}_{L^2}}.\] 
 We calculate 
 $ f_{0,m}-f_{1,m} = f_{0,0}-f_{1,0} +\frac{2}{2-c\delta}g_m$ and hence, using that $\norm{f_{0,0}-f_{1,0}}_{L^2}=\phi_{3,0}=\zeta$,
 \begin{align*} \abs{\phi_{3,m}-\phi_{3,0}} &= \norm{f_{0,m}-f_{1,m}}_{L^2}-\norm{f_{0,0}-f_{1,0}}_{L^2}
 	\leq \tfrac{2}{2-c\delta}\norm{g_m}_{L^2}, 
 \end{align*}
 and
 \begin{align*}
   \norm{\psi_{2,m}-\psi_{2,0}}_{L^2} &= \norm[\Big]{ \frac{f_{0,m}-f_{1,m}}{\phi_{3,m}} -\frac{f_{0,0}-f_{1,0}}{\phi_{3,0}}}_{L^2}\\
   &\leq \frac{\abs{\phi_{3,0}-\phi_{3,m}}}{\phi_{3,m}} +\frac{2\norm{g_m}_{L^2}}{2-c\delta\phi_{3,m}}\lesssim \zeta^{-1}\norm{g_m}_{L^2},
 \end{align*}
yielding 
\begin{equation}\label{eqn:RhoGm} \rho\brackets[\big]{(\phi,\psi)(\theta_m);(\phi,\psi)(\theta_0))} \leq C' \delta\epsilon \zeta \norm{g_m}_{L^2}.\end{equation} [provided $c\delta \leq 1$, say, and the condition $\phi_{3,m}\asymp \zeta$ reduces to $\norm{g_m}_{L^2}\leq \zeta/3$, say.]. 
 
Now we verify that there are $M$ valid choices of $g_m$ such that $f_{0,m}$ and $f_{0,m'}$ are suitably separated in $L^2$ distance but suitably close in Kullback--Leibler divergence as in \eqref{eqn:KullbackSmallf0Far}, and $f_{0,m}$ and $f_{1,m}$ are in the appropriate Sobolev balls. Fix $S\geq s_0$, and let $\varphi_{jk},$ $k\leq 2^j$ be a collection of wavelet functions supported in the interior of $[0,1]$ given as scaled translates $\varphi_{jk}=2^{j/2}\varphi(2^j(\cdot)-k)$ of an $S$-regular Daubechies wavelet function $\varphi$ supported in $[1,2N]$ for some $N=N(S)$. We may choose a collection of $c_02^j$ of these functions whose supports are pairwise disjoint for some $c_0=c_0(S)>0$; we denote these $\braces{\varphi_{jp} : 1\leq p\leq c_0 2^j}$ in a slight abuse of notation. By the Varsharmov--Gilbert bound \citep[Example 3.1.4]{GN16} there exist $c_1,c_2>0$ such that we may choose a set $\mathcal{M}=\braces{\beta_{m,\cdot}\in\braces{-1,1}^{c_0 2^j} : m\leq 2^{c_1 2^j}}$ for which
\[ \sum_p \abs{\beta_{mp}-\beta_{m'p}}^2 \geq c_2 2^{j}, \quad \forall p'\neq p.\]
Set $g_m = \alpha_1 \sum_p \beta_{m,p} \varphi_{jp}$ for $\alpha_1$ to be chosen and observe that
\begin{align*} \norm{f_{0,m}}_{B^{s_0}_{2,\infty}}\leq 1 + \norm{g_m}_{B^{s_0}_{2,\infty}} &=  1+\alpha_1 2^{js_0} \brackets[\big]{\sum_{p} \beta_{m,p}^2}^{1/2} 
	= 1 + c_0\alpha_1 2^{j(s_0+1/2)}, \\
	\norm{g_m}_{L^2}^2 & = \alpha_1^2 \sum_p \beta_{m,p}^2 \norm{\varphi_{jp}}_{L^2}^2= c_0 \alpha_1^2 2^{j}, \\
\norm{f_{0,m}-f_{0,m'}}^2=\norm{g_m-g_{m'}}_{L^2}^2 &= \alpha_1^2 \sum_p \abs{\beta_{m,p}-\beta_{m,p'}}^2 
 \geq c_2 \alpha_1^2 2^{j}.
\end{align*}
The first line ensures that $\norm{f_{0,m}}_{B^{s_0}_{2,\infty}}\leq R$ if $\alpha_1^2 \asymp 2^{-j(2s_0+1)}$; note also that consequently $\norm{f_{1,m}}_{B^{s_1}_{2,\infty}}\leq 1+\delta\norm{g_m}_{B^{s_1}_{2,\infty}} 
\lesssim 1+\delta 2^{j[s_{1}-s_{0}]}$. For this choice of $\alpha_1$, the second line, in conjunction with \eqref{eqn:RhoGm} and Proposition~\ref{prop:UpperBoundOnKullback} yields that $K(P_m^{(n)},P_0^{(n)})\lesssim n\delta^2\epsilon^2\zeta^2 2^{-2js_0}$, so that choosing $j$ such that $2^{j(2s_0+1)}\asymp n\delta^2\epsilon^2\zeta^2$ gives the required bound on Kullback--Leibler divergences in \eqref{eqn:KullbackSmallf0Far}. Note also that $\norm{g}_{\infty} \asymp \alpha_1 2^{j/2}$ so that for this choice of $j$ we have 
$f_{0,m}\geq 1/2, f_{1,m}\geq 1/2$ on $[0,1]$ for $n$ large, hence Proposition~\ref{prop:UpperBoundOnKullback} indeed applies, and as soon as  $(n\delta^2\epsilon^2\zeta^2)^{-s_0/(1+2s_0)}\lesssim \zeta$ we get as needed $\phi_{3,m}\asymp \zeta$. Also, $f_{1,m}$ is in the appropriate Sobolev ball if 
 $\delta^{2s_{1}+1}(n\epsilon^2\zeta^2)^{s_{1}-s_0}\lesssim 1$.
 Finally, for these choices of $\alpha_1$ and $j$, the third line yields $\norm{f_{0,m}-f_{0,m'}}_{L^2}\gtrsim (n\delta^2\epsilon^2\zeta^2)^{-s_0/(2s_0+1)}.$ \\

We finally prove the general lower bound 
\[ R_{\textnormal{rough}}=(n\delta^2)^{-s_0/(2s_0+1)},\] 
again using a reduction to multiple testing. As before choose $\phi_{1,m}=-1+c\delta,\phi_{2,m}=\epsilon$, and choose $f_{0,0},f_{1,0}$ as in proving $R_{\textnormal{smooth}}$. Now set
\[ f_{0,m} = f_{0,0}+ g_m, \quad f_{1,m}=f_{1,0}.\]
We now have $f_{0,m}-f_{1,m} = f_{0,0}-f_{1,0} + g_m$ which is of the same form as before up to the coefficient $2/(2-c\delta)\in [1,2]$ which no longer appears. The calculations for $\rho$ then go through fundamentally unchanged except that we no longer have $\psi_{1,m}=\psi_{1,0}$, hence
\[ \rho\brackets[\big]{(\phi,\psi)(\theta_m);(\phi,\psi)(\theta_0)}\leq C' \max\brackets[\big]{\delta\epsilon\zeta \norm{g_m}_{L^2}, \norm{\psi_{1,m}-\psi_{1,0}}_{L^2}}.\]
We calculate
\[ \psi_{1,m}-\psi_{1,0} = \tfrac{1}{2}(1+\phi_{1,m})f_{0,m} + \tfrac{1}{2} (1-\phi_{1,m})f_{1,m} = \tfrac{1}{2}c\delta g_m,\]
hence calculating the upper bound $C'' \delta \norm{g_m}_{L^2}$ for $\rho$.

Choosing $M=\floor{2^{c2^j}}$ functions $g_m$ as before, we again choose the factor $\alpha_1$ proportional $2^{-j(2s_0+1)}$ to ensure $\norm{f_{0,m}}_{B^{s_0}_{2,\infty}}\leq R$; note now that $\norm{f_{1,m}}_{B^{s_1}_{2,\infty}}=\norm{f_{1,0}}_{B^{s_1}_{2,\infty}}$ for all $m$ so that these are suitably bounded.

Where before we chose $2^{j(2s_0+1)}\asymp n\delta^2\epsilon^2\zeta^2$ to obtain the required bound on the KL divergences in equation~\eqref{eqn:KullbackSmallf0Far}, we now must choose $2^{j(2s_0+1)}\asymp n\delta^2$. This leads to $\norm{f_{0,m}-f_{0,m'}}_{L^2}\gtrsim (n\delta^2)^{-s_0/(2s_0+1)}$ so that equation~\eqref{eqn:KullbackSmallf0Far} holds with $R_{\textnormal{rough}}= (n\delta^2)^{-s_0/(2s_0+1)}$ in place of $R_{\textnormal{smooth}}$. This yields the claim.

\section{Proofs for the upper bounds}
\label{sec:proofs}

\subsection{Overview of the proofs}

The proofs of the upper bounds proceed according to the following strategy. 

In Section~\ref{sec:useful-lemmas}, we first state a series of lemmas whose purpose is to simplify further proofs. These lemmas are elementary -- yet crucial -- results about the reparameterization $\theta \mapsto \phi$ and its inversion, the simplification of the expression of $m(\phi)$ given in \eqref{def:m}, and the (quasi) inversion formula to recover $\phi$ from $m(\phi)$. The section also contains auxiliary results on $p_\theta^{(k)}$ and useful concentration inequalities for Markov chains.

The Section~\ref{sec:conc-ineq-param} establishes deviation inequalities for $|\hat{m}_j - m(\phi)_j|$ and $|\hat{m}_j/m(\phi)_j - 1|$. These deviation inequalities are used many times after when using the method of moments to estimate $\theta$. It is worth noting that $m(\phi)$, together with the coefficients $\{\psi_1^{\Phi_{Jk}}\}$, $\{\psi_1^{\Psi_{jk}}\}$, $\{G^{\Phi_{Jk}}\}$ and $\{G^{\Psi_{jk}}\}$ are all easy functionals of $p_\theta^{(s)}$ for some $s\geq 1$, and can all be estimated at a universal marginal rate $c\sqrt{n}$, with $c$ eventually depending on $\gamma_*$ and $L$ but nothing else. Thus one can think of estimating those quantities as solving the \textit{direct} problem. The main challenge is to translate the inequalities for the direct problem onto inequalities for the \textit{inverse} problem, \textit{i.e}. for $(Q_\theta,f_0,f_1)$, which is the purpose of the subsequent subsections.

The Section~\ref{sec:proof-prop-refpr} proves the Theorem~\ref{thm:concent-pq}, \textit{ie}. the minimax upper bound for estimating $Q_\theta$. This is done in many steps, that can be on a high level summarized by upper bounding $|\hat{p} - p|$ (similarly $|\hat{q} - q|$) by a parameter dependent term times $\max_{j=1,2,3}|\hat{m}_j - m(\phi)_j|$, and then using the concentration inequalities for $|\hat{m}_j - m(\phi)_j|$ to conclude. Here, we emphasize that the obtention of a tight upper bound for $|\hat{p} - p|$ in terms of the deviation of the moments is crucial in obtaining the exact minimax rate and requires substantial work.

The Section~\ref{sec:proof-thm:1} proves the Theorem~\ref{thm:1}, \textit{ie.} the minimax upper bounds to estimate $f_0$ and $f_1$ when $s_0 = s_1$.  The proof relies on a somewhat classical decomposition of the risk when studying block-threshold wavelet density estimators, with additional cares to be taken due to the optimal threshold depending on the parameters and being estimated. Modulo these additional cares, the proofs follows the classical steps and is based on deviation inequalities for $\|\hat{f}_m^{\Block_{j\ell}} - f_m^{\Block_{j\ell}}\|$ and similar quantities, to establish that the chosen threshold balances the bias and variance optimally. In contrast with classical density estimation, estimation of the empirical wavelets coefficients requires here to solve an inverse problem. This is done by upper bounding $\|\hat{f}_m^{\Block_{j\ell}} - f_m^{\Block_{j\ell}}\|$ in term of $\|\hat{\psi}_1^{\Block_{j\ell}} - \psi_1^{\Block_{j\ell}}\|$, $\|\hat{G}_1^{\Block_{j\ell}} - G^{\Block_{j\ell}}\|$ and $\max_{j=1,2,3}|\hat{m}_1 - m(\phi)_j|$, and then using deviation inequalities for the the direct problem.

The Section~\ref{sec:proof:thm:rough} proves the Theorem~\ref{thm:rough}, \textit{ie}. the minimax upper bounds to estimate $f_0$ and $f_1$ when $s_0 < s_1$. The ideas of the proof are very similar to Theorem~\ref{thm:1}. The main difference resides in the definition of the empirical wavelet coefficients.

The Section~\ref{sec:proof-theorem-psitilde2} proves the Theorem~\ref{thm:psitilde2} about the estimation of the separating hyperplane. Recall that the estimator of the hyperplane is obtained by estimating the leading eigenvector of a certain gram matrix $\mathcal{G}$ from  the leading eigenvector of its empirical version $\tilde{\mathcal{G}}$. The proof of the theorem is based on the celebrated Davis-Kahan theorem and the obtention of a deviation bound for $\|\tilde{\mathcal{G}} - \mathcal{G}\|_{\mathrm{op}}$, which is based on a $\varepsilon$-net argument together with concentration inequalities for Markov chains.

Finally, the Sections~\ref{sec:proof-cor:smooth-ub} and~\ref{sec:proof-cor:rough-ub} proves the Corollaries~\ref{cor:smooth-ub} and~\ref{cor:rough-ub}, respectively. Those follow immediately from the Theorems~\ref{thm:1} and~\ref{thm:rough} and straightforward computations.

\subsection{Useful lemmas}
\label{sec:useful-lemmas}

\begin{lemma}\label{lem:invert-param}
	The parametrisation $\theta\mapsto (\phi,\psi)$ from \eqref{eqn:def:phi-psi} is invertible: 
	\begin{align*}
		p&= \tfrac{1}{2} (1-\phi_2)(1-\phi_1), \\
		q &= \tfrac{1}{2}(1-\phi_2)(1+\phi_1),\\
		f_0 &=\psi_1-\tfrac{1}{2}\phi_1\phi_3\psi_2+\tfrac{1}{2}\phi_3\psi_2,\\
		f_1 &=\psi_1-\tfrac{1}{2}\phi_1\phi_3\psi_2-\tfrac{1}{2}\phi_3\psi_2.
	\end{align*}
	Defining $p_\pm = \tfrac{1}{2} (1\mp \tilde{s}\phi_1) (1-\phi_2)$, where $\tilde{s}\coloneqq \sign(\Inner{\psi_2,\tilde{\psi}_2})$ we have 
	\begin{equation*}
		(p_+,p_-)%
		\coloneqq%
		\begin{cases}
			(p,q) &\mathrm{if}\ \tilde{s} > 0,\\
			(q,p) &\mathrm{if}\ \tilde{s} < 0.
		\end{cases}
	\end{equation*}
	
Recalling the definition \eqref{def:m} of $m$, define
	\begin{equation*}
		g \coloneqq \phi_3|\tilde{\mathcal{I}}|%
		= \frac{\sqrt{4m_1^2m_2 + m_3^2}}{m_2},
	\end{equation*}
	and define 
	\begin{equation*}
		f_{\pm} \coloneqq \psi_1 \pm \frac{g(1 \mp \tilde{s}\phi_1)}{2 m_1}G,\qquad G \coloneqq \frac{m_1 \psi_2}{\tilde{\mathcal{I}}}.
	\end{equation*}
	Then 
	\begin{equation*}
		(f_+,f_-)%
		\coloneqq%
		\begin{cases}
			(f_0,f_1) &\mathrm{if}\ \tilde{s} > 0,\\
			(f_1,f_0) &\mathrm{if}\ \tilde{s} < 0.
		\end{cases}
	\end{equation*} 
\end{lemma}

The proof is elementary. Note that $\PP_n^{(1)}(\Phi_{Jk})$ is the empirical estimator of $\EE_{\theta}[\Phi_{Jk}] = \ip{\Phi_{Jk},\psi_1}$, hence the above lemma justifies the use of $\hat{f}^{\Phi_{Jk}}_0$, $\hat{f}^{\Phi_{Jk}}_1$ from Section~\ref{sec:nonparametric-part:1}.

\begin{lemma}
  \label{lem:compute-m}
  Given $p^{(3)}_{\phi,\psi}$ as defined in \eqref{eq:17} and any function $\tilde{\psi}_2$, one can compute%
    \begin{align*}
      r(\phi)\tilde{\mathcal{I}}^2%
      &= \EE_{\theta}(\tilde{\psi}_2 \otimes \tilde{\psi}_2) - \EE_{\theta}(\tilde{\psi}_2)^2\\
      r(\phi)\phi_2 \tilde{\mathcal{I}}^2%
      &= \EE_{\theta}(\tilde{\psi}_2 \otimes 1 \otimes \tilde{\psi}_2) - \EE_{\theta}(\tilde{\psi}_2)^2\\
      r(\phi)\phi_1\phi_2\phi_3 \tilde{\mathcal{I}}^3%
      &=-\EE_{\theta}(\tilde{\psi}_2\otimes \tilde{\psi}_2 \otimes \tilde{\psi}_2)%
        + \EE_{\theta}(\tilde{\psi}_2)^3%
        + \Big(2 r(\phi)\tilde{\mathcal{I}}^2 + r(\phi)\phi_2\tilde{\mathcal{I}}^2 \Big) \EE_{\theta}(\tilde{\psi_2}).
    \end{align*}
  Also if $G= m_1\psi_2/\tilde{\mathcal{I}}$, then $\ip{\Phi_{Jk},G}=\EE[\tilde{\psi}_2\otimes \Phi_{Jk}]-\EE_\theta [\tilde{\psi}_2]\EE_\theta[\Phi_{Jk}] $.
\end{lemma}

\begin{proof}
	We compute, from the expression for $p^{(3)}_{\phi,\psi}$, applied for example to $\tilde{\psi}_2\otimes 1 \otimes 1$ and using that $\ip{\psi_1,1}=\int \psi_1 =1$, $\ip{\psi_2,1}=0$, 
\begin{align*}
	\EE_{\theta}(\tilde{\psi}_2)%
	&= \Inner{\psi_1,\tilde{\psi}_2}\\
	\EE_{\theta}(\tilde{\psi}_2 \otimes \tilde{\psi}_2)%
	&= \Inner{\psi_1,\tilde{\psi}_2}^2 + r(\phi)\Inner{\psi_2,\tilde{\psi}_2}^2\\
	\EE_{\theta}(\tilde{\psi}_2\otimes 1 \otimes \tilde{\psi}_2)%
	&= \Inner{\psi_1,\tilde{\psi}_2}^2 + r(\phi)\phi_2\Inner{\psi_2,\tilde{\psi}_2}^2\\
	\EE_{\theta}(\tilde{\psi}_2\otimes \tilde{\psi}_2 \otimes \tilde{\psi}_2)%
	&= \Inner{\psi_1,\tilde{\psi}_2}^3 + (2r(\phi) + r(\phi)\phi_2)\Inner{\psi_2,\tilde{\psi}_2}^2\Inner{\psi_1,\tilde{\psi}_2}%
	- r(\phi)\phi_1\phi_2\phi_3\Inner{\psi_2,\tilde{\psi}_2}^3
\end{align*}
Then $m \coloneqq (r(\phi)\tilde{\mathcal{I}}^2,r(\phi)\phi_2\tilde{\mathcal{I}}^2, r(\phi)\phi_1\phi_2\phi_3\tilde{\mathcal{I}}^3 )$, $\tilde{\mathcal{I}}\coloneqq \ip{\psi_2,\tilde{\psi}_2}$ is easily extracted.

Similarly, $\EE_\theta [ \tilde{\psi}_2 \otimes \Phi_{Jk}] = \ip{\psi_1,\tilde{\psi}_2}\ip{\psi_1,\Phi_{Jk}} + r(\phi) \tilde{\mathcal{I}}\ip{\psi_2,\Phi_{Jk}}$, and the expression for the coefficient of $G$ can be extracted.
\end{proof}

\begin{lemma}[Inversion formulas for $m$]
  \label{lem:invformulam}
  Let $m(\phi) = ( r(\phi)\tilde{\mathcal{I}}^2, r(\phi)\phi_2\tilde{\mathcal{I}}^2,\, r(\phi)\phi_1\phi_2\phi_3\tilde{\mathcal{I}}^3)$ with $\tilde{\mathcal{I}} \ne 0$. Then,
  \begin{align*}
    \sgn(\tilde{\mathcal{I}})\phi_1
    &= \frac{m_3(\phi)}{\sqrt{4m_1(\phi)^2m_2(\phi) + m_3(\phi)^2}},\\
    \phi_2%
    &= \frac{m_2(\phi)}{m_1(\phi)},\\
    \phi_3|\tilde{\mathcal{I}}|%
    &= \frac{\sqrt{4m_1(\phi)^2m_2(\phi) + m_3(\phi)^2}}{m_2(\phi)}.
  \end{align*}
\end{lemma}
\begin{proof}
  This can be checked via direct computations.
\end{proof}

The following bounds are immediate from the definition of the parameter space \eqref{eqn:smoothclass} and the reparametrisation \eqref{eqn:def:phi-psi} (recall also the definition \eqref{eqn:def:r} of $r$).
\begin{lemma}
	\label{lem:bounds-parameters}
	
	
	For $\phi$ corresponding to $\theta\in\SmoothClass$ we have the bounds
	\begin{equation*}
		-\frac{1-\delta}{1+\delta} \leq \phi_1\leq \frac{1-\delta}{1+\delta}, \quad \epsilon\leq \abs{\phi_{2}}\leq 1-2\delta, \quad  \phi_3\geq \zeta, \quad \delta\epsilon\zeta^2/4\leq \abs{r(\phi)}\leq \phi_3^2/4.
	\end{equation*}
\end{lemma}

\begin{lemma}
  \label{lem:basic-relations}
  Let $m_1,m_2,m_3$ be defined as in \eqref{def:m} and let $v \coloneqq 4m_1^2m_2 + m_3^2$. Then $0 \leq m_2 \leq \abs{m_1}$ and $\sqrt{v}=\tilde{\mathcal{I}}^3r(\phi)\phi_2\phi_3=\tilde{\mathcal{I}}m_2\phi_3$. Furthermore, for every $\theta \in \SmoothClass$ and $0 < \delta \leq 1$, $0 < \epsilon \leq 1$, and $0 < \zeta \leq 1$:
  \begin{equation*}
    \abs[\Big]{\frac{g}{m_1}}\leq \frac{4}{\delta \epsilon \zeta|\tilde{\mathcal{I}}|}%
    ,\qquad%
    \frac{\max(1,g)}{m_2}%
    \leq \frac{4}{\delta \epsilon^2\zeta^2|\tilde{\mathcal{I}}|^2}
    ,\qquad%
    \frac{\max(1,g)}{gm_2}%
    \leq \frac{4}{\delta \epsilon^2\zeta^3|\tilde{\mathcal{I}}|^3}.
  \end{equation*}
\end{lemma}
\begin{proof}
  Observe that $m_2 = m_1\phi_2$ and $|\phi_2| \leq 1$. Also, $m_2 = r(\phi)\phi_2\tilde{\mathcal{I}}^2 = \frac{1}{4}(1-\phi_1^2)\phi_2^2\phi_3^2\tilde{\mathcal{I}}^2 \geq 0$. Similarly,
  \begin{align*}
    v%
    &= 4r(\phi)^2\tilde{\mathcal{I}}^4\cdot r(\phi)\phi_2\tilde{\mathcal{I}}^2 + r(\phi)^2\phi_1^2\phi_2^2\phi_3^2 \tilde{\mathcal{I}}^6%
      = r(\phi)^2 \tilde{\mathcal{I}}^6\Big( 4r(\phi)\phi_2 + \phi_1^2\phi_2^2\phi_3^2 \Big) = r(\phi)^2\phi_2^2\phi_3^2 \tilde{\mathcal{I}}^6.
  \end{align*}
  Next, observe that $\frac{g}{m_1} = \frac{\phi_3|\tilde{\mathcal{I}}|}{\frac{1}{4}(1-\phi_1^2)\phi_2\phi_3^2|\tilde{\mathcal{I}}|^2} = \frac{4}{(1-\phi_1^2)\phi_2\phi_3|\tilde{\mathcal{I}}|}$. But $0 \geq 1 - \phi_1^2 \geq \frac{4\delta}{(1+\delta)^2} \geq \delta$, $|\phi_2| \geq \epsilon$, and $\phi_3 \geq \zeta$ by Lemma~\ref{lem:bounds-parameters}. Similarly, since $g = \phi_3|\tilde{\mathcal{I}}| \leq \zeta \leq 1$, $0 \leq \frac{\max(1,g)}{m_2} = \frac{1}{ m_2} = \frac{4}{(1-\phi_1^2)\phi_2^2\phi_3^2|\tilde{\mathcal{I}}|^2} \leq \frac{4}{\delta \epsilon^2\zeta^2|\tilde{\mathcal{I}}|^2}$.
\end{proof}



\begin{lemma}
  \label{lem:2}
  For any $k \geq 1$,
  \begin{equation*}
    \norm{p_{\theta}^{(k)}}_{\infty}%
    \leq \max\brackets[\big]{ \norm{f_0}_{\infty},\ \norm{f_1}_{\infty} }^k.
  \end{equation*}
Consequently, for any $\theta\in\RegulClass$ and any measurable function $h:\RR^k\to \RR$, we have
\[ \EE_\theta[h(Y_1,\dots,Y_k)^2]\leq L^{k}\norm{h}_{L^2}^2.\]
\end{lemma}
\begin{proof}
  Observe that $p_{\theta}^{(k)}(y_1,\dots,y_k) = \sum_{x_1,\dots,x_k}\PP_{\theta}(X_1=x_1,\dots,X_k=x_k)\prod_{i=1}^kf_{x_i}(y_i)$. The first conclusion is immediate, and the second follows from \[\EE_\theta h(Y_1,\dots,Y_k)^2 = \int p^{(k)}_\theta(y_1,\dots,y_k)h(y_1,\dots,y_k) \intd y_1\cdots \intd y_k\leq \norm{p^{(k)}}_{\infty} \norm{h}_{L^2}^2. \]
\end{proof}
\begin{remark}\label{rem:lem2}
	The proof adapts to yield $E_\theta [h(Y_1,Y_3)^2] \leq L^2 \norm{h}_{L^2}^2$ rather than the weaker bound $L^3\norm{h}_{L^2}^2$ directly obtainable using the lemma. Indeed, we have
	\[ \sup_{y_1,y_3} \abs[\Big] {\int p^{(3)}_\theta (y_1,y_2,y_3) \intd y_2} = \sum_{x_1,x_2,x_3} \PP_\theta(X_1=x_1,X_2=x_2,X_3=x_3) f_{x_1}(y_1)f_{x_3}(y_3)\leq L^2,  \] and the rest of the proof is the same.
\end{remark}

\begin{lemma}
  \label{lem:ub-phi3}
  For all $\theta \in \RegulClass$, $\phi_3 \leq \sqrt{2L}$.
\end{lemma}
\begin{proof}
	We compute $\phi_3^2=\int_0^1(f_0-f_1)^2 \leq \norm{f_0-f_1}_{\infty} \int_0^1 (\abs{f_0}+\abs{f_1})=2\norm{f_0-f_1}_{\infty}$. 
Since we have the pointwise bounds $0 \leq f_0,f_1 \leq L$ for every $\theta\in \RegulClass$, it follows that $\phi_3^2 \leq 2 L$. We remark that this upper bound is tight since it is attained for instance when $f_0$ is the uniform density on $[0,1/L]$ and $f_1$ the uniform density on $[1-1/L,1]$.
\end{proof}

We now recall the following result, which is adapted from \citep{Paulin2015} 
 and will be key to getting deviation inequalities of empirical ingredients in our procedures. 

\begin{lemma}
  \label{lem:paulin}
  Let $1 \leq k \leq 3$ and let $h : \Reals^k \to \Reals$ be measurable. There is a universal constant $C > 0$ such that for all $\theta$, all $n\geq 4$ such that $n\gamma^* \geq 1/99$, and all $t \geq 0$
  \begin{equation*}
    \PP_{\theta}\Big( \abs{\PP_n^{(k)}(h) - \EE_{\theta}(h)}  \geq t\Big)%
    \leq \exp\Big(- \frac{Cnt^2 \gamma^*}{\EE_{\theta}(h^2) + \norm{h}_{\infty}t} \Big).
  \end{equation*}
  This in particular implies that there is a is a universal constant $C > 0$ such that for all $\theta$, all $n\geq 4$ such that $n\gamma^* \geq 1/99$, and all $x \geq 0$
  \begin{equation*}
    \PP_{\theta}\Bigg( \abs{\PP_n^{(k)}(h) - \EE_\theta(h)} \geq C \sqrt{ \frac{\EE_\theta[h^2]x}{n\gamma^*}} + \frac{C \norm{h}_{\infty}x}{n\gamma^*}  \Bigg) \leq e^{-x}.
  \end{equation*}
\end{lemma}
\begin{proof}
  Since $1 \leq k \leq 3$, we can view any function $h : \Reals^k \to \Reals$ as $\tilde{h} : \Reals^6 \to \Reals$ with $h(Y_i,\dots,Y_{i+k}) = \tilde{h}(X_i,X_{i+1},X_{i+2},Y_i,Y_{i+1},Y_{i+2})$. The process $\big((X_i,X_{i+1},X_{i+2},Y_i,Y_{i+1},Y_{i+2}) \big)_{i\geq 1}$ is a stationary Markov Chain with pseudo spectral gap (defined as in \cite{Paulin2015}) $\psgap\geq \gamma^*/8$, by our assumptions. Indeed, calculations in \citep[Lemma 1]{AGNparamhmm} based on the relationship between the pseudo spectral gap and the mixing time show that $\psgap \geq 0.5((\log 4/ \gamma^*)+2)^{-1}$, and the bound $\max(\gamma^*,\log 2)\leq 1$ yields the claimed bound.

 By Theorem~3.4 in  \citep{Paulin2015} (though note there is an updated version of the paper on arXiv), for $S_n \coloneqq \sum_{i=1}^{n-k+1}\tilde{h}(X_i,X_{i+1},X_{i+2},Y_i,Y_{i+1},Y_{i+2})$ we do have for any $t\geq 0$
  \begin{equation*}
    \PP_{\theta}\Big( \abs{ S_n - \EE_{\theta}(S_n) }\geq t\Big)%
    \leq \exp\Big(- \frac{t^2 \psgap}{8(n - k +1 + 1/\psgap)\EE_{\theta}(h^2) + 20\norm{h}_{\infty}t} \Big).
  \end{equation*}
  Dividing $S_n$ by $n-k+1$ and replacing $n-k+1$ and $\psgap$ by the respective lower bounds $n/2$ and $\gamma^*/8$, 
  we find that 
  \begin{align*}
    \PP_{\theta}\Big( \abs{\PP_n^{(k)}(h) - \EE_{\theta}(h)}  \geq t\Big)%
    &\leq \exp\Big(- \frac{nt^2 \gamma^*/16}{8(1 + \frac{16}{n\gamma^*})\EE_{\theta}(h^2) + 20\norm{h}_{\infty}t} \Big)\\
    &\leq \exp\Big(- \frac{nt^2 \gamma^*}{16 \times 8 \times (1+16\times 99)\times \EE_{\theta}(h^2) + 320\norm{h}_{\infty}t} \Big)
  \end{align*}
  under the assumption that $n\gamma^* \geq 1/99$. 
  The result follows by taking $t=C\sqrt{\EE_\theta[h^2]x/(n\gamma^*)}+C\norm{h}_{\infty} x/(n\gamma^*)$ for $C$ a sufficiently large constant that the argument of the exponential is smaller than $-x$ (by splitting into cases based on which of the two terms in the denominator is larger it can be seen that it suffices to take $C=\max(\sqrt{2\times 16\times 8\times (1+16\times 99)},640)=640$), yielding the claim. 
\end{proof}

The following consequence of deviation inequalities to get bounds in expectation will  also be used.

\begin{lemma}
  \label{lem:9}
  Suppose $X$ is a non-negative random variable and there exist $a,b,c> 0$ such that $\PP(X > b\sqrt{x/n} + ax/n ) \leq ce^{-x}$ for all $x > 0$. There  for all $d \geq 0$ 
    \begin{equation*}
  	\EE(X^2\1_{ \{X > d\} })%
  	\leq c\Big(d^2 + \frac{5b^2}{4n} + \frac{7a^2}{2n^2} \Big)\exp\Big(- \frac{n d^2}{2b^2 + 8a d} \Big).
  \end{equation*}
\end{lemma}
\begin{proof}
  Applying the standard identity $\EE(Y)= \int_0^\infty \PP(Y>y)\intd y$ for any non-negative random variable $Y$ to $Y=X^2\1_{\{X > d \}}$ and making the substitution $y=u^2$ we obtain
  \begin{align*}
    \EE(X^2\1_{ \{X > d\} })%
    &= \int_0^{\infty}\PP\big( X^2\1_{\{X > d \} } > y \big)\intd y\\
    &= \int_0^{\infty}\PP\big( X > \max(d,\sqrt{y}) \big)\intd y\\
    &= \int_0^{d^2}\PP(X > d)\intd y%
      + \int_{d^2}^{\infty}\PP\big( X > \sqrt{y} \big)\intd y\\
    &= d^2\PP(X > d)%
      + \int_d^{\infty}2u\PP(X > u)\intd u.
  \end{align*}
  Define $\varphi(x) \coloneqq \frac{b}{2a}\big(\sqrt{1 + 4ax/b^2} - 1 \big)$. For the change of variables $u = b\sqrt{x/n} + a x/ n$ one calculates that $x=n\varphi(u)^2$ and hence computes, using Cauchy--Schwarz for the penultimate line,
\begin{align*}
	\int_d^{\infty}u\PP(X > u)\intd u
	&=\int_{n\varphi(d)^2}^{\infty}\Big(b \sqrt{\frac{x}{n}} + a\frac{x}{n} \Big)\Big(\frac{b}{2\sqrt{nx}} + \frac{a}{n} \Big)\PP\Big(X > b \sqrt{\frac{x}{n}} + a\frac{x}{n} \Big)\intd x\\
	&\leq c\int_{n\varphi(d)^2}^{\infty}\Big(\frac{b^2}{2n} + \frac{3}{2}\frac{b}{\sqrt{n}}\frac{a\sqrt{x}}{n} + \frac{a^2x}{n^2}\Big)e^{-x}\intd x\\
	&\leq c\int_{n\varphi(d)^2}^{\infty}\Big(\frac{5b^2}{4n} + \frac{7a^2x}{4n^2}\Big)e^{-x}\intd x\\
	&= \frac{c}{4} \Big(\frac{5b^2}{n} + \frac{7a^2}{n^2} (n\varphi(d)^2+1) \Big)e^{-n\varphi(d)^2}.
\end{align*}  
Similarly one has
  \begin{align*}
    \PP(X > d)%
    &= \PP\Big(X > b\sqrt{\frac{n\varphi(d)^2}{n}} + a\frac{n \varphi(d)^2}{n} \Big)%
      \leq c e^{-n\varphi(d)^2}.
  \end{align*}
   To obtain the final expression, we remark that $x e^{-x} \leq \frac{2}{e}e^{-x/2}$, that $2/e+1\leq 2$ and that for all $x > 0$%
  \begin{equation*}
    \varphi(x) \geq \frac{b}{2a}\frac{4ax /b^2}{2\sqrt{1 + 4ax/b^2}}%
    = \frac{x/b}{\sqrt{1 + 4ax/b^2}}. 
  \end{equation*}
\end{proof}

\subsection{Inequalities for the $m$ functional}
\label{sec:conc-ineq-param}

Recall the definitions 
\begin{align*}
	\hat{m}_1%
	&\coloneqq \PP_n^{(2)}(\tilde{\psi}_2 \otimes \tilde{\psi}_2) - \PP_n^{(1)}(\tilde{\psi}_2)^2,\\
	\hat{m}_2%
	&\coloneqq \PP_n^{(3)}(\tilde{\psi}_2\otimes 1 \otimes \tilde{\psi}_2) - \PP_n^{(1)}(\tilde{\psi}_2)^2\\
	\hat{m}_3%
	&= - \PP_n^{(3)}(\tilde{\psi}_2\otimes \tilde{\psi}_2 \otimes \tilde{\psi}_2)%
	+ \PP_n^{(1)}(\tilde{\psi}_2)^3%
	+ \big(2\hat{m}_1 + \hat{m}_2 \big)\PP_n^{(1)}(\tilde{\psi}_2),
\end{align*}
estimators of the functional $m$ defined in \eqref{def:m} as $m = (r(\phi)\tilde{\mathcal{I}}^2,r(\phi)\phi_2\tilde{\mathcal{I}}^2,r(\phi)\phi_1\phi_2\phi_3\tilde{\mathcal{I}}^3)$ with $\tilde{\mathcal{I}}=\ip{\psi_2,\tilde{\psi}_2}$, and deduced from Lemma~\ref{lem:compute-m} to be equal to what is obtained in the expressions for $\hat{m}$ on replacing every instance of an empirical estimator by the expectation operator. [This does not mean that $\EE_\theta \hat{m} = m$, since there are powers and products in the expressions.] In this section, we prove deviation inequalities for the estimators of $m$, from which we deduce bounds in expectation. The results of this section will be used to prove Theorem \ref{thm:concent-pq} and Theorem  \ref{thm:1}.

We remark that the results are mostly uniform over the whole class $\RegulClass$, not our final parameter set $\SmoothClass\cap\RegulClass$. The need to intersect with $\SmoothClass$ arises for ensuring the parameters $\theta$ are identifiable from $m$. 
\begin{proposition}
  \label{pro:8alt}
    Let $n\gamma^* \geq 1/99$. Then there exists a universal constant $C>0$ such that for all $x\geq 0$
  \begin{equation*}
  \sup_{\theta \in \RegulClass}  \PP_{\theta}\Bigg(\max_{j=1,2}|\hat{m}_j - m_j| \geq  CL\sqrt{\frac{x}{n\gamma^*}} + C\max(\tau,\sqrt{L})^2\frac{x}{n\gamma^*} \Bigg) \leq 3e^{-x}.
  \end{equation*}
\end{proposition}

\begin{proposition}
  \label{pro:8}
  Let $n\gamma^* \geq 1/99$. Then there exists a universal constant $C>0$ such that for all $x\geq 0$
  \begin{equation*}
\sup_{\theta\in\RegulClass}    \PP_{\theta}\Bigg(\max_{j=1,2,3}|\hat{m}_j - m_j| \geq  CL^{3/2}\sqrt{\frac{x}{n\gamma^*}} + C\max(\tau,\sqrt{L})^3\frac{x}{n\gamma^*} \Bigg) \leq 4e^{-x}.
  \end{equation*}
\end{proposition}
\begin{proposition}
	\label{pro:expec-maxm}
	There exists a constant $K>0$ such that whenever $n\gamma^*\geq 1/99$, 
	\begin{equation*}
		\sup_{\theta \in \RegulClass}   \EE_{\theta}\Big(\max_{j=1,2,3}|\hat{m}_j - m_j|^2 \Big)%
		\leq K \brackets[\Big]{ \frac{L^3}{n\gamma^*}%
			+ \frac{\max(\tau,\sqrt{L})^6}{(n\gamma^*)^2}}.
	\end{equation*}
\end{proposition}

\begin{proposition}
	\label{pro:2}
	Assume $n\gamma^* \geq 1/99$, $\abs{\tilde{\mathcal{I}}}\geq 7/8$ and $\zeta\leq 1$, and define the event 
	\begin{equation}\label{eq:def-event-omegan}
		\Omega_n \coloneqq \Set*{\max_{j=1,2}\Big|\frac{\hat{m}_j}{m_j} - 1\Big| \leq \frac{1}{2},\ \max_{j=1,2,3}|\hat{m}_j - m_j| \leq \frac{gm_2}{44\max(1,g)}  }.
	\end{equation}
	Then there exists a universal constant $C > 0$ such that
	\begin{align*}
		\sup_{\theta\in\RegulClass} \PP_{\theta}(\Omega_n^c)%
		&\leq 7\exp\Bigg(- \frac{C n \gamma^* g^2m_2^2/\max(1,g)^2}{L^3 + \max(\tau,\sqrt{L})^3 gm_2/\max(1,g)} \Bigg),\\
		\sup_{\theta\in \SmoothClass \cap \RegulClass}\PP_{\theta}(\Omega_n^c)%
		&\leq
		7\exp\Bigg(- \frac{C n \sg \delta^2\epsilon^4\zeta^6 }{L^3 + \max(\tau,\sqrt{L})^3 \delta \epsilon^2\zeta^3} \Bigg).
	\end{align*}
\end{proposition}

The proof of Proposition~\ref{pro:8} is the most involved of these, and we outline how to prove the other results before addressing it.
\begin{proof}[Proof of Proposition~\ref{pro:8alt}]
	The proof is similar to the proof of Proposition~\ref{pro:8}, where $\max_{j=1,2,3}|\hat{m}_j - m_j|$ is controlled. Here, since only $\hat{m}_1$ and $\hat{m}_2$ are involved, the proxy variance is no more than $L$ since only $\PP_n^{(2)}$ is involved (versus $L^{3/2}$ when $\PP_n^{(3)}$ is involved).
\end{proof}

\begin{proof}[Proof of Proposition~\ref{pro:expec-maxm}]
	In view of Proposition~\ref{pro:8} we may apply Lemma~\ref{lem:9} with $a = C\max(\tau,\sqrt{L})^3/\gamma^*$, $b = CL^{3/2}/\sqrt{\gamma^*}$, $c=8$ and $d=0$ to obtain the claimed bound.
\end{proof}
\begin{proof}[Proof of Proposition~\ref{pro:2}]
	The first inequality essentially follows from Propositions~\ref{pro:8alt} and \ref{pro:8} and a change of variables: see Lemmas~\ref{lem:pro:2a} and \ref{lem:pro:2b} (and the sentence after the former) below where this change of variables is explicitly made.
The second inequality follows from the fact that $\frac{\max(1,g)}{gm_2} \leq \frac{16}{\delta\epsilon^2\zeta^3 \tilde{\mathcal{I}}^2}$ on $\SmoothClass$ by Lemma~\ref{lem:basic-relations}. 
\end{proof}

\begin{proof}[Proof of Proposition~\ref{pro:8}]
We have that $\max_{j=1,2,3}\abs{\hat{m}_j - m_j} \leq 16\norm{\tilde{\psi}_2}_{\infty}^3 \leq 16\tau^3$ by construction. Hence whenever $x > n\gamma^*$ we have with probability $1 \geq 1 - e^{-x}$ under $\PP_{\theta}$ that
  \begin{equation*}
    \max_{j=1,2,3}|\hat{m}_j - m_j| \leq 16\tau^3%
    \leq CL^{3/2}\sqrt{\frac{x}{n\gamma^*}} + C\max(\tau,\sqrt{L})^3\frac{x}{n\gamma^*}
  \end{equation*}
Next we address the case $x \leq n\sg$. It is seen that
  \begin{align*}
    \hat{m}_1 - m_1%
    &= \PP_n^{(2)}(\tilde{\psi}_2 \otimes \tilde{\psi}_2) - \EE_{\theta}(\tilde{\psi}_2 \otimes \tilde{\psi}_2) - \Big( \PP_n^{(1)}(\tilde{\psi}_2)^2 - \EE_{\theta}(\tilde{\psi}_2)^2\Big)
  \end{align*}
  ie.
  \begin{multline*}
    \hat{m}_1 - m_1%
    = \Big(\PP_n^{(2)}(\tilde{\psi}_2 \otimes \tilde{\psi}_2) - \EE_{\theta}(\tilde{\psi}_2 \otimes \tilde{\psi}_2) \Big)%
      - 2 \EE_{\theta}(\tilde{\psi}_2)\Big( \PP_n^{(1)}(\tilde{\psi}_2)\\ - \EE_{\theta}(\tilde{\psi}_2) \Big)%
      - \Big( \PP_n^{(1)}(\tilde{\psi}_2) - \EE_{\theta}(\tilde{\psi}_2) \Big)^2.
  \end{multline*}
Noting that $\EE_{\theta}(\abs{\tilde{\psi}_2}) \leq \EE_{\theta}(\tilde{\psi}_2^2)^{1/2} \leq \sqrt{L}\norm{\tilde{\psi}_2}_{L^2} = \sqrt{L}$ whenever $\theta \in \RegulClass$ by Lemma~\ref{lem:2}, we deduce
  \begin{equation*}
    |\hat{m}_1 - m_1|%
    \leq |Z_2| + 2\sqrt{L}|Z_1| + |Z_1|^2,
  \end{equation*}
where $Z_1 = \PP_n^{(1)}(\tilde{\psi}_2) - \EE_{\theta}(\tilde{\psi}_2)$ and $Z_2 = \PP_n^{(2)}(\tilde{\psi}_2 \otimes \tilde{\psi}_2) - \EE_{\theta}(\tilde{\psi}_2 \otimes \tilde{\psi}_2)$. The same reasoning yields, with , $Z_3 = \PP_n^{(3)}(\tilde{\psi}_2\otimes 1 \otimes \tilde{\psi}_2) - \EE_{\theta}(\tilde{\psi}_2\otimes 1 \otimes \tilde{\psi}_2)$,
  \begin{equation*}
    \abs{\hat{m}_2 - m_2}%
    \leq \abs{Z_3} + 2\sqrt{L}\abs{Z_1} + \abs{Z_1}^2.
  \end{equation*}
The decomposition for $\hat{m}_3-m_3$ is similar but slightly more involved. Since $ m_3 =  -\EE_{\theta}(\tilde{\psi}_2\otimes \tilde{\psi}_2 \otimes \tilde{\psi}_2) + \EE_{\theta}(\tilde{\psi}_2)^3%
+ \big(2 m_1 + m_2 \big) \EE_{\theta}(\tilde{\psi}_2)$, we deduce
  \begin{align*}
    \hat{m}_3 - m_3%
    &=-\Big(\PP_n^{(3)}(\tilde{\psi}_2\otimes \tilde{\psi}_2 \otimes \tilde{\psi}_2) - \EE_{\theta}(\tilde{\psi}_2\otimes \tilde{\psi}_2 \otimes \tilde{\psi}_2) \Big)\\
    &\quad%
      + \PP_n^{(1)}(\tilde{\psi}_2)^3 - \EE_{\theta}(\tilde{\psi}_2)^3\\
    &\quad%
      + \big[(2\hat{m}_1 + \hat{m}_2) - (2m_1 + m_2) \big] \EE_{\theta}(\tilde{\psi}_2)\\
    &\quad%
      + (2m_1 + m_2)\big( \PP_n^{(1)}(\tilde{\psi}_2) - \EE_{\theta}(\tilde{\psi}_2) \big)\\
    &\quad%
      + \big[(2\hat{m}_1 + \hat{m}_2) - (2m_1 + m_2) \big]\big( \PP_n^{(1)}(\tilde{\psi}_2) - \EE_{\theta}(\tilde{\psi}_2)\big).
  \end{align*}
  But $\PP_n^{(1)}(\tilde{\psi}_2)^3 - \EE_{\theta}(\tilde{\psi}_2)^3 = 3\EE_{\theta}(\tilde{\psi}_2)^2Z_1 + 3\EE_{\theta}(\tilde{\psi}_2)Z_1^2 + Z_1^3$, and thus recalling $\EE_{\theta}(|\tilde{\psi}_2|) \leq \sqrt{L}$ and $m_2\leq |m_1| \leq \frac{1}{4}\phi_3^2 \leq \tfrac{1}{2}L$ by Lemmas~\ref{lem:ub-phi3} and~\ref{lem:bounds-parameters}, writing $Z_4 = \PP_n^{(3)}(\tilde{\psi}_2\otimes \tilde{\psi}_2 \otimes \tilde{\psi}_2) - \EE_{\theta}(\tilde{\psi}_2\otimes \tilde{\psi}_2 \otimes \tilde{\psi}_2)$ we have
  \begin{align*}
    |\hat{m}_3 - m_3|%
      &\leq |Z_4| + 3L|Z_1| + 3\sqrt{L}|Z_1|^2 + |Z_1|^3
        + 2\sqrt{L}|\hat{m}_1 - m_1| + \sqrt{L}|\hat{m}_2 - m_2|\\
      &\quad + \frac{3L}{2} |Z_1|%
        + 2|\hat{m}_1 - m_1||Z_1|%
        + |\hat{m}_2 - m_2| |Z_1|.
  \end{align*}
  It follows (recall $L\geq 1$ necessarily)
  \begin{multline*}
    \max_{j=1,2,3}|\hat{m}_j - m_j|%
    \leq |Z_4| + \sqrt{L}|Z_3| + 2\sqrt{L}|Z_2|%
      + 10.5L|Z_1|\\%
      + 9\sqrt{L}Z_1^2 + 4 |Z_1|^3 + 2|Z_1Z_2| + |Z_1Z_3|.
  \end{multline*}
Feeding in bounds on the $Z_i$ from Lemma~\ref{lem:bounds-on-Z} below, we deduce with probability at least $1-4e^{-x}$ under $\PP_{\theta}$ that 
 \begin{align*}
	\max_{j=1,2,3}|\hat{m}_j - m_j|%
	&\leq%
	C\Bigg(L^{3/2}\sqrt{\frac{x}{n\gamma^*}} + \tau^3\frac{x}{n\gamma^*}\Bigg)%
	+ 3C\Bigg(L^{3/2}\sqrt{\frac{x}{n\gamma^*}} + L^{1/2}\tau^2\frac{x}{n\gamma^*} \Bigg)\\%
	&\quad%
	+ 10.5C \Bigg(L^{3/2}\sqrt{\frac{x}{n\gamma^*}} + L\tau\frac{x}{n\gamma^*} \Bigg)%
	+ 9C^2\sqrt{L}\Bigg(L^{1/2}\sqrt{\frac{x}{n\gamma^*}} + \tau\frac{x}{n\gamma^*} \Bigg)^2\\
	&\quad%
          + 4C^3\Bigg(L^{1/2}\sqrt{\frac{x}{n\gamma^*}} + \tau\frac{x}{n\gamma^*} \Bigg)^3\\%
   &\quad%
	+ 3C^2\Bigg(L^{1/2}\sqrt{\frac{x}{n\gamma^*}} + \tau\frac{x}{n\gamma^*}\Bigg)\Bigg( L\sqrt{\frac{x}{n\gamma^*}} + \tau^2\frac{x}{n\gamma^*} \Bigg).
\end{align*}
Grouping together the terms with same powers, still with probability at least $1-8e^{-x}$ under $\PP_{\theta}$
   \begin{align*}
  	\max_{j=1,2,3}\abs{\hat{m}_j - m_j}%
  	&\leq 14.5 C L^{3/2} \brackets[\Big]{{\frac{x}{n\gamma^*}}}^{1/2} 
  	+ C\Big(\tau^3 + 3L^{1/2}\tau^2 + 10.5L\tau + 12CL^{3/2} \Big)\frac{x}{n\gamma^*}\\
  	&\quad%
  	+ C^2\Big(18\tau L + 4CL^{3/2} + 3\tau^2\sqrt{L} + 3\tau L \Big)\brackets[\Big]{\frac{x}{n\gamma^*}}^{3/2}\\
  	&\quad%
  	+ C^2\Big( 9 \sqrt{L}\tau^2 + 12C\tau L + 3\tau^3 \Big) \brackets[\Big]{\frac{x}{n\gamma^*}}^2%
          + 12 C^3\tau^2\sqrt{L} \brackets[\Big]{\frac{x}{n\gamma^*}}^{5/2}\\%
     &\quad%
  	+ 4C^3\tau^3 \brackets[\Big]{\frac{x}{n\gamma^*}}^3.
  \end{align*}
  The conclusion follows since we are in the case where $x \leq n\gamma^*$, and because $L\geq 1$ and $\tau \geq 1$.
\end{proof}

\begin{lemma}
  \label{lem:bounds-on-Z}
  Assume $\theta\in \RegulClass$ and $n\gamma^*\geq 1/99$.  Write
  $Z_1 = \PP_n^{(1)}(\tilde{\psi}_2) - \EE_{\theta}(\tilde{\psi}_2)$,
  $Z_2 = \PP_n^{(2)}(\tilde{\psi}_2 \otimes \tilde{\psi}_2) -
  \EE_{\theta}(\tilde{\psi}_2 \otimes \tilde{\psi}_2)$,
  $Z_3 = \PP_n^{(3)}(\tilde{\psi}_2\otimes 1 \otimes \tilde{\psi}_2) -
  \EE_{\theta}(\tilde{\psi}_2\otimes 1 \otimes \tilde{\psi}_2)$, and
  $Z_4 = \PP_n^{(3)}(\tilde{\psi}_2\otimes \tilde{\psi}_2 \otimes
  \tilde{\psi}_2) - \EE_{\theta}(\tilde{\psi}_2\otimes \tilde{\psi}_2 \otimes
  \tilde{\psi}_2)$. Then
	\begin{gather*}
	\PP_{\theta}\Big( |Z_1| \geq C\sqrt{\frac{Lx}{n\gamma^*}} + C\tau \frac{x}{n\gamma^*} \Big)%
	\leq e^{-x}, \\
		\PP_{\theta}\Big( |Z_j| \geq CL\sqrt{\frac{x}{n\gamma^*}} + C\tau^2\frac{x}{n\gamma^*} \Big)%
	\leq e^{-x}, \quad j=2,3, \\
		\PP_{\theta}\Big( |Z_4| \geq CL^{3/2}\sqrt{\frac{x}{n\gamma^*}} + C\tau^3\frac{x}{n\gamma^*} \Big)%
	\leq e^{-x}.
\end{gather*}

\end{lemma}
\begin{proof}
  For $Z_4$, use Lemma~\ref{lem:paulin} together with the facts that
  $\norm{\tilde{\psi}_2\otimes\tilde{\psi}_2\otimes \tilde{\psi}_2}_{\infty} =
  \norm{\tilde{\psi}_2}_{\infty}^3 \leq \tau^3$ and that
  $\EE_{\theta}[(\tilde{\psi}_2 \otimes \tilde{\psi} _2 \otimes
  \tilde{\psi}_2)^2] \leq L^3\norm{\tilde{\psi}_2}_{L^2}^6 = L^3$ by
  Lemma~\ref{lem:2}.  The arguments are similar for $j=1,2,3$, though note for
  $j=3$ we use Remark~\ref{rem:lem2} rather than Lemma~\ref{lem:2} itself.
\end{proof}

\begin{lemma}

  \label{lem:pro:2a}
  Let $n\gamma^* \geq 1/99$. Then, there exists a universal constant $C > 0$ such that for all $\theta \in \RegulClass$
  \begin{equation*}
    \PP_{\theta}\Big(\max_{j=1,2}\Big|\frac{\hat{m}_j}{m_j} - 1\Big| \geq \frac{1}{2} \Big)%
    \leq 3\exp\Bigg(- \frac{Cn \gamma^* m_2^2}{L^2 + \max(\tau,\sqrt{L})^2m_2} \Bigg).
  \end{equation*}
\end{lemma}
 Note that $\frac{gm_2}{\max(1,g)} \leq m_2$ and that $L\geq 1$ necessarily, hence the the absolute value of the exponent in Lemma~\ref{lem:pro:2a} is larger than that in Lemma~\ref{lem:pro:2b}.
\begin{proof}
  We apply Proposition~\ref{pro:8alt} with $x\geq 0$ such that
  \begin{equation*}
    CL\sqrt{\frac{x}{n\gamma^*}} + C\max(\tau,\sqrt{L})^2\frac{x}{n\gamma^*} = \frac{m_2}{2},
  \end{equation*}
  i.e., 
  \begin{align*}
    \sqrt{\frac{x}{n\gamma^*}}%
    &= \frac{L}{2\max(\tau,\sqrt{L})^2}\Big(\sqrt{1 + \frac{2\max(\tau,\sqrt{L})^2m_2}{CL^2}} -1 \Big)\\
    &\geq \frac{L}{2} \frac{ m_2 / (CL^2)}{\sqrt{1 + \frac{2\max(\tau,\sqrt{L})^2m_2}{CL^2}}}.
  \end{align*}
Then, using that $0 \leq m_2\leq \abs{m_1}$, (Lemma~\ref{lem:basic-relations}), we have
  \begin{align*}
    \PP_{\theta}\Big(\max_{j=1,2}\Big|\frac{\hat{m}_j}{m_j} - 1\Big| \geq \frac{1}{2} \Big)%
    &\leq \PP_{\theta}\Big(\max_{j=1,2}\Big|\hat{m}_j - m_j\Big| \geq \frac{m_2}{2} \Big)\\
    &\leq 6\exp\Bigg(- \frac{n \gamma^* m_2^2}{2C^2L^2 + 2C\max(\tau,\sqrt{L})^2|m_2|} \Bigg)
  \end{align*}
  concluding the proof.
\end{proof}

\begin{lemma}
  \label{lem:pro:2b}
  Let $n\gamma^* \geq 1/99$. Then, there exists a universal constant $C > 0$ such that for all $\theta \in \RegulClass$
  \begin{equation*}
    \PP_{\theta}\Big(  \max_{j=1,2,3}|\hat{m}_j - m_j| \geq \frac{gm_2}{44\max(1,g)}
 \Big) \leq 4\exp\Bigg(- \frac{C n \gamma^* g^2m_2^2/\max(1,g)^2}{L^3 + \max(\tau,\sqrt{L})^3 gm_2/\max(1,g)} \Bigg).
  \end{equation*}
\end{lemma}
\begin{proof}
 By Proposition~\ref{pro:8}, applied with $x\geq 0$ such that
  \begin{align*}
    CL^{3/2}\sqrt{\frac{x}{n\gamma^*}}%
    + C\max(\tau,\sqrt{L})^3 \frac{x}{n\gamma^*}%
    = \frac{gm_2}{44\max(1,g)}
  \end{align*}
  ie,
  \begin{align*}
    \sqrt{\frac{x}{n\gamma^*}}%
    &= \frac{L^{3/2}}{2\max(\tau,\sqrt{L})^3}\Bigg(\sqrt{1 + \frac{4\max(\tau,\sqrt{L})^3gm_2}{44CL^3\max(1,g)}} - 1 \Bigg)\\
    &\geq%
      \frac{1}{44CL^{3/2}} \frac{ gm_2 / \max(1,g) }{\sqrt{1 + \frac{4\max(\tau,\sqrt{L})^3gm_2}{44CL^3\max(1,g)}}},
  \end{align*}
we obtain the result.
\end{proof}

\subsection{Proof of Theorem~\ref{thm:concent-pq}}
\label{sec:proof-prop-refpr}

In the whole proof, since $\tilde{\psi}_2$ is computed independently of the rest, we assume for convenience and without loss of generality that $\tilde{\psi}_2$ is non random  and we work implicitly conditional on $\tilde{\psi}_2$. It is assumed that $\tilde{\psi}_2$ satisfies the properties stated in the Theorem~\ref{thm:psitilde2}. Since the loss function is almost-surely bounded by $1$, the contribution of estimating $\tilde{\psi}_2$ to the risk is easily deduced from the Theorem~\ref{thm:psitilde2}.

Due to label switching, $\hat{\phi}_1$ may be either an estimator of $\phi_1$ or $-\phi_1$, depending on the value of  $\tilde{s} \coloneqq \sign(\Inner{\psi_2,\tilde{\psi}_2})$. In the proofs, rather than allow an arbitrary permutation, we define $p_\pm$ as an (unobserved) permutation of $(p,q)$ and we define $\hat{p}_+,\hat{p}_-$ such that $\hat{p}_\pm$ estimates $p_\pm$. To this end, define $p_\pm = \tfrac{1}{2} (1\mp \tilde{s}\phi_1) (1-\phi_2)$ (as in Lemma~\ref{lem:invert-param} already) and define $\hat{p}_\pm$ accordingly:
\begin{equation}\label{eqn:def:hatp+-}
	\hat{p}_\pm = \tfrac{1}{2} (1\mp \hat{\phi}_1)(1-\hat{\phi}_2).
	\end{equation}
It is noted in Lemma~\ref{lem:invert-param} that we may equivalently define
\begin{equation*}
	(p_+,p_-)%
	\coloneqq%
	\begin{cases}
		(p,q) &\mathrm{if}\ \tilde{s} > 0,\\
		(q,p) &\mathrm{if}\ \tilde{s} < 0.
	\end{cases}
\end{equation*}
Recall the definitions $g\coloneqq\phi_3\abs{\tilde{\mathcal{I}}} = m_2^{-1} \sqrt{4m_1^2m_2+m_3^2}$, $m_1 \coloneqq r(\phi)\tilde{\mathcal{I}}^2$, $m_2 \coloneqq r(\phi)\phi_2 \tilde{\mathcal{I}}^2$, and $m_3 \coloneqq r(\phi)\phi_1\phi_2\phi_3\tilde{\mathcal{I}}^3$. Also recall the event $\Omega_n$ defined in Proposition~\ref{pro:2}, and proved therein to satisfy $\sup_{\theta\in \SmoothClass \cap \RegulClass}\PP_{\theta}(\Omega_n^c)%
	\leq
	14\exp\Big(- \frac{C n \sg \delta^2\epsilon^4\zeta^6 }{L^3 + \max(\tau,\sqrt{L})^3 \delta \epsilon^2\zeta^3} \Big)$ for a constant $C>0$: 
\begin{equation*}
  \Omega_n \coloneqq \Set*{\max_{j=1,2}\Big|\frac{\hat{m}_j}{m_j} - 1\Big| \leq \frac{1}{2},\ \max_{j=1,2,3}|\hat{m}_j - m_j| \leq \frac{gm_2}{44\max(1,g)}  }.
\end{equation*}
Its definition is according to the needs of the proof of Theorem \ref{thm:1} which are more stringent than those of the current result. In particular, note that on $\Omega_n$ we have $\max_{j=1,2,3}|\hat{m}_j - m_j| \leq
\frac{|\tilde{\mathcal{I}}|^3r(\phi)\phi_2\phi_3 }{20 \max(|\phi_1|,
	(1-\phi_1^2)\phi_3|\tilde{\mathcal{I}}|)}$, as a consequence of the fact that $\abs{\phi_1}\leq 1$; this latter bound is what we will use for the current theorem.

We decompose
\begin{align*}
  \EE_{\theta}\Big(|\hat{p}_{\pm} - p_{\pm}|^2 \Big)%
  &= \EE_{\theta}\Big(|\hat{p}_{\pm} - p_{\pm}|^2\1_{\Omega_n^c} \Big)%
    + \EE_{\theta}\Big(|\hat{p}_{\pm} - p_{\pm}|^2\1_{\Omega_n} \Big)\\
  &\leq \PP_{\theta}(\Omega_n^c) +  \EE_{\theta}\Big(|\hat{p}_{\pm} - p_{\pm}|^2\1_{\Omega_n} \Big),
\end{align*}
We have
\begin{align*}
  \hat{p}_{\pm} - p_{\pm}%
  &= -\frac{1}{2}(\hat{\phi}_2 - \phi_2)%
  \mp \frac{1}{2}(\hat{\phi}_1 - \tilde{s}\phi_1)%
  \pm \frac{\tilde{s}\phi_1}{2}(\hat{\phi}_2 - \phi_2)%
  \mp \frac{\hat{\phi}_2}{2}(\hat{\phi}_1 - \tilde{s}\phi_1),
\end{align*}
hence, using that $\abs{\hat{\phi}_2} \leq 1$ and $\abs{\phi_1}\leq 1$,
\begin{equation*}
  |\hat{p}_{\pm} - p_{\pm}|%
  \leq |\hat{\phi}_1 - \phi_1| + |\hat{\phi}_2 - \phi_2|.
\end{equation*}
Using Lemmas~\ref{lem:10AA} and~\ref{lem:10BB} below and Proposition~\ref{pro:expec-maxm}, we get for a constant $K$
\begin{align*}
  \EE_{\theta}\Big(|\hat{p}_{\pm} - p_{\pm}|^2\1_{\Omega_n} \Big)%
  &\leq 2\EE_{\theta}\Big(|\hat{\phi}_1 - \tilde{s}\phi_1|^2\1_{\Omega_n} \Big)%
    +  2\EE_{\theta}\Big(|\hat{\phi}_2 - \phi_2|^2\1_{\Omega_n} \Big)\\
  &\leq 2\Big(\frac{53^2\max(1,g^2)}{\phi_2^4\phi_3^6|\tilde{\mathcal{I}}|^6}%
    + \frac{16}{m_1^2}
    \Big)\EE_{\theta}\Big(\max_{j=1,2,3}|\hat{m}_j - m_j|^2 \Big)\\
  &\leq 2K\Big(\frac{53^2\max(1,g^2)}{\phi_2^4\phi_3^6\abs{\tilde{\mathcal{I}}}^6}%
    + \frac{16}{m_1^2}
    \Big)\Big(\frac{L^3}{n\gamma^*}%
      + \frac{\max(\tau,\sqrt{L})^6}{(n\gamma^*)^2}\Big).
\end{align*}
Therefore, there is a universal constant $B \geq 1$ such that
\begin{align*}
  &\sup_{\theta\in \SmoothClass \cap \RegulClass}\EE_{\theta}\Big(|\hat{p}_{\pm} - p_{\pm}|^2\1_{\Omega_n} \Big)\\
  &\qquad\qquad\leq%
    \frac{BL^3\max(\delta^2,\epsilon^2\zeta^2)}{\delta^2\epsilon^4\zeta^6}\frac{1}{n\sg}%
  + \frac{B\max(\tau,L)^6\max(\delta^2,\epsilon^2\zeta^2)}{\delta^2\epsilon^4\zeta^6}\frac{1}{(n\sg)^2}\\
  &\qquad\qquad\leq%
    \frac{2BL^3\max(\delta^2,\epsilon^2\zeta^2)}{\delta^2\epsilon^4\zeta^6}\frac{1}{n\sg},
\end{align*}
since $L\geq 1$ and  $\sup_{\theta\in \SmoothClass \cap \RegulClass}\EE_{\theta}\Big(|\hat{p}_{\pm} - p_{\pm}|^2\1_{\Omega_n} \Big) \leq 1$.
Lemmas~\ref{lem:10AA} and~\ref{lem:10BB} therefore conclude the proof.

\begin{lemma}
  \label{lem:10AA}
  Suppose
  \begin{equation*}
    \max_{j=1,2}\big|\frac{\hat{m}_j}{m_j} - 1\big| \leq \frac{1}{2},\quad \mathrm{and},\quad%
    \max_{j=1,2,3}|\hat{m}_j - m_j| \leq
  \frac{|\tilde{\mathcal{I}}|^3r(\phi)\phi_2\phi_3 }{20 \max(|\phi_1|,
    (1-\phi_1^2)\phi_3|\tilde{\mathcal{I}}|)}.
\end{equation*}
Then,
  \begin{equation*}
    |\hat{\phi}_1 - \tilde{s}\phi_1|%
    \leq%
    \frac{53\max(1,\phi_3|\tilde{\mathcal{I}}|)}{\phi_2^2\phi_3^3|\tilde{\mathcal{I}}|^3}
    \max_{j=1,2,3}|\hat{m}_j - m_j|.
  \end{equation*}
\end{lemma}
\begin{proof}
  We use the notations $\Delta_1 = \hat{m}_1 - m_1$, $\Delta_2 = (\hat{m}_2)_+ - m_2$, and $\Delta_3 = \hat{m}_3 - m_3$.
  Then, we define 
  \begin{align*} \hat{v} &\coloneqq 4\hat{m}_1^2(\hat{m}_2)_+ + \hat{m}_3^2, \\
  	v &\coloneqq 4m_1^2m_2 + m_3^2, \\
  	h &\coloneqq \hat{v} - v,\\ 
  	\xi &\coloneqq 8m_1m_2\Delta_1 + 4m_1^2\Delta_2 + 8m_1\Delta_1\Delta_2 + 4m_2\Delta_1^2  + 4\Delta_1^2\Delta_2,\\
  	\eta &\coloneqq 2m_3\Delta_3 + \Delta_3^2.
  \end{align*}
  Lemma~\ref{lem:controlofv} below tells us that  $|h|\leq 10\max(\abs{\phi_1},(1-\phi_1^2)\phi_3\abs{\tilde{\mathcal{I}}})\abs{r(\phi)\phi_2\phi_3\tilde{\mathcal{I}}^3}\max_{j=1,2,3}\abs{\Delta_j}$.   
  Furthermore, it is seen that  $\sqrt{v} = |\tilde{\mathcal{I}}|^3 r(\phi) \phi_2\phi_3 = |\tilde{\mathcal{I}}| m_2\phi_3$ (see Lemma~\ref{lem:basic-relations}) and then under the conditions of this lemma, we have $\abs{h} \leq v/2$ and $\abs{\Delta_3} \leq (1/2)\abs{m_3}=(1/2)\phi_1\phi_3\tilde{\mathcal{I}}\abs{m_2}\leq \sqrt{v}/2$. 
   Consequently, $1-\frac{\Delta_3^2}{(\sqrt{v+h}+\sqrt{v})^2}\geq 3/4$ and $(v+h)^{1/2}[(v+h)^{1/2}+v^{1/2}]\geq (1+\sqrt{2})v/2\geq v$ and hence using Lemma~\ref{lem:invformulam}
 \begin{align*}
	\abs{\hat{\phi}_1 - \tilde{s}\phi_1}%
	&\leq \frac{\abs{\phi_1\xi}}{v}%
	+ \frac{4}{3v}\sqbrackets[\Big]{2\abs{\Delta_3}(1-\phi_1^2)v^{1/2}+\abs{\phi_1}\Delta_3^2\abs{\xi}v^{-1}+\abs{\Delta_3\xi}v^{-1/2}} \\
	&\leq \frac{28}{v}m_1^2\max_{j=1,2}\abs{\Delta_j}  + \frac{8}{3}(1-\phi_1^2)v^{-1/2}\abs{\Delta_3} + \frac{4}{3}\abs{\xi}[1/2+\abs{\phi_1}/4] \\
	&\leq 28 \frac{m_1^2}{v} \max_{j=1,2}\abs{\Delta_j} + \frac{8}{3} (1-\phi_1^2)v^{-1/2} \abs{\Delta_3} + 56 \frac{m_1^2}{v} \max_{j=1,2}\abs{\Delta_j} \\
	& \leq 42 (\phi_2^2\phi_3^2 \tilde{\mathcal{I}}^2)^{-1} \max_{j=1,2}\abs{\Delta_j} + \frac{32}{3} (\phi_2^2\phi_3^3\tilde{\mathcal{I}}^3)^{-1} \abs{\Delta_3} \\
	 & \leq 53 (\phi_2^2\phi_3^3\tilde{\mathcal{I}}^3)^{-1}\max(\phi_3 \tilde{\mathcal{I}},1)\max_{j=1,2,3}\abs{\Delta_j}.
%
\end{align*}
  The conclusion follows since $x \mapsto (x)_+$ is $1$-Lipschitz and thus $|\Delta_2| = |(\hat{m}_2)_+ - m_2| = |(\hat{m}_2)_+ - (m_2)_+| \leq |\hat{m}_2 - m_2|$, so that $\max_{j=1,2,3}|\Delta_j| \leq \max_{j=1,2,3}|\hat{m}_j - m_j|$.
\end{proof}

\begin{lemma}\label{lem:controlofv}
	Define  $v = 4m_1^2m_2 + m_3^2$, $\hat{v} =4\hat{m}_1^2(\hat{m}_2)_+ + \hat{m}_3^2$.
	Then \[\abs{\hat{v}-v}\leq 10\max(\abs{\phi_1},(1-\phi_1^2)\phi_3\abs{\tilde{\mathcal{I}}})\abs{r(\phi)\phi_2\phi_3\tilde{\mathcal{I}}^3}\max_{j=1,2,3}\abs{\Delta_j},\]
	where $\Delta_j=\hat{m}_j-m_j,$ $j=1,3$ and $\Delta_2=(\hat{m}_2)_+-m_2$.
\end{lemma}
\begin{proof}
	  Define 
	\begin{align*}
		h &\coloneqq \hat{v} - v,\\ 
		\xi &\coloneqq 8m_1m_2\Delta_1 + 4m_1^2\Delta_2 + 8m_1\Delta_1\Delta_2 + 4m_2\Delta_1^2  + 4\Delta_1^2\Delta_2,\\
		\eta &\coloneqq 2m_3\Delta_3 + \Delta_3^2.
	\end{align*} Note that $h = \xi + \eta$. By Lemma~\ref{lem:invformulam} and mimicking the proof of \citep[Proposition 3]{AGNparamhmm}, it is found that
	\begin{align*}
		\hat{\phi}_1 - \tilde{s}\phi_1%
		&=%
		\frac%
		{\phi_1\xi + \frac{-2\Delta_3(1-\phi_1^2)v^{1/2} + \frac{\phi_1\Delta_3^2\xi}{((v+h)^{1/2} + v^{1/2} )^2}  - \frac{\Delta_3\xi}{(v+h)^{1/2} + v^{1/2}} }{1 - \Delta_3^2/((v+h)^{1/2} + v^{1/2})^2 } }%
		{ (v+h)^{1/2}[ (v+h)^{1/2} + v^{1/2} ]  }
	\end{align*}
	We note that the assumptions of the lemma imply that $|\Delta_j| \leq |m_j|$ for $j=1,2,3$; recall also that $0\leq m_2=m_1\leq \abs{m_1}$. Thus,
	\begin{align*}
		|\xi|%
		&= \Big|8m_1m_2\Delta_1 + 4m_1^2\Delta_2 + 8m_1\Delta_1\Delta_{2} + 4m_2\Delta_1^2 + 4\Delta_1^2\Delta_2\Big|\\
		&\leq 28m_1^2\max_{j=1,2}|\Delta_j|.
	\end{align*}
	Since $|\eta| \leq 2|m_3\Delta_3| + \Delta_3^2 \leq 3|m_3\Delta_3|$, it also follows that (recall $m_1=r(\phi)\tilde{\mathcal{I}}^2$, $m_3= \phi_1\phi_2\phi_3r(\phi)\tilde{\mathcal{I}}^3$, $r(\phi)=(1/4)(1-\phi_1^2)\phi_2\phi_3^2$)
	\begin{align*}
		|h|
		&\leq \brackets[\big]{28m_1^2 + 3|m_3|} \max_{j=1,2,3}|\Delta_j|\\
		&= |r(\phi)\phi_2\phi_3 \tilde{\mathcal{I}}^3|\Big(3|\phi_1| +  \frac{28|r(\phi) \tilde{\mathcal{I}}|}{|\phi_2\phi_3|} \Big) \max_{j=1,2,3}|\Delta_j|\\
		&= |r(\phi)\phi_2\phi_3 \tilde{\mathcal{I}}^3|\Big(3|\phi_1| + 7(1-\phi_1^2)\phi_3|\tilde{\mathcal{I}}| \Big) \max_{j=1,2,3}|\Delta_j|\\
		&\leq 10\max\big(|\phi_1|, (1-\phi_1^2)\phi_3|\tilde{\mathcal{I}}| \big)|r(\phi)\phi_2\phi_3 \tilde{\mathcal{I}}^3| \max_{j=1,2,3}|\Delta_j|.
	\end{align*}
        This concludes the proof.
\end{proof}
\begin{lemma}
  \label{lem:10BB}
  The following bounds holds true.
  \begin{equation*}
    |\hat{\phi}_2 - \phi_2|%
    \leq%
    2\min\Bigg(1, \frac{2 \max_{j=1,2}|\hat{m}_j - m_j|}{|m_1|} \Bigg).
  \end{equation*}
\end{lemma}
\begin{proof}
  We let $\Delta_1 \coloneqq \hat{m}_1 - m_1$ and $\Delta_2 \coloneqq \hat{m}_2 - m_2$. We also let $f(x) \coloneqq \max(-1,\min(x,1))$. It is easily seen that $|f(x) - f(y)| \leq \min(2, |x-y|)$. Suppose first that $|\Delta_1| > |m_1|/2$. Then, $|\hat{\phi}_2 - \phi_2| \leq 2 \leq \min(2, \frac{4|\Delta_1|}{|m_1|})$. On the other hand, if $|\Delta_1| \leq |m_1|/2$, then, recalling that $m_2\leq \abs{m_1}$ we have from Lemma~\ref{lem:invformulam}
  \begin{align*}
    |\hat{\phi}_2 - \phi_2|%
    &= |f(\hat{m}_2 / \hat{m}_1) - f(m_2/m_1)|\\
    &\leq \min\Big(2,\, \Big|\frac{m_2 + \Delta_2}{m_1 + \Delta_1} - \frac{m_2}{m_1}\Big| \Big)\\
    &= \min\Big(2,\, \Big|\frac{m_1\Delta_2 - m_2\Delta_1}{m_1(m_1 + \Delta_1)}\Big| \Big)\\
    &\leq \min\Big(2,\, \frac{2|\Delta_1| + 2|\Delta_2|}{|m_1|} \Big).
  \end{align*}
  The conclusion follows since $x \mapsto (x)_+$ is $1$-Lipschitz and thus $|\Delta_2| = |(\hat{m}_2)_+ - m_2| = |(\hat{m}_2)_+ - (m_2)_+| \leq |\hat{m}_2 - m_2|$. 
\end{proof}

\subsection{Proof of Theorem~\ref{thm:1}}
\label{sec:proof-thm:1}

In the whole proof, since $\tilde{\psi}_2$ is computed independently of the rest, we assume for convenience and without loss of generality that $\tilde{\psi}_2$ is non random  and we work implicitly conditional on $\tilde{\psi}_2$. It is assumed that $\tilde{\psi}_2$ satisfies the properties stated in the Theorem~\ref{thm:psitilde2}. The loss function is almost-surely bounded by $\check{T}^2$ so the contribution of estimating $\tilde{\psi}_2$ to the risk is easily deduced from the Theorem~\ref{thm:psitilde2}.

As in Appendix~\ref{sec:proof-prop-refpr}, rather than allow an arbitrary permutation to account for the label-switching, we give a specific (unobserved) permutation.  
We recall the definitions of the estimators of $f_0$ and $f_1$ from Section~\ref{sec:nonparametric-part:1}, here writing as $\check{f}_\pm$ to align with notation used in Lemma~\ref{lem:invert-param}. We define (see also Lemma~\ref{lem:invformulam})
\begin{align*}
	g \coloneqq \phi_3|\tilde{\mathcal{I}}|%
	= \frac{\sqrt{4m_1^2m_2 + m_3^2}}{m_2}, \qquad 	G \coloneqq \frac{m_1 \psi_2}{\tilde{\mathcal{I}}},\qquad
	f_{\pm} \coloneqq \psi_1 \pm \frac{g(1 \mp \tilde{s}\phi_1)}{2 m_1}G,
\end{align*}
and
\begin{gather*}
\hat{g} \coloneqq \frac{\sqrt{4\hat{m}_1^2(\hat{m}_2)_+ + \hat{m}_3^2}}{\hat{m}_2}\1_{\{ \hat{m}_2 > 0\} }, \qquad%
	\hat{G}^{\Phi_{Jk}} \coloneqq%
	\PP_n^{(2)}(\tilde{\psi}_2 \otimes \Phi_{Jk} ) - \PP_n^{(1)}(\tilde{\psi}_2)\PP_n^{(1)}(\Phi_{Jk}),\\
	\HatFather_{\pm}%
	\coloneqq \PP_n^{(1)}(\Phi_{Jk}) \pm \frac{\hat{g}(1 \mp \hat{\phi}_1)}{2\hat{m}_1}\1_{ \{ \hat{m}_1 \ne 0 \} } \hat{G}^{\Phi_{Jk}}.
\end{gather*}
Then, defining $\HatMother_{\pm}$ and $\hat{G}^{\Psi_{jk}}$ correspondingly we set
\begin{align*}
  \hat{f}_{\pm}%
  &\coloneqq \sum_{k=0}^{2^J-1} \HatFather_{\pm} \Phi_{Jk}%
    + \sum_{j=J}^{J_n-1}\sum_{k=0}^{2^j-1}\HatMother_{\pm}\Psi_{jk}%
    + \sum_{j=J_n}^{\tilde{\jmath}_n}\sum_{\ell}\Bigg(\sum_{k\in \mathfrak{B}_{j\ell}}\HatMother_{\pm} \Psi_{jk} \Bigg)\1_{ \{ \norm{\hat{f}_{\pm}^{\mathfrak{B}_{j\ell}}} > \Gamma \hat{S}_n \} },\\
  \check{f}_{\pm} &\coloneqq \max\big(0,\, \min\big(\check{T},\, \hat{f}_{\pm}\big)\big),
\end{align*}
where $J_n \coloneqq \inf\Set{j\geq J \given 2^j \geq \log(n)}$, $N= 2^{J_n}$, and $\mathfrak{B}_{j\ell} \coloneqq \Set{k \given (\ell-1)N \leq k \leq \ell N - 1}$ and $\tilde{\jmath}_n$ is the largest integer such that $2^{\tilde{\jmath}_n} \leq \frac{n}{\log(n)\tau^2}$ (recall we assume that  $\tilde{\jmath}_n$ is larger than $J_n$)
and 
where $\norm{\hat{f}_{\pm}^{\mathfrak{B}_{j\ell}}}^2 \coloneqq \sum_{k\in \mathfrak{B}_{j\ell}}( \HatMother_{\pm} )^2$, $\Gamma >0$ is a tuning parameter, and
\begin{equation*}
	\hat{S}_n
	\coloneqq \sqrt{\frac{\log(n)}{n}}\max\Big(1,\, \frac{\hat{g}}{|\hat{m}_1|}\Big)\1_{ \{\hat{m}_1 \ne 0 \} }.
\end{equation*}
Recall the event  $\Omega_n = \Set*{\max_{j=1,2}\Big|\frac{\hat{m}_j}{m_j} - 1\Big| \leq \frac{1}{2},\ \max_{j=1,2,3}|\hat{m}_j - m_j| \leq \frac{gm_2}{44\max(1,g)}  }$ defined in Proposition~\ref{pro:2} which by the proposition satisfies for a universal constant $C > 0$
\begin{align*}
  \sup_{\theta\in \SmoothClass \cap \RegulClass}\PP_{\theta}(\Omega_n^c)%
  \leq
   7\exp\Bigg(- \frac{C n \sg \delta^2\epsilon^4\zeta^6 }{L^3 + \max(\tau,\sqrt{L})^3 \delta \epsilon^2\zeta^3} \Bigg).
\end{align*}
Decompose
\begin{align*}
  \EE_{\theta}\Big(\norm{\check{f}_{\pm} - f_{\pm}}_{L^2}^2 \Big)%
  &= \EE_{\theta}\Big(\norm{\check{f}_{\pm} - f_{\pm}}_{L^2}^2\1_{\Omega_n^c} \Big)%
    + \EE_{\theta}\Big(\norm{\check{f}_{\pm} - f_{\pm}}_{L^2}^2\1_{\Omega_n} \Big)\\
  &\leq \check{T}^2 \PP_{\theta}(\Omega_n^c) + \EE_{\theta}\Big(\norm{\hat{f}_{\pm} - f_{\pm}}_{L^2}^2\1_{\Omega_n} \Big)
\end{align*}
where the last line follows because $0\leq f_{\pm},\check{f}_{\pm}\leq \check{T}$ since $\check{T}\geq L $ by assumption, and because $|\check{f}_{\pm} - f_{\pm}| \leq |\hat{f}_{\pm} - f_{\pm}|$ pointwise. The first term is included in the theorem and it remains to bound the second term. We decompose as follows (recall that $\tilde{\jmath}_n > J_n$ by assumption and the sum over $\ell$ is the sum over blocks from $\ell=0$ to $\ell = 2^j/N - 1$)
\begin{align*}
  \EE_{\theta}\Big( \norm{\hat{f}_{\pm} - f_{\pm}}_{L^2}^2\1_{\Omega_n}\Big) %
  &= \EE_{\theta}\Big( \norm{\hat{f}_{\pm}^{J_n} - f_{\pm}^{J_n}}_{L^2}^2 \1_{\Omega_n}\Big)\\
  &\quad%
    + \EE_{\theta}\Bigg( \sum_{j=J_n}^{\tilde{\jmath}_n}\sum_{\ell}\norm{f_{\pm}^{\Block_{j\ell}}}^2\1_{\{ \norm{\hat{f}_{\pm}^{\Block_{j\ell}}} \leq \Gamma \hat{S}_n\}  }\1_{\{ \norm{f_{\pm}^{\Block_{j\ell}}} \leq 2\Gamma \hat{S}_n \} }\1_{\Omega_n}\Bigg) \\
   &\quad%
    + \EE_{\theta}\Bigg( \sum_{j=J_n}^{\tilde{\jmath}_n}\sum_{\ell}\norm{f_{\pm}^{\Block_{j\ell}}}^2\1_{\{ \norm{\hat{f}_{\pm}^{\Block_{j\ell}}} \leq \Gamma \hat{S}_n\}  }\1_{\{ \norm{f_{\pm}^{\Block_{j\ell}}} > 2\Gamma \hat{S}_n \} }\1_{\Omega_n}\Bigg) \\
  &\quad%
    + \EE_{\theta}\Bigg( \sum_{j=J_n}^{\tilde{\jmath}_n}\sum_{\ell}\norm{\hat{f}_{\pm}^{\Block_{j\ell}} - f_{\pm}^{\Block_{j\ell}}}^2\1_{\{ \norm{\hat{f}_{\pm}^{\Block_{j\ell}}} > \Gamma \hat{S}_n\}  }\1_{ \{ \norm{f_{\pm}^{\Block_{j\ell}}} \leq \frac{1}{2}\Gamma \hat{S}_n \} \} } \1_{\Omega_n}\Bigg)\\
  &\quad%
    + \EE_{\theta}\Bigg( \sum_{j=J_n}^{\tilde{\jmath}_n}\sum_{\ell}\norm{\hat{f}_{\pm}^{\Block_{j\ell}} - f_{\pm}^{\Block_{j\ell}}}^2\1_{\{ \norm{\hat{f}_{\pm}^{\Block_{j\ell}}} > \Gamma \hat{S}_n\}  }\1_{ \{ \norm{f_{\pm}^{\Block_{j\ell}}} > \frac{1}{2}\Gamma \hat{S}_n \} } \1_{\Omega_n}\Bigg)\\
  &\quad%
    + \PP_{\theta}(\Omega_n)\sum_{j > \tilde{\jmath}_n}\sum_{k=0}^{2^j-1}|f_{\pm}^{\Psi_{jk}}|^2
\end{align*}
where we have used the convention that for any function $f$ the notation $f^{J_n}$ stands for the projection $f_{\pm}^{J_n} \coloneqq \sum_{k=0}^{2^J-1}f_{\pm}^{\Phi_{Jk}}\Phi_{Jk} + \sum_{j=J}^{J_n-1}\sum_{k=0}^{2^j-1}f_{\pm}^{\Psi_{jk}}\Psi_{jk}$. Recall that $f^{\Block_{j\ell}}$ denotes the vector of coefficients $(\Inner{f,\Psi_{jk}}\,:\, (j,k) \in \Block_{j\ell})$ and $\|\cdot\|$ the euclidean norm. We call the terms in the previous decomposition $R_1(\theta),\dots,R_6(\theta)$, respectively. To ease the notations in the proof, we also introduce the quantities
\begin{equation}\label{eqn:def:omegapm}
  \hat{\omega}_{\pm}%
  \coloneqq  \pm \frac{\hat{g}(1\mp\hat{\phi}_1)}{\hat{m}_1}\1_{\{\hat{m}_1 \ne 0 \} },\qquad%
  \omega_{\pm}%
  \coloneqq \pm \frac{g(1 \mp \tilde{s}\phi_1)}{m_1}
\end{equation}
and
\begin{equation}\label{eqn:def:Sn}
  S_n%
  \coloneqq \sqrt{\frac{\log(n)}{n}}\max\Big(1,\,\frac{g}{|m_1|}\Big).
\end{equation}

In the next subsections we prove the following bounds, uniformly over $\theta \in \SmoothClass \cap \RegulClass$:
\begin{align*}
  R_1(\theta)%
  &\leq
    \frac{B L^2}{\delta^2\epsilon^2\zeta^2}\frac{\log(n)}{n\sg}%
    + \frac{BL^3}{\delta^2\epsilon^4\zeta^4}\frac{1}{n\sg}%
    +\frac{B\max(\tau,\sqrt{L})^6}{\delta^2\epsilon^4\zeta^4}\frac{1}{(n\sg)^2}\\
  R_2(\theta)%
  &\leq \frac{BR^2}{\min(1,s_{\pm})}\Big(\frac{\Gamma^2}{R^2\delta^2\epsilon^2\zeta^2 n} \Big)^{2s_{\pm}/(2s_{\pm} + 1)}%
    + \frac{BR^2}{\min(1,s_{\pm})}\Big(\frac{\tau^2\log(n)}{n} \Big)^{2s_{\pm}} \\
  R_3(\theta)%
  &\leq \frac{BL^3}{\delta^2\epsilon^4\zeta^4}\frac{1}{n\sg}%
    + \frac{B\max(\tau,\sqrt{L})^6}{\delta^2\epsilon^4\zeta^4}\frac{1}{(n\sg)^2}, \\
  R_4(\theta)%
  &\leq \frac{BL^3}{\delta^2\epsilon^4\zeta^4}\frac{1}{n\sg}%
    +   \frac{B\max(\tau,\sqrt{L})^6}{\delta^2\epsilon^4\zeta^4}\frac{1}{(n\sg)^2},\\
  R_5(\theta)%
  &\leq%
    \frac{BL^2}{\Gamma^2\sg} \Bigg( \frac{R^2}{\min(1,s_{\pm})}\Big(\frac{\Gamma^2}{R^2\delta^2\epsilon^2\zeta^2 n} \Big)^{2s_{\pm}/(2s_{\pm} + 1)}%
    + \frac{R^2}{\min(1,s_{\pm})}\Big(\frac{\tau^2\log(n)}{n} \Big)^{2s_{\pm}} \Bigg)\\%
  &\quad+ \frac{BL^3}{\delta^2\epsilon^4\zeta^6}\frac{1}{n\sg}%
  + \frac{B\max(\tau,\sqrt{L})^6}{\delta^2\epsilon^4\zeta^4}\frac{1}{(n\sg)^2},\\
  R_6(\theta)%
  &\leq \frac{BR^2}{\min(1,s_{\pm})}\Big( \frac{\tau^2 \log(n)}{n} \Big)^{2s_{\pm}}.        
\end{align*}
Combining will yield the theorem.

\subsubsection{Control of $R_1$}
\label{sec:control-r_1}


Using Lemma~\ref{lem:pro:7} to control $\norm{\hat{f}_{\pm}^{J_n} - f_{\pm}^{J_n}}_{L^2}$ and Proposition~\ref{pro:9} in Section \ref{sec:auxiliary-results} to control $\abs{\hat{\omega}_\pm-\omega_\pm}$, the bounds $(a+b+c)^2\leq 3(a^2+b^2+c^2)$ and  $\norm{G^{J_n}}_{L^2} =  \abs{m_1} \norm{\psi_2^{J_n}}_{L^2}/\abs{\tilde{\mathcal{I}}} \leq (8/7)\abs{m_1}$ allow us to deduce
\begin{align*}
  R_1(\theta)&\coloneqq \EE_{\theta}\Big( \norm{\hat{f}_{\pm}^{J_n} - f_{\pm}^{J_n}}_{L^2}^2 \1_{\Omega_n}\Big)\\%
  &\leq
    3\EE_{\theta}\Big(\norm{\hat{\psi}_1^{J_n} - \psi_1^{J_n}}_{L^2}^2 \Big)%
    + \frac{12g^2}{m_1^2}\EE_{\theta}\Big(\norm{\hat{G}^{J_n} - G^{J_n}}_{L^2}^2 \Big)%
    + \frac{3 \norm{G^{J_n}}_{L^2}^2}{4} \EE_{\theta}\Big(| \hat{\omega}_{\pm} - \omega_{\pm} |^2\1_{\Omega_n}\Big).
\end{align*}
ie.
\begin{multline}
  \label{eqn:R1bound}
  R_1(\theta) \leq 3\EE_{\theta}\Big(\norm{\hat{\psi}_1^{J_n} - \psi_1^{J_n}}_{L^2}^2 \Big)%
    + \frac{12g^2}{m_1^2}\EE_{\theta}\Big(\norm{\hat{G}^{J_n} - G^{J_n}}_{L^2}^2 \Big)\\%
    + \frac{3\cdot 8^2\cdot 83^2 \max(1,\phi_3^2\tilde{\mathcal{I}}^2)}{4\cdot 7^2 m_2^2}  \EE_{\theta}\Big(\max_{j=1,2,3}|\hat{m}_j - m_j|^2 \Big).
\end{multline}
Proposition~\ref{pro:7a} tells us that $$\PP_{\theta}\Bigg(\norm{\hat{\psi}_1^{J_n} - \psi_1^{J_n}}_{L^2} \geq C\sqrt{\frac{Lx}{n\gamma^*}} + C2^{J_n/2}\frac{x}{n\gamma^*} \Bigg)%
	\leq 24^{2^{J_n}}e^{-x},$$
        hence, using that $2^{J_n}\leq 2\log(n)$ for $n \geq 2$,  for a sufficient large constant $\alpha > 0$ we may apply Lemma~\ref{lem:9} with $a = C\sqrt{2\log(n)}/\gamma^*$, 
$b=C\sqrt{L/\gamma^*}$, $c = 24^{2\log(n)}$ 
and $d^2 = \alpha C^2L\log(n)/(n\gamma^*)$
\begin{align*}
  &\EE_{\theta}\Big(\norm{\hat{\psi}_1^{J_n} - \psi_1^{J_n}}_{L^2}^2 \Big)\\%
  &\qquad\leq \alpha C^2L\frac{\log(n)}{n\gamma^*} + \EE_{\theta}\Big(\norm{\hat{\psi}_1^{J_n} - \psi_1^{J_n}}_{L^2}^2\1_{ \{ \norm{\hat{\psi}_1^{J_n} - \psi_1^{J_n}}_{L^2}^2 > \alpha C^2L\log(n)/(n\gamma^*)  \} } \Big)\\
  &\qquad\leq  \alpha C^2L \frac{\log(n)}{n\gamma^*}%
    + c\Big(d^2 +  \frac{5b^2}{2n} + \frac{7a^2}{2n^2} \Big)e^{-nd^2/(2b^2 + 8ad)}\\
  &\qquad\leq \alpha C^2L \frac{\log(n)}{n\gamma^*}%
    + C^2 24^{2\log(n)}\Big(\frac{\alpha L \log(n)}{n\sg} + \frac{5L}{2n\sg} + \frac{14\log(n)}{2(n\sg)^2} \Big)e^{-nd^2/(2b^2 + 8ad)}\\
  &\qquad\leq \alpha C^2L \frac{\log(n)}{n\gamma^*}%
    + C^2 24^{2\log(n)}\Big(\alpha L + \frac{5L}{2} + 7 \Big)\log(n)e^{-nd^2/(2b^2 + 8ad)}
\end{align*}
where the last line follows because $n\gamma^* \geq \tau^3\geq 1$. 
Let us now study the argument of the exponential in the last display. If $2b^2 \geq 8ad$, then
\begin{align*}
  \frac{nd^2}{2b^2 + 8ad}%
  &\geq \frac{nd^2}{4b^2}%
    = \frac{\alpha}{4}\log(n),
\end{align*}
while if $2b^2 < 8ad$, then
\begin{align*}
  \frac{nd^2}{2b^2 + 8ad}%
  &\geq \frac{nd^2}{16ad}%
    = \frac{n\gamma^*\sqrt{\alpha C^2 L \log(n)/(n\gamma^*)} }{16 C\sqrt{2\log n}}%
    \geq \frac{\sqrt{\alpha L}}{16\sqrt{2}}\sqrt{n\gamma^*}%
    \geq \frac{\sqrt{\alpha}}{16\sqrt{2}}\log(n)
\end{align*}
because by assumption $n\gamma^* \geq \frac{\log(n)^2}{L}$. Hence, since $L \leq n$ and $\sg \leq 1$ it is possible to choose $\alpha > 0$ universally such that
\begin{equation*}
  \EE_{\theta}\Big(\norm{\hat{\psi}_1^{J_n} - \psi_1^{J_n}}_{L^2}^2 \Big)%
  \leq 2\alpha C^2L \frac{\log(n)}{n\gamma^*}.
\end{equation*}
Similarly, Proposition~\ref{pro:7b} tells us that  $$\PP_{\theta}\Bigg(\norm{\hat{G}^{J_n} - G^{J_n}}_{L^2} \geq  CL \sqrt{\frac{x}{n\gamma^*}}%
	+ C \max(\tau 2^{J_n/2},\sqrt{L}2^{J_n/2},\tau\sqrt{L}) \frac{x}{n\gamma^*}
	\Bigg)%
	\leq 4\cdot 24^{2^{J_n}}e^{-x},$$ hence, for any $\alpha > 0$, using that $2^{J_n}\leq 2\log (n)$ for $n\geq 2$, Lemma~\ref{lem:9} with $a = C\tau \sqrt{2L\log(n)}/\gamma^*$, $b = CL/\sqrt{\gamma^*}$, $c = 4\times 24^{2\log n}$, and $d^2 = \alpha C^2L^2 \log(n)/(n\gamma^*)$ [and by remarking that $\max(\tau 2^{J_n/2},\sqrt{L}2^{J_n/2},\tau \sqrt{L}) \leq \tau \sqrt{L} 2^{J_n/2}$] yields
\begin{align*}
  &\EE_{\theta}\Big(\norm{\hat{G}^{J_n} - G^{J_n}}_{L^2}^2 \Big)\\%
  &\qquad\leq \alpha C^2 L^2 \frac{\log(n)}{n\gamma^*}
  + c\Big(d^2 + \frac{5b^2}{2n} + \frac{7a^2}{2n^2}%
  \Big)%
    e^{-nd^2/(2b^2 + 8ad)}\\
  &\qquad\leq \alpha C^2 L^2 \frac{\log(n)}{n\gamma^*}
  + 4C^2 24^{2\log(n)}\Big(\frac{\alpha L^2\log(n)}{n\sg} + \frac{5L^2}{2n\sg} + \frac{14 \tau^2L \log(n)}{2(n\sg)^2}%
  \Big)%
    e^{-nd^2/(2b^2 + 8ad)}.
\end{align*}
Let us study the argument of the exponential in the last display. If $2b^2 \geq 8ad$, then
\begin{align*}
  \frac{nd^2}{2b^2 + 8ad}%
  &\geq \frac{nd^2}{4b^2}%
  =\frac{\alpha}{4} \log(n)
\end{align*}
while if $2b^2 < 8 ad$, then
\begin{align*}
  \frac{nd^2}{2b^2 + 8ad}%
  &\geq \frac{nd^2}{16ad}
    =\frac{n\gamma^* \sqrt{\alpha C^2 L^2\log(n)/(n\gamma^*)}}{16C\tau \sqrt{L}2^{J_n/2} }%
    \geq \frac{\sqrt{\alpha L}}{32\tau}\sqrt{n\gamma^*}%
    \geq \frac{\sqrt{\alpha}}{32} \log(n)
\end{align*}
because by assumption $n\gamma^* \geq \frac{\tau^2 \log(n)^2}{L}$.  Since by assumption $L \leq n$ and $n\sg \geq \tau^3 \geq 1$, it is possible to choose $\alpha > 0$ universally such that
\begin{equation*}
  \EE_{\theta}\Big(\norm{\hat{G}^{J_n} - G^{J_n}}_{L^2}^2 \Big)%
  \leq 2\alpha C^2 L^2 \frac{\log(n)}{n\gamma^*}.
\end{equation*}
Returning to \eqref{eqn:R1bound} and feeding the bound for $E_\theta \max_j \abs{\hat{m}_j-m_j}^2$ from Proposition~\ref{pro:expec-maxm}, we deduce that
\begin{equation*}
  R_1(\theta)%
  \leq 6\alpha C^2L\Big(1 + \frac{g^2L}{m_1^2}\Big)\frac{\log(n)}{n\gamma^*}
  + \frac{3\cdot 83^2\cdot 40C^2L^3\max(1,g^2)}{n\gamma^* m_2^2}%
  + \frac{3\cdot 83^2\cdot 64C^2\max(\tau,\sqrt{L})^6}{(n\gamma^*)^2 m_2^2}.
\end{equation*}
Finally, we remark  $\frac{g^2}{m_1^2} \leq \frac{16}{\delta^2\epsilon^2 \zeta^2 \tilde{\mathcal{I}}^2}$ and $\frac{\max(1,g^2)}{m_2^2} \leq \frac{16}{\delta^2\epsilon^4\zeta^4\tilde{\mathcal{I}}^4}$ by Lemma~\ref{lem:basic-relations} and by the assumption that $\zeta \leq 1$. Thus, there exists a universal constant $B > 0$ such that
\begin{equation*}
  \sup_{\theta \in \SmoothClass \cap \RegulClass}R_1(\theta)%
  \leq \frac{B L^2}{\delta^2\epsilon^2\zeta^2}\frac{\log(n)}{n\sg}%
  + \frac{BL^3}{\delta^2\epsilon^4\zeta^4}\frac{1}{n\sg}%
  +\frac{B\max(\tau,\sqrt{L})^6}{\delta^2\epsilon^4\zeta^4}\frac{1}{(n\sg)^2}.
\end{equation*}

\subsubsection{Control of $R_2$}


From equation~\eqref{eq:def:besov} whenever $\theta \in \SmoothClass$ it is the case that $\sup_{j\geq J}2^{2js_{\pm}}\sum_k\abs{f_{\pm}^{\Psi_{jk}}}^2 \leq R^2$. This in particular implies that $\sum_{ \ell }\norm{f_{\pm}^{\Block_{j\ell}}}^2 \leq R^2 2^{-2js_{\pm}}$.
Moreover $\hat{S}_n \leq 4 S_n$ on $\Omega_n$ by Proposition~\ref{pro:4} in Section \ref{sec:auxiliary-results}. Then, since $J_n \leq \tilde{\jmath}_n$, 
\begin{align*}
  R_2(\theta)
  &\coloneqq  \EE_{\theta}\Bigg( \sum_{j=J_n}^{\tilde{\jmath}_n}\sum_{\ell }\norm{f_{\pm}^{\Block_{j\ell}}}^2\1_{\{ \norm{\hat{f}_{\pm}^{\Block_{j\ell}}} \leq \Gamma \hat{S}_n\}  }\1_{\{ \norm{f_{\pm}^{\Block_{j\ell}}} \leq 2\Gamma \hat{S}_n \} }\1_{\Omega_n}\Bigg)\\ %
  &\leq \sum_{j=J_n}^{\tilde{\jmath}_n}\sum_{ \ell }\min\Big(\norm{f_{\pm}^{\Block_{j\ell}}}^2, 8\Gamma S_n \Big)^2\\
  &\leq%
    \sum_{j=J_n}^{\tilde{\jmath}_n}\min\Bigg(\sum_{\ell}\norm{f_{\pm}^{\Block_{j\ell}}}^2,\, \frac{2^j}{N}\cdot 64\Gamma^2S_n^2 \Bigg)
    \\
  &\leq \sum_{j=J_n}^{\tilde{\jmath}_n}\min\Big( R^22^{-2js_{\pm}},\, \frac{2^j}{N}\cdot 64\Gamma^2S_n^2 \Big).
\end{align*}
Define $A=\sup\braces{0 \leq j\leq \tilde{\jmath}_n : 2^{-j(s_{\pm}+1/2)} > 8\Gamma S_n/(R\sqrt{N})}$, so that the first term in the minimum is the smaller exactly when $j> A$. 
Then we observe that $2^A<(R^2N/(64 \Gamma^2 S_n^2))^{1/(2s_\pm +1)}$ and  $2^{A+1} \geq \min\braces{ (R^2N/(64\Gamma^2 S_n^2))^{1/(2s_\pm + 1)},n/(\tau^2\log n)}$ (for the latter recall that $\tilde{\jmath}$ is the largest integer such that $2^{\tilde{\jmath}}\leq n/(\tau^2\log n)$), and we calculate
\begin{align*}
  R_2(\theta)%
  &\leq \frac{64 \Gamma^2 S_n^2}{N} \sum_{j=0}^{A}2^j%
	+ R^2\sum_{j = A+1}^{\infty }2^{-2js_{\pm}}\\
	&\leq \frac{128\Gamma^2S_n^2}{N}\Bigg( \frac{c^2R^2N}{64\Gamma^2S_n^2} \Bigg)^{1/(2s_{\pm} + 1)} + \frac{R^2}{1-2^{-2s_{\pm}}}\max\Bigg(\frac{\tau^2 \log(n)}{n},\ \Big(\frac{64\Gamma^2 S_n^2}{R^2N }\Big)^{1/(2s_{\pm}+1)} \Bigg)^{2s_{\pm}}\\
	&=2R^2\Bigg(\frac{64\Gamma^2 S_n^2}{R^2N} \Bigg)^{2s_{\pm}/(2s_{\pm} + 1)}%
	+ \frac{R^2}{1-2^{-2s_{\pm}}}\max\Bigg( \frac{\tau^2\log(n)}{n} ,\ \Big(\frac{64\Gamma^2 S_n^2}{R^2N}\Big)^{1/(2s_{\pm}+1)} \Bigg)^{2s_{\pm}}.
\end{align*}

Recalling that $S_n=\sqrt{(\log n)/n}\max(1,g/\abs{m_1})$ and $N>\log n$, we deduce that
\begin{align*}
  R_2(\theta)%
  &\leq 2R^2\Bigg(\frac{64\Gamma^2\max(1,g^2/m_1^2) }{R^2 n} \Bigg)^{2s_{\pm}/(2s_{\pm} + 1)}\\%
  &\quad
    + \frac{R^2}{1-2^{-2s_{\pm}}}\max\Bigg( \frac{\tau^2\log(n)}{n} ,\ \Big(\frac{64\Gamma^2\max(1,g^2/m_1^2) }{R^2 n}\Big)^{1/(2s_{\pm}+1)} \Bigg)^{2s_{\pm}}.
\end{align*}
Hence, recalling that $\abs{\tilde{\mathcal{I}}} \geq 7/8$ and the result of Lemma~\ref{lem:basic-relations}, there exists a universal constant $B > 0$ such that
\begin{equation*}
  \sup_{\theta \in \SmoothClass \cap \RegulClass}R_2(\theta)%
  \leq \frac{BR^2}{\min(1,s_{\pm})}\Big(\frac{\Gamma^2}{R^2\delta^2\epsilon^2\zeta^2 n} \Big)^{2s_{\pm}/(2s_{\pm} + 1)}%
  + \frac{BR^2}{\min(1,s_{\pm})}\Big(\frac{\tau^2\log(n)}{n} \Big)^{2s_{\pm}}
\end{equation*}

\subsubsection{Control of $R_3$}
\label{sec:control-r_3}

We remark that on the event $\Set{ \norm{\hat{f}^{\Block_{j\ell}}} \leq \Gamma \hat{S}_n } \cap \Set{  \norm{f^{\Block_{j\ell}}} > 2\Gamma \hat{S}_n }$ it must that
\begin{align*}
  \norm{f^{\Block_{j\ell}}_{\pm}}
  &\leq \norm{\hat{f}^{\Block_{j\ell}}_{\pm} - f^{\Block_{j\ell}}_{\pm}} + \norm{\hat{f}^{\Block_{j\ell}}_{\pm}}%
    \leq \norm{\hat{f}^{\Block_{j\ell}}_{\pm} - f^{\Block_{j\ell}}_{\pm}} +  \frac{1}{2}\norm{f^{\Block_{j\ell}}_{\pm}}
\end{align*}
and thus $\norm{f^{\Block_{j\ell}}_{\pm}} \leq 2\norm{\hat{f}^{\Block_{j\ell}}_{\pm} - f^{\Block_{j\ell}}_{\pm}}$. Then, since  $\frac{1}{4}S_n \leq \hat{S}_n \leq 4S_n$  on the event $\Omega_n$ by Proposition~\ref{pro:4} in Section \ref{sec:auxiliary-results},
\begin{align*}
  R_3(\theta) &\coloneqq \EE_{\theta}\Bigg( \sum_{j=J_n}^{\tilde{\jmath}_n}\sum_{\ell}\norm{f_{\pm}^{\Block_{j\ell}}}^2\1_{\{ \norm{\hat{f}_{\pm}^{\Block_{j\ell}}} \leq \Gamma \hat{S}_n\}  }\1_{\{ \norm{f_{\pm}^{\Block_{j\ell}}} > 2\Gamma \hat{S}_n \} }\1_{\Omega_n}\Bigg) \\
  &\leq 4\sum_{j=J_n}^{\tilde{\jmath}_n}\sum_{\ell}\EE_{\theta}\Big(\norm{\hat{f}_{\pm}^{\Block_{j\ell}} - f_{\pm}^{\Block_{j\ell}}}^2\1_{\{ \norm{\hat{f}_{\pm}^{\Block_{j\ell}}} \leq \Gamma \hat{S}_n\}  }\1_{\{ \norm{f_{\pm}^{\Block_{j\ell}}} > 2\Gamma \hat{S}_n \} }\1_{\Omega_n}\Big)\\
  &\leq 4\sum_{j=J_n}^{\tilde{\jmath}_n}\sum_{\ell}\EE_{\theta}\Big(\norm{\hat{f}_{\pm}^{\Block_{j\ell}} - f_{\pm}^{\Block_{j\ell}}}^2\1_{\{ \norm{\hat{f}_{\pm}^{\Block_{j\ell}} - f_{\pm}^{\Block_{j\ell}}} > \Gamma S_n / 4\}  } \1_{\Omega_n} \Big).
\end{align*}
Recalling that $\hat{f}_{\pm} = \hat{\psi}_1 + \tfrac{1}{2}\hat{\omega}_\pm \hat{G}$, we define $W_1^{\Block_{j\ell}} \coloneqq \norm{\hat{\psi}_1^{\Block_{j\ell}} - \psi_1^{\Block_{j\ell}}}$, $W_2^{\Block_{j\ell}} \coloneqq \frac{4g}{|m_1|}\norm{\hat{G}^{\Block_{j\ell}} - G^{\Block_{j\ell}}}$, and $W_3^{\Block_{j\ell}} \coloneqq \frac{1}{2}|\hat{\omega}_{\pm} - \omega_{\pm}|\norm{G^{\Block_{j\ell}}}$, so that a direct calculation (see Lemma~\ref{lem:pro:7}) yields $\norm{\hat{f}_\pm^{\Block_{j\ell}} - f_\pm^{\Block_{j\ell}}}_{L^2}\leq W_1^{\Block_{j\ell}}+W_2^{\Block_{j\ell}}+W_3^{\Block_{j\ell}}$. We then observe, writing  $\bar{W}^{\Block_{j\ell}} = \max(W_1^{\Block_{j\ell}},W_2^{\Block_{j\ell}},W_3^{\Block_{j\ell}})$, that
\begin{align*}
  R_3(\theta)
  &\leq 4\sum_{j=J_n}^{\tilde{\jmath}_n}\sum_{\ell}\EE_{\theta}\Big(\norm{\hat{f}_{\pm}^{\Block_{j\ell}} - f_{\pm}^{\Block_{j\ell}}}^2\1_{\{ \norm{\hat{f}_{\pm}^{\Block_{j\ell}} - f_{\pm}^{\Block_{j\ell}}} > \Gamma S_n / 4\}  }\1_{ \{\bar{W}^{\Block_{j\ell}}  = W_1^{\Block_{j\ell}} \} } \1_{\Omega_n} \Big)\\
  &\quad%
    + 4\sum_{j=J_n}^{\tilde{\jmath}_n}\sum_{\ell}\EE_{\theta}\Big(\norm{\hat{f}_{\pm}^{\Block_{j\ell}} - f_{\pm}^{\Block_{j\ell}}}^2\1_{\{ \norm{\hat{f}_{\pm}^{\Block_{j\ell}} - f_{\pm}^{\Block_{j\ell}}} > \Gamma S_n / 4\}  }\1_{ \{ \bar{W}^{\Block_{j\ell}} = W_2^{\Block_{j\ell}} \} } \1_{\Omega_n} \Big)\\
  &\quad%
   + 4\sum_{j=J_n}^{\tilde{\jmath}_n}\sum_{\ell}\EE_{\theta}\Big(\norm{\hat{f}_{\pm}^{\Block_{j\ell}} - f_{\pm}^{\Block_{j\ell}}}^2\1_{\{ \norm{\hat{f}_{\pm}^{\Block_{j\ell}} - f_{\pm}^{\Block_{j\ell}}} > \Gamma S_n / 4\}  }\1_{ \{\bar{W}^{\Block_{j\ell}} =  W_3^{\Block_{j\ell}} \} } \1_{\Omega_n} \Big)
\end{align*}
We call these terms $R_{3,1}$, $R_{3,2}$, and $R_{3,3}$, respectively. Let us start with $R_{3,1}$. Observe that on the event $\Omega_n \cap \Set{ \bar{W}^{\Block_{j\ell}} = W_1^{\Block_{j\ell}} }$ we have $\norm{\hat{f}_{\pm}^{\Block_{j\ell}} - f_{\pm}^{\Block_{j\ell}}} \leq 3W_1^{\Block_{j\ell}}$. Therefore,
\begin{align*}
  R_{3,1}%
  &\leq 36 \sum_{j=J_n}^{\tilde{\jmath}_n} \sum_{\ell} \EE_{\theta}\Big( \big(W_1^{\Block_{j\ell}}\big)^2\1_{ \{ W_1^{\Block_{j\ell}} > \Gamma S_n/12 \} } \Big)
\end{align*}
Proposition~\ref{pro:8a} in Section \ref{sec:auxiliary-results} tells us that, for $n\gamma^* \geq 1/99$, there is a universal constant $C > 0$ such that for all $\theta \in \RegulClass$ and all $x \geq 0$
\begin{equation*}
	\PP_{\theta}\Bigg(\norm{\hat{\psi}_1^{\Block_{j\ell}} - \psi_1^{\Block_{j\ell}}}\geq C \sqrt{\frac{Lx}{n\gamma^*}} + C2^{j/2}\frac{x}{n\gamma^*} \Bigg)%
	\leq 24^Ne^{-x}.
\end{equation*}
Then by Lemma~\ref{lem:9} with $a = C 2^{j/2} / \gamma^*$, $b = C\sqrt{L/\gamma^*}$, $c=24^{N}\leq 24^{2\log(n)}$ [$n\geq 2$ so $N \leq 2\log(n)$], we find that
\begin{align*}
R_{3,1}%
  &\leq 36 \cdot 24^N \sum_{j=J_n}^{\tilde{\jmath}_n}\sum_{\ell}\Big(\frac{\Gamma^2 S_n^2}{144} +  \frac{5C^2 L}{2n\gamma^*} + \frac{7C^22^j}{2(n\gamma^*)^2} \Big)\exp\Big(- \frac{n\gamma^* \Gamma^2 S_n^2/144}{2C^2L + 8C2^{j/2}\Gamma S_n/12 } \Big).\\
  &\leq
    36 \cdot 24^N \Big(\frac{\Gamma^2 \max(1,g^2/m_1^2)}{144}%
    + 5C^2L n + \frac{14C^2 n^2}{2}\Big)\exp\Big(- \frac{n\gamma^* \Gamma^2 S_n^2/144}{2C^2L + 8C2^{j/2}\Gamma S_n/12 } \Big)
\end{align*}
where the last line follows since there are $2^j/N \leq 2^j$ blocks at each level $j$, and because $2^{\tilde{\jmath}_n} \leq n$ by construction whenever $n\geq 3$, and because $n\gamma^{*} \geq \tau^3 \geq 1$. Let us analyse the argument of the exponential in the last display. Firstly if $8C2^{j/2}\Gamma S_n/12 \leq 2C^2 L$, it is the case that
\begin{align*}
  \frac{n \gamma^*\Gamma^2 S_n^2/144}{2C^2L + 8C2^{j/2}\Gamma S_n/12 }%
  &\geq \frac{n\gamma^* \Gamma^2 S_n^2}{576C^2L} \geq \frac{\gamma^*\Gamma^2}{576 C^2 L} \log(n)
\end{align*}
since $S_n = \sqrt{\log(n)/n}\max(1,g/|m_1|)$. Secondly, if $8C2^{j/2}\Gamma S_n/12 > 2C^2 L$, it is the case that for any $j \leq \tilde{\jmath}_n$
\begin{align*}
  \frac{n \gamma^*\Gamma^2 S_n^2/144}{2C^2L + 8C2^{j/2}\Gamma S_n/12 }%
  &\geq \frac{n \gamma^*\Gamma S_n}{192 C2^{j/2} }%
    \geq \frac{\gamma^* \Gamma}{192C}2^{-\tilde{\jmath}_n/2}\sqrt{n \log(n)}%
    \geq \frac{\gamma^* \Gamma}{192C} \log(n)
\end{align*}
since by construction $2^{\tilde{\jmath}_n}\leq \frac{n}{\tau^2\log(n)} \leq \frac{n}{\log(n)}$. Therefore since $L\leq n$ by assumption, for any $A > 0$ there exists $c_0 > 0$ such that whenever $\Gamma \geq c_0\max(L^{1/2}(\gamma^*)^{\mathrm{-1/2}},(\gamma^*)^{-1} )$:
\begin{equation*}
  R_{3,1} \leq \max\Big(1,\, \frac{g^2}{m_1^2}\Big)n^{-A}.
\end{equation*}
We now control $R_{3,2}$. With the same argument as before,
\begin{align*}
  R_{3,2}%
  &\leq 36 \sum_{j=J_n}^{\tilde{\jmath}_n} \sum_{\ell} \EE_{\theta}\Big( \big(W_2^{\Block_{j\ell}}\big)^2\1_{ \{ W_2^{\Block_{j\ell}} > \Gamma S_n/12 \} } \Big).
\end{align*}
Proposition~\ref{pro:8b} tells us that $$\PP_{\theta}\Big(\norm{\hat{G}^{\Block_{j\ell}} - G^{\Block_{j\ell}}} \geq  CL \sqrt{\frac{x}{n\gamma^*}} + C \max(\tau 2^{j/2},\sqrt{L}2^{j/2},\tau \sqrt{L}) \frac{x}{n\gamma^*}
\Big) \leq 4\cdot 24^Ne^{-x}.$$ 
Thus, applying Lemma~\ref{lem:9} with $a = \frac{4Cg}{|m_1|\gamma^*} \tau \sqrt{L} 2^{j/2}$, $b =\frac{4CL g}{|m_1|\sqrt{\gamma^*}}$, $c=24^{N}$, and $d = \Gamma S_n/12$ [note that $\max(\tau 2^{j/2},\sqrt{L}2^{j/2},\tau\sqrt{L}) \leq \tau \sqrt{L} 2^{j/2}$], we find that
\begin{multline*}
  R_{3,2}%
  \leq 36\cdot 24^N \sum_{j=J_n}^{\tilde{\jmath}_n}\sum_{\ell}\Bigg(\frac{\Gamma^2S_n^2}{144} + \frac{10C^2L^2 g^2}{n\gamma^* m_1^2} + \frac{7\cdot 4^2C^2\tau^2L 2^j g^2}{2(n\gamma^*)^2m_1^2} \Bigg)\\
  \times \exp\Bigg(%
  - \frac{n \gamma^* \Gamma^2 S_n^2 / 144}{ \frac{8C^2L^2g^2}{m_1^2} + \frac{16C\tau \sqrt{L} 2^{j/2} g}{12|m_1|}\Gamma S_n } \Bigg)
\end{multline*}
ie.
\begin{multline*}
  R_{3,2}%
  \leq 36\cdot 24^N\max\Big(1,\frac{g^2}{m_1^2}\Big)\Big(\frac{\Gamma^2}{144}%
    + 20C^2L^2 n%
    + \frac{14\cdot 4^2 \tau^2Ln^2}{2}
    \Big)\\%
    \times%
    \exp\Bigg(%
    - \frac{n \gamma^* \Gamma^2 S_n^2 / 144}{ \frac{8C^2L^2g^2}{m_1^2} + \frac{16C\tau \sqrt{L} 2^{j/2} g}{12|m_1|}\Gamma S_n } \Bigg)
\end{multline*}
Let us analyse the argument of the exponential in the previous display. Firstly, in the case where  $\frac{16C\tau \sqrt{L} 2^{j/2} g}{12|m_1|}\Gamma S_n \leq \frac{8C^2L^3g^2}{m_1^2}$,
\begin{align*}
  \frac{n \gamma^* \Gamma^2 S_n^2 / 144}{ \frac{8C^2L^2g^2}{m_1^2} + \frac{16C\tau \sqrt{L} 2^{j/2} g}{12|m_1|}\Gamma S_n }%
  &\geq \frac{n \gamma^* \Gamma^2 S_n^2}{ \frac{2304C^2L^2g^2}{m_1^2}}%
    \geq \frac{\gamma^* \Gamma^2}{2304C^2L^2} \log(n)
\end{align*}
since $S_n = \sqrt{\log(n)/n}\max(1,g/|m_1|)$. Secondly, in the case where  $\frac{16C\tau \sqrt{L} 2^{j/2} g}{12|m_1|}\Gamma S_n \leq \frac{8C^2L^2g^2}{m_1^2}$, for any $j \leq \tilde{\jmath}_n$
\begin{multline*}
  \frac{n \gamma^* \Gamma^2 S_n^2 / 144}{ \frac{8C^2L^2g^2}{m_1^2} + \frac{16C\tau \sqrt{L} 2^{j/2} g}{12|m_1|}\Gamma S_n }%
  \geq \frac{n \gamma^* \Gamma S_n}{ \frac{384C\tau \sqrt{L} 2^{j/2} g}{|m_1|} }%
    \geq \frac{\gamma^* \Gamma }{384C\tau \sqrt{L}}2^{-\tilde{\jmath}_n/2}\sqrt{n \log(n)}\\%
    \geq \frac{\gamma^* \Gamma}{384 C \sqrt{L}} \log(n)
\end{multline*}
since by construction $2^{\tilde{\jmath}_n}\leq \frac{n}{\tau^2}\log(n)$. Therefore, for any $A > 0$ there exits a constant $c_0 >0$ such that whenever $\Gamma \geq c_0 L^{1/2}\max(L^{1/2}(\gamma^*)^{-1/2},(\gamma^*)^{-1})$
\begin{equation*}
  R_{3,2} \leq  \max\Big(1,\frac{g^2}{m_1^2}\Big)n^{-A}.
\end{equation*}
We now control $R_{3,3}$. With the same argument as before,
\begin{align*}
  R_{3,3}%
  &\leq 36 \sum_{j=J_n}^{\tilde{\jmath}_n} \sum_{\ell} \EE_{\theta}\Big( \big(W_3^{\Block_{j\ell}}\big)^2\1_{ \{ W_3^{\Block_{j\ell}} > \Gamma S_n/12 \} }\1_{\Omega_n} \Big)\\
  &\leq 36 \sum_{j=J_n}^{\tilde{\jmath}_n} \sum_{\ell} \EE_{\theta}\Big( \big(W_3^{\Block_{j\ell}}\big)^2\1_{\Omega_n} \Big).
\end{align*}
Proposition~\ref{pro:9} in Section \ref{sec:auxiliary-results} tells us that $\abs{\hat{\omega}_{\pm} - \omega_{\pm}}%
\leq \frac{83\max(1,\phi_3|\tilde{\mathcal{I}}|)}{|m_1m_2|}\max_{j=1,2,3}\abs{\hat{m}_j - m_j}$ on the event $\Omega_n$, hence
\begin{align*}
  R_{3,3}%
  &\leq \frac{9\cdot 83^2\max(1,\phi_3^2\tilde{\mathcal{I}}^2)}{m_1^2m_2^2}\EE_{\theta}\Big(\max_{j=1,2,3}|\hat{m}_j - m_j|^2 \Big)%
    \sum_{j=J_n}^{\tilde{\jmath}_n}\sum_{\ell}\norm{G^{\Block_{j\ell}}}^2\\
  &\leq \frac{9\cdot 83^2\max(1,\phi_3^2\tilde{\mathcal{I}}^2)}{m_2^2}\EE_{\theta}\Big(\max_{j=1,2,3}|\hat{m}_j - m_j|^2 \Big)
\end{align*}
because $\norm{G}_{L^2} = \abs{m_1} \norm{\psi_2}_{L^2} = \abs{m_1}$. Furthermore, by  Proposition~\ref{pro:expec-maxm}, we deduce
\begin{equation*}
  R_{3,3} \leq \frac{9\cdot 83^2\cdot 40C^2 L^3\max(1,g^2)}{n\gamma^* m_2^2}%
  + \frac{9\cdot 83^2 \cdot 64C^2\max(\tau,\sqrt{L})^6\max(1,g^2)}{(n\gamma^*)^2m_2^2}.
\end{equation*}
In the end for every $A > 0$ there exists $c_0 > 0$ such that whenever the threshold constant satisfies $\Gamma \geq c_0 L^{1/2}\max(L^{1/2}(\gamma^*)^{-1/2},(\gamma^*)^{-1})$
\begin{multline*}
  R_3(\theta)%
  \leq 2\max\Big(1,\frac{g^2}{m_1^2}\Big)n^{-A} + \frac{9\cdot 83^2\cdot 40C^2 L^3\max(1,g^2)}{n\gamma^* m_2^2}\\%
    + \frac{9\cdot 83^2 \cdot 64C^2\max(\tau,\sqrt{L})^6\max(1,g^2)}{(n\gamma^*)^2m_2^2}.
\end{multline*}
By choosing $\beta > 0$ carefully enough, there is a universal constant $B > 0$ such that
\begin{equation*}
  \sup_{\theta\in \SmoothClass \cap \RegulClass}R_3(\theta)%
  \leq \frac{BL^3}{\delta^2\epsilon^4\zeta^4}\frac{1}{n\sg}%
  + \frac{B\max(\tau,\sqrt{L})^6}{\delta^2\epsilon^4\zeta^4}\frac{1}{(n\sg)^2}.
\end{equation*}

\subsubsection{Control of $R_4$}

Observe that
\begin{align*}
  R_4(\theta)%
  &\coloneqq \EE_{\theta}\Bigg( \sum_{j=J_n}^{\tilde{\jmath}_n}\sum_{\ell}\norm{\hat{f}_{\pm}^{\Block_{j\ell}} - f_{\pm}^{\Block_{j\ell}}}^2\1_{\{ \norm{\hat{f}_{\pm}^{\Block_{j\ell}}} > \Gamma \hat{S}_n\}  }\1_{ \{ \norm{f_{\pm}^{\Block_{j\ell}}} \leq \frac{1}{2}\Gamma \hat{S}_n \} \} } \1_{\Omega_n}\Bigg)\\
  &\leq \EE_{\theta}\Bigg( \sum_{j=J_n}^{\tilde{\jmath}_n}\sum_{\ell}\norm{\hat{f}_{\pm}^{\Block_{j\ell}} - f_{\pm}^{\Block_{j\ell}}}^2\1_{\{ \norm{\hat{f}_{\pm}^{\Block_{j\ell}} - f_{\pm}^{\Block_{j\ell}}} > \frac{1}{2}\Gamma \hat{S}_n\}  } \1_{\Omega_n}\Bigg)\\
  &\leq \EE_{\theta}\Bigg( \sum_{j=J_n}^{\tilde{\jmath}_n}\sum_{\ell}\norm{\hat{f}_{\pm}^{\Block_{j\ell}} - f_{\pm}^{\Block_{j\ell}}}^2\1_{\{ \norm{\hat{f}_{\pm}^{\Block_{j\ell}} - f_{\pm}^{\Block_{j\ell}}} > \frac{1}{8}\Gamma S_n\}  } \1_{\Omega_n}\Bigg)
\end{align*}
since $\hat{S}_n \geq S_n/4$ on the event $\Omega_n$ by Proposition~\ref{pro:4} in Section \ref{sec:auxiliary-results}. From here, we see that the bounds derived for $R_3$ adapts mutatis mutandis by letting $\Gamma \mapsto\Gamma/2$. In the end it is found that for $\beta > 0$ chosen carefully enough
\begin{equation*}
  \sup_{\theta\in \SmoothClass \cap \RegulClass}R_4(\theta)%
  \leq \frac{BL^3}{\delta^2\epsilon^4\zeta^4}\frac{1}{n\sg}%
  +   \frac{B\max(\tau,\sqrt{L})^6}{\delta^2\epsilon^4\zeta^4}\frac{1}{(n\sg)^2}.
\end{equation*}

\subsubsection{Control of $R_5$}
\label{sec:control-r_5}

First see that, since $\hat{S}_n \geq S_n/4$ on the event $\Omega_n$ by Proposition~\ref{pro:4},
\begin{align*}
  R_5(\theta) &\coloneqq \EE_{\theta}\Bigg( \sum_{j=J_n}^{\tilde{\jmath}_n}\sum_{\ell}\norm{\hat{f}_{\pm}^{\Block_{j\ell}} - f_{\pm}^{\Block_{j\ell}}}^2\1_{\{ \norm{\hat{f}_{\pm}^{\Block_{j\ell}}} > \Gamma \hat{S}_n\}  }\1_{ \{ \norm{f_{\pm}^{\Block_{j\ell}}} > \frac{1}{2}\Gamma \hat{S}_n \} } \1_{\Omega_n}\Bigg) \\ %
  &\leq
   \sum_{j=J_n}^{\tilde{\jmath}_n}\sum_{\ell}\EE_{\theta}\Big(\norm{\hat{f}_{\pm}^{\Block_{j\ell}} - f_{\pm}^{\Block_{j\ell}}}^2\1_{\Omega_n}\Big) \1_{ \{ \norm{f_{\pm}^{\Block_{j\ell}}} > \frac{1}{8}\Gamma S_n \} } .
\end{align*}
Let $W_j^{\Block_{j\ell}}$ be defined as in Section~\ref{sec:control-r_3}. Then, by Lemma~\ref{lem:pro:7} in Section \ref{sec:auxiliary-results},
\begin{align*}
  \EE_{\theta}\Big(\norm{\hat{f}_{\pm}^{\Block_{j\ell}} - f_{\pm}^{\Block_{j\ell}}}^2\1_{\Omega_n}\Big)%
  &\leq 3\EE_{\theta}\Big(\big(W_1^{\Block_{j\ell}} \big)^2\Big)%
    + 3\EE_{\theta}\Big(\big(W_2^{\Block_{j\ell}} \big)^2 \Big)%
    + 3\EE_{\theta}\Big(\big(W_3^{\Block_{j\ell}} \big)^2\1_{\Omega_n}\Big)
\end{align*}
By computations made in Section~\ref{sec:control-r_3}, for any $A > 0$ we can choose $\alpha > 0$ such that
\begin{align*}
  \EE_{\theta}\Big(\big(W_1^{\Block_{j\ell}} \big)^2\Big)%
  &\leq \alpha^2C^2 L \frac{\log(n)}{n \gamma^*}%
    + \EE_{\theta}\Big(\big(W_1^{\Block_{j\ell}} \big)^2\1_{ \{ W_1^{\Block_{j\ell}} >   \alpha C \sqrt{L \log(n)/(n\gamma^*)} \}  } \Big)\\
  &\leq \alpha^2C^2 L \frac{\log(n)}{n \gamma^*}
    + \max\Big(1,\frac{g^2}{m_1^2}\Big)2^{-\tilde{\jmath}_n}n^{-A}\\
  &\leq \frac{\alpha^2C^2L S_n^2}{\gamma^*} + \max\Big(1,\frac{g^2}{m_1^2}\Big)2^{-\tilde{\jmath}_n}n^{-A}.
\end{align*}
Similarly,
\begin{align*}
  \EE_{\theta}\Big(\big(W_2^{\Block_{j\ell}} \big)^2\Big)%
  &\leq \alpha^2C^2 L^2 \frac{g^2\log(n)}{n\gamma^* m_1^2}%
    + \EE_{\theta}\Big(\big(W_2^{\Block_{j\ell}} \big)^2\1_{ \{ W_2^{\Block_{j\ell}} >   \frac{\alpha C L g}{|m_1|} \sqrt{\log(n)/(n\gamma^*)} }\Big)\\
  &\leq \alpha^2C^2 L^2 \frac{g^2\log(n)}{n\gamma^* m_1^2}%
    + \max\Big(1,\frac{g^2}{m_1^2}\Big)2^{-\tilde{\jmath}_n}n^{-A}\\
  &\leq \frac{\alpha^2C^2L^2 S_n^2}{\gamma^*} + \max\Big(1,\frac{g^2}{m_1^2}\Big)2^{-\tilde{\jmath}_n}n^{-A}.
\end{align*}
Also, by computations made in Section~\ref{sec:control-r_3}, we know that
\begin{multline*}
  \sum_{j=J_n}^{\tilde{\jmath}_n}\sum_{\ell}\EE_{\theta}\Big(\big(W_3^{\Block_{j\ell}} \big)^2\1_{\Omega_n} \Big)%
  \leq \frac{9\cdot 83^2\cdot 40C^2 L^3\max(1,g^2)}{36n\gamma^* m_2^2}\\%
  + \frac{9\cdot 83^2 \cdot 64C^2\max(\tau,\sqrt{L})^6\max(1,g^2)}{36(n\gamma^*)^2m_2^2}.
\end{multline*}
Consequently,
\begin{align*}
  R_5(\theta)%
  &\leq%
   \frac{6\alpha^2C^2L^2S_n^2}{\gamma^*}
    \sum_{j=J_n}^{\tilde{\jmath}_n}\sum_{\ell}\1_{ \{ \norm{f_{\pm}^{\Block_{j\ell}}} > \frac{1}{8}\Gamma S_n \} }\\
  &\quad%
   +\frac{27\cdot 83^2\cdot 40C^2 L^3\max(1,g^2)}{36n\gamma^* m_2^2}%
    + \frac{27\cdot 83^2 \cdot 64C^2\max(\tau,\sqrt{L})^6\max(1,g^2)}{36(n\gamma^*)^2m_2^2}\\
  &\quad%
    + 2\max\Big(1,\frac{g^2}{m_1^2}\Big) n^{-A}.%
\end{align*}
Whenever $\theta \in \SmoothClass$, it is the case (recall \eqref{eqn:BesovBound}) that $\sup_{j\geq J_n}2^{2js_{\pm}}\sum_k|f_{\pm}^{\Psi_{jk}}|^2 \leq R^2$. This in particular implies that for all $j \geq J_n$
\begin{align*}
  R^22^{-2j s_{\pm}}%
  &\geq \sum_{\ell}\norm{f_{\pm}^{\Block_{j\ell}}}^2\\
  &\geq \sum_{\ell}\norm{f_{\pm}^{\Block_{j\ell}}}^2\1_{ \{\norm{f_{\pm}^{\Block_{j\ell}}} > \frac{1}{8}\Gamma S_n\} }\\
  &\geq \frac{\Gamma^2 S_n^2}{64}\sum_{\ell}\1_{ \{ \norm{f_{\pm}^{\Block_{j\ell}}} > \frac{1}{8}\Gamma S_n\} }.
\end{align*}%
Since there are $2^j/N$ blocks at level $j$, deduce that
\begin{equation*}
  \sum_{\ell}\1_{ \{ \norm{f_{\pm}^{\Block_{j\ell}}} > \frac{1}{8}\Gamma S_n\} }%
  \leq \min\Big(\frac{2^j}{N},\, \frac{64 R^2}{\Gamma^2S_n^2}2^{-2js_{\pm}} \Big) = \frac{1}{\Gamma^2S_n^2} \min\Big(\frac{2^j}{N}\Gamma^2S_n^2,\, 64 R^2 2^{-2js_{\pm}} \Big) 
\end{equation*}
Therefore,
\begin{align*}
  R_5(\theta)%
  &\leq%
   \frac{6\alpha^2C^2L^2}{\Gamma^2\gamma^*}
    \sum_{j=J_n}^{\tilde{\jmath}_n} \min\Big(\frac{2^j}{N}\Gamma^2S_n^2,\, 64 R^2 2^{-2js_{\pm}} \Big)\\
  &\quad%
   +\frac{27\cdot 83^2\cdot 40C^2 L^3\max(1,g^2)}{36n\gamma^* m_2^2}%
    + \frac{27\cdot 83^2 \cdot 64C^2\max(\tau,\sqrt{L})^6\max(1,g^2)}{36(n\gamma^*)^2m_2^2}\\
  &\quad%
    + 2\max\Big(1,\frac{g^2}{m_1^2}\Big)n^{-A}.%
\end{align*}
Then by inspecting the proof of the bound of $R_2(\theta)$ and by choosing $\alpha$ sufficiently large it follows immediately that there exists a universal constant $B > 0$ such that
\begin{multline*}
  \sup_{\theta\in \SmoothClass \cap \RegulClass}R_5(\theta)%
  \leq \frac{BL^3}{\delta^2\epsilon^4\zeta^6}\frac{1}{n\sg}%
  + \frac{B\max(\tau,\sqrt{L})^6}{\delta^2\epsilon^4\zeta^4}\frac{1}{(n\sg)^2}\\
  + \frac{BL^2}{\Gamma^2\sg} \Bigg( \frac{R^2}{\min(1,s_{\pm})}\Big(\frac{\Gamma^2}{R^2\delta^2\epsilon^2\zeta^2 n} \Big)^{2s_{\pm}/(2s_{\pm} + 1)}%
    + \frac{R^2}{\min(1,s_{\pm})}\Big(\frac{\tau^2\log(n)}{n} \Big)^{2s_{\pm}} \Bigg).%
\end{multline*}

\subsubsection{Control of $R_6$}
\label{sec:control-r_6}

Whenever $\theta \in \SmoothClass$, it is the case (recall equation \eqref{eqn:BesovBound}) that $\sup_{j\geq J_n}2^{2js_{\pm}}\sum_k|f_{\pm}^{\Psi_{jk}}|^2 \leq R^2$. Therefore, 
\begin{align*}
  R_6(\theta) &\coloneqq \PP_{\theta}(\Omega_n)\sum_{j > \tilde{\jmath}_n}\sum_{k=0}^{2^j-1}|f_{\pm}^{\Psi_{jk}}|^2 %
  \leq R^2 \sum_{j > \tilde{\jmath}_n}2^{-2js_{\pm}}
  = \frac{L^2}{2^{2s_{\pm}} - 1} 2^{-2\tilde{\jmath}_n s_{\pm}}
  \\ & \leq \frac{R^2}{2^{2s_{\pm}} - 1} \Big( \frac{2\tau^2 \log(n)}{n} \Big)^{2s_{\pm}}
\end{align*}
because by construction $2^{\tilde{\jmath}_n + 1} > \frac{n}{\tau^2\log(n)}$. Hence, there is a universal constant $B > 0$ such that
\begin{equation*}
  \sup_{\theta\in \SmoothClass \cap \RegulClass}R_6(\theta)%
  \leq \frac{BR^2}{\min(1,s_{\pm})}\Big( \frac{\tau^2 \log(n)}{n} \Big)^{2s_{\pm}}.
\end{equation*}

\subsubsection{Auxiliary results}
\label{sec:auxiliary-results}

\begin{proposition}
  \label{pro:7a}
  Let $n\gamma^* \geq 1/99$. Then, there is a universal constant $C > 0$ such that for all $\theta \in \RegulClass$ and all $x \geq 0$
  \begin{equation*}
    \PP_{\theta}\Bigg(\norm{\hat{\psi}_1^{J_n} - \psi_1^{J_n}}_{L^2} \geq C\sqrt{\frac{Lx}{n\gamma^*}} + C2^{J_n/2}\frac{x}{n\gamma^*} \Bigg)%
    \leq 24^{2^{J_n}}e^{-x}.
  \end{equation*}
\end{proposition}
\begin{proof}
  The strategy is classical and consists on remarking that $\norm{\hat{\psi}_1^{J_n} - \psi_1^{J_n}}_{L^2} = \sup_{u\in U}\Inner{\hat{\psi}_1^{J_n} - \hat{\psi}_1^{J_n}, u}$ where $U$ is the unit ball of the appropriate vector space (which has dimension $2^J + \sum_{j=J}^{J_n-1}2^j = 2^{J_n}$). Then, letting $\mathcal{N}$ be a $(1/2)$-net over $U$ and $\pi : U \to \mathcal{N}$ mapping any point $u \in U$ to its closest element in $\mathcal{N}$, we see that
  \begin{align*}
    \norm{\hat{\psi}_1^{J_n} - \psi_1^{J_n}}_{L^2}%
    &= \sup_{u\in U}\Inner{\hat{\psi}_1^{J_n} - \hat{\psi}_1^{J_n}, u}\\
    &= \sup_{u\in U}\Big(\Inner{\hat{\psi}_1^{J_n} - \hat{\psi}_1^{J_n}, \pi(u)}%
      + \Inner{\hat{\psi}_1^{J_n} - \hat{\psi}_1^{J_n}, u - \pi(u)}\Big)\\
    &\leq \max_{u\in \mathcal{N}}\Inner{\hat{\psi}_1^{J_n} - \hat{\psi}_1^{J_n}, u}  + \frac{1}{2}\norm{\hat{\psi}_1^{J_n} - \psi_1^{J_n}}_{L^2}
  \end{align*}
  and hence
  \begin{equation*}
    \norm{\hat{\psi}_1^{J_n} - \psi_1^{J_n}}_{L^2}%
    \leq 2 \max_{u\in \mathcal{N}}\Inner{\hat{\psi}_1^{J_n} - \hat{\psi}_1^{J_n}, u}.
  \end{equation*}
  It follows that
  \begin{align*}
    \PP_{\theta}\Big( \norm{\hat{\psi}_1^{J_n} - \psi_1^{J_n}}_{L^2} \geq 2x \Big)%
    &\leq |\mathcal{N}|\max_{u\in \mathcal{N}}\PP_{\theta}\Big( \Inner{\hat{\psi}_1^{J_n} - \hat{\psi}_1^{J_n}, u} \geq x \Big)
  \end{align*}
  The conclusion follows by Lemma~\ref{lem:paulin} applied to the function $h(y) = \sum_{k=0}^{2^J-1}u_{Jk}\Phi_{Jk}(y) + \sum_{j=J}^{J_n}\sum_{k=0}^{2^j-1}u_{jk}\psi_{jk}(y)$, because $\EE_{\theta}(h^2) \leq L \norm{h}_{L^2}^2 = L$ for every $\theta \in \RegulClass$ by Lemma~\ref{lem:2}, because $\norm{h}_{\infty} \leq c2^{J_n/2}$ for a universal $c >0$, by standard localization properties of wavelets \citep[Theorem~4.2.10 or Definition~4.2.14]{GN16} and because $\mathcal{N}$ can be chosen so that $ |\mathcal{N}| \leq 24^{2^{J_n}}$ because $\mathcal{N}$ can always be chosen to have cardinality no more than $24^{2^{J_n}}$ \citep[e.g.][Theorem 4.3.34]{GN16}.
\end{proof}

\begin{proposition}
  \label{pro:7b}
  Let $n\gamma^* \geq 1/99$ and $\norm{\tilde{\psi}_2}_{\infty}\leq \tau$. Then, there is a universal constant $C > 0$ such that for all $\theta \in \RegulClass$ and all $x \geq 0$
  \begin{equation*}
    \PP_{\theta}\Bigg(\norm{\hat{G}^{J_n} - G^{J_n}}_{L^2} \geq  CL \sqrt{\frac{x}{n\gamma^*}}%
    + C \max(\tau 2^{J_n/2},\sqrt{L}2^{J_n/2},\tau\sqrt{L}) \frac{x}{n\gamma^*}
    \Bigg)%
    \leq 4\cdot 24^{2^{J_n}}e^{-x}.
  \end{equation*}
\end{proposition}
\begin{proof}
  We remark that $\hat{G}^{\Phi_{Jk}} = \PP_n^{(2)}(\tilde{\psi}_2 \otimes \Phi_{Jk}) - \PP_n^{(1)}(\tilde{\psi}_2)\PP_n^{(1)}(\Phi_{Jk})$; similarly for $\hat{G}^{\Psi_{jk}}$. Recall that $\norm{\tilde{\psi}_2}_{\infty} \leq \tau$ by assumption. Hence, $\norm{\hat{G}^{J_n}}_{L^2} \leq c \tau 2^{J_n/2}$ for a universal constant $c > 0$. Similarly $\norm{G^{J_n}}_{L^2} \leq c \tau 2^{J_n/2}$. Hence with probability $1 \geq 1 - e^{-x}$, whenever $x > n\gamma^*$
  \begin{align*}
    \norm{\hat{G}^{J_n} - G^{J_n}}_{L^2}%
    \leq 2c\tau 2^{J_n/2}%
    \leq CL^{3/2}\sqrt{\frac{x}{n\gamma^*}}%
  \end{align*}
  provided $C > 2c$. We now consider the case where $0 \leq x \leq n\gamma^*$. We decompose
  \begin{align*}
    \hat{G}^{J_n} - G^{J_n}%
    &= \sum_{k=0}^{2^J-1}\Big(\PP_n^{(2)}\big(\tilde{\psi}_2 \otimes \Phi_{Jk} \big) - \EE_{\theta}\big(\tilde{\psi}_2 \otimes \Phi_{Jk}\big)\Big)\Phi_{Jk}\\
    &\quad%
      + \sum_{j=J}^{J_n}\sum_{k=0}^{2^j-1}\Big(\PP_n^{(2)}\big(\tilde{\psi}_2 \otimes \Psi_{jk} \big) - \EE_{\theta}\big(\tilde{\psi}_2 \otimes \Psi_{jk}\big)\Big)\Psi_{jk}\\
    &\quad%
      - \EE_{\theta}(\tilde{\psi}_2)\Big(\hat{\psi}_1^{J_n} - \psi_1^{J_n}\Big)%
      - \psi_1^{J_n}\Big(\PP_n^{(1)}(\tilde{\psi}_2) - \EE_{\theta}(\tilde{\psi}_2) \Big)\\%
    &\quad%
      - \Big( \PP_n^{(1)}(\tilde{\psi}_2) - \EE_{\theta}(\tilde{\psi}_2) \Big)\Big(\hat{\psi}_1^{J_n} - \psi_1^{J_n} \Big).
  \end{align*}
  But $\norm{\psi_1^{J_n}}_{L^2} \leq \norm{\psi_1}_{L^2} \leq \max(\norm{f_0}_{L^2},\norm{f_1}_{L^2})$ and $\norm{f_j}_{L^2}^2 = \int_0^1f_j^2 \leq \norm{f_j}_{\infty}\int_0^1f_j \leq L$ whenever $\theta \in \RegulClass$. Thus $\norm{\psi_1^{J_n}}_{L^2} \leq \sqrt{L}$. Similarly by Cauchy-Schwarz' $\abs{\EE_{\theta}(\tilde{\psi}_2)} \leq \EE_{\theta}(\tilde{\psi}_2^2)^{1/2} \leq \norm{\psi_1}_{\infty}^{1/2}\norm{\tilde{\psi}_2}_{L^2} \leq  \sqrt{L}$. Therefore, letting $v^{J_n}\coloneqq \sum_{k=0}^{2^J-1}\EE_{\theta}(\tilde{\psi}_2\otimes \Phi_{Jk})\Phi_{Jk} + \sum_{j=J}^{J_n}\sum_{k=0}^{2^{j-1}}\EE_{\theta}(\tilde{\psi}_2\otimes \Psi_{jk})\Psi_{jk}$ and its empirical counterpart $\hat{v}^{J_n}$ defined similarly:
  \begin{multline*}
    \norm{\hat{G}^{J_n} - G^{J_n}}_{L^2}%
    \leq \norm{\hat{v}^{J_n} - v^{J_n}}_{L^2}%
    + \sqrt{L}\norm{\hat{\psi}_1^{J_n} - \psi_1^{J_n}}_{L^2}%
      +  \sqrt{L}\Big| \PP_n^{(1)}(\tilde{\psi}_2) - \EE_{\theta}(\tilde{\psi}_2) \Big|\\%
      + \Big| \PP_n^{(1)}(\tilde{\psi}_2) - \EE_{\theta}(\tilde{\psi}_2)\Big| \norm{\hat{\psi}_1^{J_n} - \psi_1^{J_n}}_{L^2}.
  \end{multline*}
  Using the same $\varepsilon$-net argument as in the proof of Proposition~\ref{pro:7a}, we find that
  \begin{multline*}
    \PP_{\theta}\Big( \norm{\hat{v}^{J_n} - v^{J_n}}_{L^2} \geq CL\sqrt{\frac{x}{n\gamma^*} } + C\tau 2^{J/2}\frac{x}{n\gamma^*} \Big)\\%
    \leq 24^{2^{J_n}}\sup_{u \in U}\PP_{\theta}\Big( \Inner{\hat{v}^{J_n} - v^{J_n}, u} \geq CL\sqrt{\frac{x}{n\gamma^*} } + C\tau 2^{J/2}\frac{x}{n\gamma^*} \Big)
    \leq 24^{2^{J_n}}e^{-x}
  \end{multline*}
  where the last inequality follows from Lemma~\ref{lem:paulin} applied to the function $h(y_1,y_2) = \sum_{k=0}^{2^J-1}u_{Jk} \tilde{\psi}_2(y_1)\Phi_{Jk}(y_2) + \sum_{j=J}^{J_n}\sum_{k=0}^{2^j-1}u_{jk}\tilde{\psi}_2(y_1)\Psi_{jk}(y_1)$ which satisfies $\EE_{\theta}(h^2) \leq L^2\norm{h}_{L^2}^2 = L^2$ for every $\theta\in \RegulClass$ by Lemma~\ref{lem:2}, and $\norm{h}_{\infty} \leq c\norm{\tilde{\psi}_2}_{\infty}2^{J/2} \leq c \tau 2^{J/2}$ by standard localization properties of wavelets \citep[Theorem~4.2.10 or Definition~4.2.14]{GN16}. Also by Lemma~\ref{lem:paulin} applies to $h = \tilde{\psi}_2$,
  \begin{align*}
    \PP_{\theta}\Bigg( \Big| \PP_n^{(1)}(\tilde{\psi}_2) - \EE_{\theta}(\tilde{\psi}_2)\Big| \geq C\sqrt{\frac{Lx}{n\gamma^*}} + C\tau \frac{x}{n\gamma^*}   \Bigg)
    &\leq e^{-x}
  \end{align*}
  and using Proposition~\ref{pro:7a}
  \begin{align*}
    \PP_{\theta}\Bigg( \norm{\hat{\psi}_1^{J_n} - \psi_1^{J_n}} \geq C\sqrt{\frac{Lx}{n\gamma^*}} + C2^{J/2}\frac{x}{n\gamma^*} \Bigg)
    &\leq 24^{2^{J_n}}e^{-x}.
  \end{align*}
  Therefore with probability at least $1 - (2\cdot 24^{2^{J_n}} + 1)e^{-x}$ under $\PP_{\theta}$
  \begin{align*}
    \norm{\hat{G}^{J_n} - G^{J_n}}_{L^2}%
    &\leq C\Bigg(\sqrt{\frac{L^2x}{n\gamma^*}} + \tau 2^{J_n/2}\frac{x}{n\gamma^*}\Bigg)%
      + C\sqrt{L}\Bigg( \sqrt{\frac{Lx}{n\gamma^*}} + 2^{J_n/2}\frac{x}{n\gamma^*} \Bigg)\\
    &\quad%
      + C\sqrt{L}\Bigg(\sqrt{\frac{Lx}{n\gamma^*}} + \tau \frac{x}{n\gamma^*} \Bigg)%
      + C^2\Bigg(\sqrt{\frac{Lx}{n\gamma^*}} + 2^{J_n/2}\frac{x}{n\gamma^*}\Bigg)\Bigg(\sqrt{\frac{Lx}{n\gamma^*}} + \tau \frac{x}{n\gamma^*} \Bigg)\\
    &\leq 3CL \sqrt{\frac{x}{n\gamma^*}}%
      + C\Big(\tau 2^{J_n/2} + \sqrt{L}2^{J_n/2} + \tau \sqrt{L} + CL \Big) \frac{x}{n\gamma^*}\\
    &\quad%
      + C^2\Big(\tau \sqrt{L} + 2^{J_n/2}\sqrt{L} \Big)\frac{x^{3/2}}{(n\gamma^*)^{3/2}}
      + C^2 \tau 2^{J_n/2}\frac{x^2}{(n\gamma^*)^2}.
  \end{align*}
  The conclusion follows since $x \leq n\gamma^*$ which implies that the last two terms are bounded by the second term, and the $Lx/(n\gamma^*)$ part of second term is bounded by the first term.
\end{proof}

\begin{proposition}
  \label{pro:8a}
  Let $n\gamma^* \geq 1/99$. Then, there is a universal constant $C > 0$ such that for all $\theta \in \RegulClass$, all $j\geq J_n$, all $\ell$, and all $x \geq 0$,
  \begin{equation*}
    \PP_{\theta}\Bigg(\norm{\hat{\psi}_1^{\Block_{j\ell}} - \psi_1^{\Block_{j\ell}}}\geq C \sqrt{\frac{Lx}{n\gamma^*}} + C2^{j/2}\frac{x}{n\gamma^*} \Bigg)%
    \leq 24^Ne^{-x}.
  \end{equation*}
\end{proposition}
\begin{proof}
  The proof is identical to Proposition~\ref{pro:7a}. (Note the vector $\psi_1^{\Block_{j\ell}}$ is in $\RR^N$, where $\psi_1^{\Phi}$ was in $\RR^{2^{J_n}}$.)
\end{proof}

\begin{proposition}
  \label{pro:8b}
  Let $n\gamma^* \geq 1/99$. Then, there is a universal constant $C > 0$ such that for all $\theta \in \RegulClass$, all $j\geq J_n$, all $\ell$, and all $x \geq 0$
  \begin{equation*}
    \PP_{\theta}\Bigg(\norm{\hat{G}^{\Block_{j\ell}} - G^{\Block_{j\ell}}} \geq  CL \sqrt{\frac{x}{n\gamma^*}}%
    + C \max(\tau 2^{j/2},\sqrt{L}2^{j/2},\tau \sqrt{L}) \frac{x}{n\gamma^*}
    \Bigg)%
    \leq 4\cdot 24^Ne^{-x}.
  \end{equation*}
\end{proposition}
\begin{proof}
  The proof is identical to  Proposition~\ref{pro:7b}.
\end{proof}

\begin{lemma}
  \label{lem:pro:7}
  On the event $\Omega_n$, for all $j\geq J_n$ and all $\ell$:   \begin{align*}
    \norm{\hat{f}_{\pm}^{\Block_{j\ell}} - f_{\pm}^{\Block_{j\ell}}}%
    &\leq \norm{\hat{\psi}_1^{\Block_{j\ell}} - \psi_1^{\Block_{j\ell}}}%
      + \frac{4g\norm{ \hat{G}^{\Block_{j\ell}} - G^{\Block_{j\ell}}}}{\abs{m_1}} %
      + \frac{\abs{\hat{\omega}_{\pm} - \omega_{\pm}} \norm{G^{\Block_{j\ell}}}}{2},
  \end{align*}
  and similarly for $\norm{\hat{f}_{\pm}^{J_n} - f_{\pm}^{J_n}}_{L^2}$.
\end{lemma}
\begin{proof}
  Trivially,
  \begin{align*}
    \hat{f}_{\pm}^{\Block_{j\ell}} - f_{\pm}^{\Block_{j\ell}}%
    &= \hat{\psi}_1^{\Block_{j\ell}} - \psi_1^{\Block_{j\ell}}%
      + \frac{\hat{\omega}_{\pm}}{2}\Big( \hat{G}^{\Block_{j\ell}} - G^{\Block_{j\ell}} \Big)%
      + \frac{\hat{\omega}_{\pm} - \omega_{\pm}}{2} G^{\Block_{j\ell}}.%
  \end{align*}
  The conclusion follows since on $\Omega_n$, Proposition \ref{pro:4} implies that $\hat{g} \leq 2g$ and $|\hat{m}_1| \geq |m_1|/2 > 0$. (Recall also that $\abs{\phi_1}\leq 1$.)
\end{proof}

\begin{proposition}
  \label{pro:9}
  On the event $\Omega_n$
  \begin{equation*}
    |\hat{\omega}_{\pm} - \omega_{\pm}|%
    \leq \frac{83\max(1,\phi_3|\tilde{\mathcal{I}}|)}{|m_1m_2|}\max_{j=1,2,3}|\hat{m}_j - m_j|.
  \end{equation*}
\end{proposition}
\begin{proof}
	On $\Omega_n$ we have $\hat{g}\leq 2g$ by Proposition~\ref{pro:4} to follow, and note that $\abs{\hat{m}_1}\geq \abs{m_1}/2>0$. Consequently, by straightforward computations, using  Lemmas \ref{lem:10AA} and~\ref{lem:3}, 
  \begin{align*}
    |\hat{\omega}_{\pm} - \omega_{\pm}|%
    &= \abs[\Big]{\frac{1}{m_1}(\hat{g}-g)(1\mp \hat{\phi}_1) + \frac{g}{m_1}(1\mp \hat{\phi}_1-(1\mp \tilde{s}\phi_1)) + \hat{g}(1\mp \hat{\phi}_1)(\frac{1}{\hat{m}_1}-\frac{1}{m_1})}\\
    &\leq \frac{2|\hat{g} - g|}{|m_1|}%
      + \frac{g|\hat{\phi}_1 - \tilde{s}\phi_1| }{|m_1|}%
      + \frac{8 g | \hat{m}_1 - m_1|}{m_1^2}\\
    &\leq \Bigg(\frac{22\max(1,\phi_3|\tilde{\mathcal{I}}|)}{|m_1m_2|}%
      + \frac{53\max(1,\phi_3|\tilde{\mathcal{I}}|) g}{|m_1|\phi_2^2\phi_3^3|\tilde{\mathcal{I}}|^3}%
      + \frac{8g}{m_1^2}
      \Bigg) \max_{j=1,2,3}|\hat{m}_j - m_j|\\
    &\leq \frac{83\max(1,\phi_3|\tilde{\mathcal{I}}|)}{|m_1m_2|}\max_{j=1,2,3}|\hat{m}_j - m_j|
  \end{align*}
  because $m_2 = \frac{1}{4}(1-\phi_1^2)\phi_2^2\phi_3^2 \tilde{\mathcal{I}}^2$, because $g = \phi_3|\tilde{\mathcal{I}}|$, and because $m_2 = m_1\phi_2 \leq |m_1|$.
\end{proof}

\begin{proposition}
  \label{pro:4}
  On the event $\Omega_n$, we have $|\frac{\hat{g}}{g} - 1| \leq \frac{1}{2}$. Consequently, $\frac{1}{4}S_n \leq \hat{S}_n \leq 4S_n$ and $\abs{\hat{\omega}_\pm}\leq 8\abs{g/m_1}$ on $\Omega_n$.
\end{proposition}

\begin{proof}
  It suffices to remark that 
  \begin{equation*}
    \frac{gm_2}{44\max(1,g)}%
    \leq \frac{g m_2}{20\max(|\phi_1|,(1-\phi_1^2)\phi_3|\tilde{\mathcal{I}}|)},
  \end{equation*}
  since $-1 \leq \phi_1 \leq 1$, so that Lemma~\ref{lem:3} applies. Replacing $\max_j\abs{\hat{m}_j-m_j}$ by its bound $gm_2/44\max(1,g)$ on the event $\Omega_n$ yields the result for $\hat{g}$. For $S_n$, recalling the definitions $S_n=\sqrt{(\log n)/n} \max(1,g/\abs{m_1})$, $\hat{S}_n=\sqrt{(\log n)/n}\max(1,\hat{g}/\abs{\hat{m}_1})\II\braces{\hat{m}_1\neq 0}$ and inserting the bounds $g/2\leq \hat{g}\leq 2g$, $\abs{m_1}/2\leq \hat{m}_1\leq 2\abs{m_1}$ 
  yields the bounds for $\hat{S}_n$.
\end{proof}

\begin{lemma}
  \label{lem:3}
  Suppose
  \begin{equation*}
    \max_{j=1,2}\big|\frac{\hat{m}_j}{m_j} - 1\big| \leq \frac{1}{2},\quad\mathrm{and},\quad%
    \max_{j=1,2,3}|\hat{m}_j - m_j| \leq
  \frac{m_2g }{20 \max(|\phi_1|,
    (1-\phi_1^2)g)}
\end{equation*}
Then,
  \begin{equation*}
    |\hat{g} - g|%
    \leq \frac{22\max(1,g)}{m_2} \max_{j=1,2,3}|\hat{m}_j - m_j|.
  \end{equation*}
\end{lemma}
	Recall that $g=\phi_3\abs{\tilde{\mathcal{I}}}$ and $m_2=\phi_2r(\phi)\tilde{\mathcal{I}}^2$, so that the conditions of Lemma~\ref{lem:3} match those of Lemma~\ref{lem:10AA}.
\begin{proof}
  We let $\Delta_1 = \hat{m}_1 - m_1$, $\Delta_2 = (\hat{m}_2)_+ - m_2$, $\Delta_3 = \hat{m}_3 - m_3$, $\hat{v} \coloneqq 4\hat{m}_1^2(\hat{m}_2)_+ + \hat{m}_3^2$, $v \coloneqq 4m_1^2m_2 + m_3^2$, and $h \coloneqq \hat{v} - v$. Then, since $\hat{m}_2 \geq m_2/2 > 0$ under the assumption of the lemma 
  \begin{align*}
    \hat{g} - g%
    &= \frac{\sqrt{v + h}}{m_2 + \Delta_2} - \frac{\sqrt{v}}{m_2}\\
    &= \frac{\sqrt{v+ h} - \sqrt{v}}{m_2 + \Delta_2}
      - \frac{\Delta_2\sqrt{v}}{m_2(m_2 + \Delta_2)}\\
    &= \frac{h}{(\sqrt{v + h} + \sqrt{v})(m_2 + \Delta_2) } - \frac{\Delta_2\sqrt{v}}{m_2(m_2 + \Delta_2)}.
  \end{align*}
  Hence it must be that
  \begin{align*}
    |\hat{g} - g|%
    &\leq \frac{2|h|}{m_2\sqrt{v}}%
      + \frac{2\sqrt{v}}{m_2^2}|\Delta_2|.
  \end{align*}
Lemma~\ref{lem:controlofv}, together with the fact that $\abs{\phi_1}\leq 1$, tells us that $$|h| \leq 10 \max(1,\phi_3|\tilde{\mathcal{I}}|) |r(\phi)\phi_2\phi_3 \tilde{\mathcal{I}}^3| \max_{j=1,2,3}|\Delta_j|$$ and $\sqrt{v} = |\tilde{\mathcal{I}}|^3 r(\phi) \phi_2\phi_3 = |\tilde{\mathcal{I}}| m_2\phi_3 \leq m_2\max(1,\phi_3|\tilde{\mathcal{I}}|)$, thus
  \begin{align*}
    |\hat{g} - g|
    &\leq \frac{20 \max(1,\phi_3|\tilde{\mathcal{I}}|)}{m_2}\max_{j=1,2,3}|\Delta_j|%
    + \frac{2\max(1,\phi_3|\tilde{\mathcal{I}}|)}{m_2}|\Delta_2|
  \end{align*}
  concluding the proof.
\end{proof}

\subsection{Proof of Theorem \ref{thm:rough}}
\label{sec:proof:thm:rough}

In the whole proof, since $\tilde{\psi}_2$ is computed independently of the rest, we assume for convenience and without loss of generality that $\tilde{\psi}_2$ is non random  and we work implicitly conditional on $\tilde{\psi}_2$. It is assumed that $\tilde{\psi}_2$ satisfies the properties stated in the Theorem~\ref{thm:psitilde2}. The loss function is almost-surely bounded by $\check{T}^2$ so the contribution of estimating $\tilde{\psi}_2$ to the risk is easily deduced from the Theorem~\ref{thm:psitilde2}.

\subsubsection{Definitions and rationale}
\label{sec:defin-rati}

To avoid issues with the non-identifiability, we once again define $p_{\pm}$ and $f_{\pm}$ as in Lemma~\ref{lem:invert-param}. The starting point of the proof is to remark that $f_{\pm}$ can be rewritten as
\begin{equation*}
  f_{\pm}%
  = \Bigg[\frac{2 \psi_1}{1 \pm \tilde{s}\phi_1} \Bigg]%
  + \Bigg[- \Bigg( \frac{1 \mp \tilde{s}\phi_1}{1 \pm \tilde{s}\phi_1}\psi_1 \mp \frac{g(1 \mp \tilde{s}\phi_1)}{2m_1}G \Bigg)\Bigg].
\end{equation*}
Then each of the two functions in brackets in the previous display is estimated separately using block-thresholded wavelets estimators. The population mother coefficients are defined as
\begin{equation*}
  \AMother_{\pm}%
  \coloneqq
  \frac{2 \psi_1^{\Psi_{jk}}}{1 \pm \tilde{s}\phi_1},\qquad%
  \BMother_{\pm}%
  \coloneqq%
  - \Bigg( \frac{1 \mp \tilde{s}\phi_1}{1 \pm \tilde{s}\phi_1}\psi_1^{\Psi_{jk}} \mp \frac{g(1 \mp \tilde{s}\phi_1)}{2m_1}G^{\Psi_{jk}} \Bigg)
\end{equation*}
and the corresponding empirical versions are
\begin{equation*}
  \HatAMother_{\pm}\coloneqq \frac{2 \hat{\psi}_1^{\Psi_{jk}}}{1 \pm \hat{\phi}_1}\1_{\{\hat{\phi}_1 \ne \mp 1\}},\qquad%
  \HatBMother_{\pm}\coloneqq%
  - \Bigg(\frac{1 \mp \hat\phi_1}{1 \pm \hat\phi_1}\1_{\{\hat{\phi}_1 \ne \mp 1\}}\hat{\psi}_1^{\Psi_{jk}}%
  \mp \frac{\hat{g}(1\mp \hat\phi_1)}{2\hat{m}_1}\1_{\{ \hat{m}_1 \ne 0 \}} \hat{G}^{\Psi_{jk}}\Bigg).
\end{equation*}
Then, the untruncated estimators can be rewritten as (here $\HatFather_{\pm}$ are the father coefficients that were defined in the beginning of Section~\ref{sec:proof-thm:1})
\begin{align*}
  \hat{f}_{\pm}^R%
  &\coloneqq \sum_{k=0}^{2^{J_n}-1} \HatFather_{\pm} \Phi_{J_nk}%
  + \sum_{j=J}^{J_n-1}\sum_{k=0}^{2^j-1}\HatMother_{\pm}\Psi_{jk}\\%
  &\quad+ \sum_{j=J_n}^{\tilde{\jmath}_n}\sum_{\ell=0}^{2^j/N-1}\Bigg(\sum_{k\in \Block_{j\ell}}\HatAMother_{\pm} \Psi_{jk} \Bigg)\1_{ \{ \norm{\hat{\psi}_1^{\Block_{j\ell}}} > \Gamma\sqrt{\log(n)/n} \} }\\%
  &\quad+ \sum_{j=J_n}^{\tilde{\jmath}_n}\sum_{\ell=0}^{2^j/N-1}\Bigg(\sum_{k\in \Block_{j\ell}}\HatBMother_{\pm} \Psi_{jk} \Bigg)\1_{ \{  \norm{\hat{\B}_{\pm}^{\Block_{j\ell}}} > \Gamma\hat{T}_n \} }
\end{align*}
while the truncated versions are
\begin{equation*}
  \check{f}_{\pm}^R \coloneqq \max\big(0,\, \min\big(\check{T},\, \hat{f}_{\pm}^R\big)\big).
\end{equation*}

\subsubsection{Decomposition of the error}
\label{sec:decomposition-error}

We define auxiliary events $$\Xi_n^{(1)} \coloneqq \Set*{\forall j=J_n,\dots,\tilde{\jmath}_n,\ \forall \ell,\ \norm{\hat{\psi}_1^{\Block_{j\ell}} - \psi_1^{\Block_{j\ell}}} \leq c_0\Gamma \sqrt{\log(n)/n} },$$ and $$\Xi_n^{(2)} \coloneqq \Set*{\forall j=J_n,\dots,\tilde{\jmath}_n,\ \forall \ell,\ \norm{\hat{G}^{\Block_{j\ell}} - G^{\Block_{j\ell}}} \leq c_1\Gamma \sqrt{\log(n)/n} }$$. We let $\Xi_n$ denote the intersection of both of these events. Then by the same argument used in Section~\ref{sec:proof-thm:1}
\begin{align*}
  \EE_{\theta}\big(\norm{\check{f}^R_{\pm} - f_{\pm}}_{L^2}^2  \big)%
  &\leq \check{T}^2\big(\PP_{\theta}(\Omega_n^c) + \PP_{\theta}(\Xi_n^c) \big)%
    + \EE_{\theta}\big(\norm{ \hat{f}^R_{\pm} - f_{\pm}}_{L^2}^2\1_{\Omega_n\cap \Xi_n} \big).
\end{align*}
The probability of the event $\Omega_n^c$ is bounded in Proposition~\ref{pro:2}, while the probability of $\Xi_n^c$ is bounded in Lemma~\ref{lem:1} (to follow). We bound the remaining term by decomposing it into several terms. For this matter, we introduce the events $$E_{j\ell} \coloneqq \Set*{ \max_{j=1,2,3}|\hat{m}_j - m _j| \norm{G^{\Block_{j\ell}}} \leq c_2 |m_1m_2| \Gamma T_n / \max(1,g) }$$ and we decompose
\begin{align*}
  &\EE_{\theta}\big(\norm{ \hat{f}^R_{\pm} - f_{\pm}}_{L^2}^2\1_{\Omega_n\cap \Xi_n} \big)
    = \EE_{\theta}\Big(\norm{\hat{f}_{\pm}^{J_n} - f_{\pm}^{J_n}}_{L^2}^2\1_{\Omega_n\cap \Xi_n} \Big)\\
  &\quad%
    \begin{aligned}
      + \sum_{j=J_n}^{\tilde{\jmath}_n}\sum_{\ell}\EE_{\theta}\Big(
      &\norm[\big]{\hat{\A}_{\pm}^{\Block_{j\ell}}+
        \hat{\B}_{\pm}^{\Block_{j\ell}} - \A_{\pm}^{\Block_{j\ell}} -
        \B_{\pm}^{\Block_{j\ell}} }^2\\
          & \times \1_{\Omega_n\cap \Xi_n}
      \1_{n\norm{\hat{\psi}_1^{\Block_{j\ell}}}^2>\Gamma^2\log(n)
      }\1_{\norm{\hat{\B}_{\pm}^{\Block_{j\ell}}} > \Gamma \hat{T}_n}
      \1_{E_{j\ell}^c}\Big)
    \end{aligned}\\
  &\quad%
    \begin{aligned}
      + \sum_{j=J_n}^{\tilde{\jmath}_n}\sum_{\ell}\EE_{\theta}\Big(
      &\norm{\hat{\A}_{\pm}^{\Block_{j\ell}}+ \hat{\B}_{\pm}^{\Block_{j\ell}} - \A_{\pm}^{\Block_{j\ell}} - \B_{\pm}^{\Block_{j\ell}} }^2\\
      &\times \1_{\Omega_n\cap \Xi_n} 
    \1_{n\norm{\hat{\psi}_1^{\Block_{j\ell}}}^2 > \Gamma^2\log(n)}
      \1_{\norm{\hat{\B}_{\pm}^{\Block_{j\ell}}} > \Gamma \hat{T}_n}\1_{E_{j\ell}}\Big)
    \end{aligned}\\
  &\quad%
    \begin{aligned}
      + \sum_{j=J_n}^{\tilde{\jmath}_n}\sum_{\ell}\EE_{\theta}\Big(
      &\|\hat{\A}_{\pm}^{\Block_{j\ell}} - \A_{\pm}^{\Block_{j\ell}} - \B_{\pm}^{\Block_{j\ell}} \|^2\\
      &\times \1_{\Omega_n\cap \Xi_n} 
    \1_{n\|\hat{\psi}_1^{\Block_{j\ell}}\|^2 > \Gamma^2\log(n)}
      \1_{\|\hat{\B}_{\pm}^{\Block_{j\ell}}\| \leq \Gamma \hat{T}_n}\1_{E_{j\ell}}\1_{\|\psi_1^{\Block_{j\ell}}\| > \frac{g(1 \pm \tilde{s}\phi_1)}{|m_1|}\|G^{\Block_{j\ell}}\|}\Big)
    \end{aligned}\\
  &\quad%
    \begin{aligned}
      + \sum_{j=J_n}^{\tilde{\jmath}_n}\sum_{\ell}\EE_{\theta}\Big(
      &\|\hat{\A}_{\pm}^{\Block_{j\ell}} - \A_{\pm}^{\Block_{j\ell}} - \B_{\pm}^{\Block_{j\ell}} \|^2\\
      &\times \1_{\Omega_n\cap \Xi_n} 
    \1_{n\|\hat{\psi}_1^{\Block_{j\ell}}\|^2 > \Gamma^2\log(n)}
  \1_{\|\hat{\B}_{\pm}^{\Block_{j\ell}}\| \leq \Gamma \hat{T}_n}\1_{E_{j\ell}}\1_{\|\psi_1^{\Block_{j\ell}}\| \leq \frac{g(1 \pm \tilde{s}\phi_1)}{|m_1|}\|G^{\Block_{j\ell}}\|}\Big)
    \end{aligned}\\
  &\quad%
    + \sum_{j=J_n}^{\tilde{\jmath}_n}\sum_{\ell}\EE_{\theta}\big(\|\hat{\B}_{\pm}^{\Block_{j\ell}} - \A_{\pm}^{\Block_{j\ell}} - \B_{\pm}^{\Block_{j\ell}} \|^2\1_{\Omega_n\cap \Xi_n} 
    \1_{n\|\hat{\psi}_1^{\Block_{j\ell}}\|^2 \leq \Gamma^2\log(n)}
      \1_{\|\hat{\B}_{\pm}^{\Block_{j\ell}}\| > \Gamma \hat{T}_n}\1_{E_{j\ell}}\big)\\
    &\quad%
    + \sum_{j=J_n}^{\tilde{\jmath}_n}\sum_{\ell}\EE_{\theta}\big(\|\A_{\pm}^{\Block_{j\ell}} + \B_{\pm}^{\Block_{j\ell}} \|^2\1_{\Omega_n\cap \Xi_n}
    \1_{n\|\hat{\psi}_1^{\Block_{j\ell}}\|^2 \leq \Gamma^2\log(n)}
    \1_{\|\hat{\B}_{\pm}^{\Block_{j\ell}}\| \leq \Gamma \hat{T}_n}\1_{E_{j\ell}}\big)\\
  &\quad%
    + \sum_{j> \tilde{\jmath}_n}\sum_k|f_{\pm}^{\Psi_{jk}}|^2\PP_{\theta}(\Omega_n \cap \Xi_n)
\end{align*}
where we have used the same convention as in Section~\ref{sec:proof-thm:1} to define $\hat{f}_{\pm}^{J_n}$ and $f_{\pm}^{J_n}$. We call $R_1(\theta),\dots,R_8(\theta)$, respectively, each of the terms of the previous right hand side. In the next subsections, after stating a couple of preliminary results, we prove the following bounds, uniformly over $\theta \in \SmoothClass \cap \RegulClass$ and for a universal constant $B > 0$:
\begin{align*}
  R_1(\theta) &\leq \frac{B L^2}{\delta^2\epsilon^2\zeta^2}\frac{\log(n)}{n\sg}%
  + \frac{BL^3}{\delta^2\epsilon^4\zeta^4}\frac{1}{n\sg}%
  +\frac{B\max(\tau,L)^6}{\delta^2\epsilon^4\zeta^4}\frac{1}{(n\sg)^2}.\\%
  R_2(\theta) &\leq \frac{B}{\delta^2\epsilon^4\zeta^4}\Big( \frac{L^3}{n \gamma^*} + \frac{\max(\tau,L)^6}{(n\gamma^*)^2} \Big).\\%
  R_3(\theta) &\leq \frac{BR^2}{\min(1, s_{\mp})} \frac{1}{\delta^2} \Big( \frac{\Gamma^2 }{R^2 n \epsilon^2\zeta^2} \Big)^{2s_{\mp}/(2s_{\mp} + 1)}.\\%
  R_4(\theta) &\leq \frac{BR^2}{\min(1,s_{\pm})}\Big(\frac{\Gamma^2}{nR^2\delta^2} \Big)^{2s_{\pm}/(2s_{\pm} + 1)}%
      +
      \frac{BR^2}{\min(1, s_{\mp})} \frac{1}{\delta^2} \Big( \frac{\Gamma^2 }{R^2 n \epsilon^2\zeta^2} \Big)^{2s_{\mp}/(2s_{\mp} + 1)}.\\%
  R_5(\theta) &\leq \frac{B}{\delta^2\epsilon^4\zeta^4}\Big( \frac{L^3}{n \gamma^*} + \frac{\max(\tau,L)^6}{(n\gamma^*)^2} \Big) + \frac{BR^2}{\min(1,s_{\pm})}\Big(\frac{\Gamma^2}{nR^2\delta^2} \Big)^{2s_{\pm}/(2s_{\pm} + 1)}\\%
  &\quad%
      +
      \frac{BR^2}{\min(1, s_{\mp})} \frac{1}{\delta^2} \Big( \frac{\Gamma^2 }{R^2 n \epsilon^2\zeta^2} \Big)^{2s_{\mp}/(2s_{\mp} + 1)}\\%
  R_6(\theta) &\leq \frac{BR^2}{\min(1, s_{\mp})} \frac{1}{\delta^2} \Big( \frac{\Gamma^2 }{R^2 n \epsilon^2\zeta^2} \Big)^{2s_{\mp}/(2s_{\mp} + 1)}.\\%
  R_7(\theta) &\leq \frac{BR^2}{\min(1,s_{\pm})}\Big(\frac{\Gamma^2}{nR^2\delta^2} \Big)^{2s_{\pm}/(2s_{\pm} + 1)}%
      +
      \frac{R^2}{\min(1, s_{\mp})} \frac{1}{\delta^2} \Big( \frac{\Gamma^2 }{BR^2 n \epsilon^2\zeta^2} \Big)^{2s_{\mp}/(2s_{\mp} + 1)}.\\%
  R_8(\theta) &\leq  \frac{BR^2}{\min(1,s_{\pm})}\Big( \frac{\tau^2 \log(n)}{n} \Big)^{2s_{\pm}}.%
\end{align*}

\subsubsection{Preliminary computations}
\label{sec:prel-comp}

\begin{lemma}
  \label{lem:1}
  For all $A > 0$ and for all choice of $c_0,c_1> 0$ there exists a constant $\beta_0 >0$ such that if $\Gamma \geq \beta \max(\frac{L}{\sqrt{\gamma^{*}}}, \frac{\sqrt{L}}{\tau \gamma^{*}})$ with $\beta \geq \beta_0$ then
  \begin{equation*}
    \PP_{\theta}(\Xi_n^c) \leq n^{-A}.
  \end{equation*}
\end{lemma}
\begin{proof}
  By a union bound,
  \begin{align*}
    \PP_{\theta}\big((\Xi_n^{(1)})^c \big)%
    &\leq \sum_{j=J_n}^{\tilde{\jmath}_n}\sum_{\ell}\PP_{\theta}\Big(\|\hat{\psi}_1^{\Block_{j\ell}} - \psi_1^{\Block_{j\ell}}\| > c_0\Gamma \sqrt{\log(n)/n} \Big)\\
    &\leq \frac{2^{\tilde{\jmath}_n + 1}}{N}\max_{j\leq \tilde{\jmath}_n}\max_{\ell} \PP_{\theta}\Big(\|\hat{\psi}_1^{\Block_{j\ell}} - \psi_1^{\Block_{j\ell}}\| > c_0\Gamma \sqrt{\log(n)/n} \Big)\\
    &\leq n \max_{j\leq \tilde{\jmath}_n}\max_{\ell} \PP_{\theta}\Big(\|\hat{\psi}_1^{\Block_{j\ell}} - \psi_1^{\Block_{j\ell}}\| > c_0\Gamma \sqrt{\log(n)/n} \Big).
  \end{align*}
  Then choose $x = B \log(n)$ for some $B > 0$ to be chosen accordingly. Observe that for all $j \leq \tilde{\jmath}_n$ (recall $L \geq 1$)
  \begin{align*}
    C\sqrt{\frac{Lx}{n\gamma^{*}}}+ C2^{j/2} \frac{x}{n\gamma^{*}}%
    &\leq \frac{C\sqrt{BL}}{\sqrt{\gamma^{*}}}\cdot \sqrt{\frac{\log(n)}{n}}%
      + C \sqrt{\frac{n}{\log(n)\tau^2}} \frac{B \log(n)}{n\gamma^{*}}\\
    &\leq \frac{C\sqrt{B} + CB}{\beta}\Gamma \sqrt{\log(n)/n}.
  \end{align*}
  Hence by choosing $c_0 = (C\sqrt{B} + CB)/\beta$ we deduce from the Proposition~\ref{pro:8a} that
  \begin{equation*}
    \PP_{\theta}\big((\Xi_n^{(1)})^c \big)%
    \leq 24^N n^{-B + 1}.
  \end{equation*}
  The probability of $\Xi_n^{(2)}$ is bounded similarly, remarking that for $x = B \log(n)/n$ we have for all $j \leq \tilde{\jmath}_n$
  \begin{align*}
    &CL\sqrt{\frac{x}{n\gamma^{*}}} + C\max(\tau 2^{j/2},\sqrt{L}2^{j/2},\tau \sqrt{L})\frac{x}{n\gamma^{*}}\\
    &\qquad\qquad%
      \leq \frac{CL\sqrt{B}}{\sqrt{\gamma^{*}}} \sqrt{\frac{\log(n)}{n}}%
      + \frac{CB}{\gamma^{*}}\max\Bigg(\sqrt{\frac{n}{\log(n)}}, \frac{\sqrt{L}}{\tau}\sqrt{\frac{\log(n)}{n}}, \tau \sqrt{L} \Bigg)\frac{\log(n)}{n}\\
    &\qquad\qquad%
      \leq \frac{CL\sqrt{B}}{\sqrt{\gamma^{*}}} \sqrt{\frac{\log(n)}{n}}%
      + \frac{CB \sqrt{L}}{\gamma^{*}\tau} \sqrt{\frac{\log(n)}{n}}\\
    &\qquad\qquad%
      \leq \frac{C\sqrt{B} + CB}{\beta} \Gamma \sqrt{\log(n)/n},
  \end{align*}
  where the third line is true because by assumption $1 \leq 2^{J_n} \leq 2^{\tilde{\jmath}_n} \leq \frac{n}{\log(n)\tau^2}$ and hence $\tau \leq \sqrt{n/\log(n)}$ necessarily. We then deduce from Proposition~\ref{pro:8b} that
  \begin{equation*}
    \PP_{\theta}\big((\Xi_n^{(1)})^c \big)%
    \leq 4\cdot 24^N n^{-B+1}
  \end{equation*}
  which concludes the proof by taking $B$ sufficiently large.
\end{proof}

\begin{lemma}
  \label{lem:4}
  On the event $\Omega_n$
 \begin{equation*}
   \frac{1}{2} \leq  \frac{1 \pm \tilde{s}\phi_1}{1 \pm \hat{\phi}_1} \leq 2,\qquad\mathrm{and},\qquad%
   \frac{1}{2} \leq  \frac{1 \mp \tilde{s}\phi_1}{1 \mp \hat{\phi}_1} \leq 2.
 \end{equation*}
\end{lemma}
\begin{proof}
  Observe that
      \begin{equation*}
    \frac{1 \pm \tilde{s}\phi_1}{1 \pm \hat{\phi}_1}%
    = \frac{1}{1 + \frac{\hat{\phi}_1 - \tilde{s}\phi_1}{1 \pm \tilde{s} \phi_1 }}
  \end{equation*}
  But on the event $\Omega_n$, by Lemma~\ref{lem:10AA}
  \begin{align*}
    |\hat{\phi}_1 - \tilde{s}\phi_1|%
    &\leq \frac{53\max(1,g)}{gm_2 }\cdot \frac{1 - \phi_1^2}{4} \cdot \max_{j=1,2,3}|\hat{m}_j - m_j|\\
    &\leq \frac{53}{4\cdot 44}(1 \pm \tilde{s}\phi_1)(1 \mp \tilde{s}\phi_1)\\
    &\leq \frac{1 \pm \tilde{s}\phi_1}{2}
  \end{align*}
  which proves the first claim. The second claim is proven similarly.
\end{proof}

  \begin{lemma}
    \label{lem:7}
    On the event $\Omega_n$ we have $\hat{m}_1 \ne 0$ and $\hat{\phi}_1^2 \ne 1$.
  \end{lemma}
  \begin{proof}
    The fact that $\hat{m}_1 \ne 0$ follows immediately from the definition of $\Omega_n$. The fact that $\hat{\phi}_1^2 \ne 1$ follows from Lemma~\ref{lem:4} (either one of the two inequalities would not hold if $\hat{\phi}_1^2 = 1$).
  \end{proof}

The next Proposition controls the empirical threshold $\hat{T}_n$ in term of its population version defined as
\begin{equation*}
  T_n \coloneqq%
  \sqrt{\frac{\log(n)}{n}}\max\Big(1,\, \frac{g}{|m_1|},\, \frac{1}{1 - \phi_1^2}\Big).
\end{equation*}

\begin{lemma}
  \label{lem:6}
  On the event $\Omega_n$, $\frac{1}{4}T_n \leq \hat{T}_n \leq 4T_n$.
\end{lemma}
\begin{proof}
  Notice that $T_n = \max\big(S_n,\, \frac{\sqrt{\log(n)/n}}{1 - \phi_1^2}\big)$. Thus, in view of Proposition~\ref{pro:4} it is enough to show that $\frac{1 - \phi_1^2}{4} \leq 1 - \hat{\phi}_1^2 \leq 4(1 - \phi_1^2)$. But,
  \begin{equation*}
    1 - \hat{\phi}_1^2%
    = (1 \pm \hat{\phi}_1)(1 \mp \hat{\phi}_1)%
    = \frac{1 \pm \hat{\phi}_1}{1 \pm \tilde{s\phi_1}}%
    \frac{1 \mp \hat{\phi}_1}{1 \mp \tilde{s}\phi_1}(1 \mp \tilde{s}\phi_1)(1\pm \tilde{s}\phi_1)%
    = \frac{1 \pm \hat{\phi}_1}{1 \pm \tilde{s\phi_1}}%
    \frac{1 \mp \hat{\phi}_1}{1 \mp \tilde{s}\phi_1}%
    (1 - \phi_1^2).
  \end{equation*}
  Thus the conclusion follows from Lemma~\ref{lem:4}.
\end{proof}

\begin{lemma}
  \label{lem:5}
  It is possible to choose $c_0,c_1,c_2$ such that on the event $E_{j\ell} \cap \Xi_n \cap \Omega_n$:
  \begin{enumerate}
    \item\label{item:lem:5:1} $\|\hat{\B}_{\pm}^{\Block_{j\ell}}\| > \Gamma \hat{T}_n \implies \|\B_{\pm}^{\Block_{j\ell}}\| > \frac{1}{32}\Gamma T_n$;
    \item\label{item:lem:5:2} $\|\hat{\B}_{\pm}^{\Block_{j\ell}}\| \leq \Gamma \hat{T}_n \implies \|\B_{\pm}^{\Block_{j\ell}}\| \leq 32\Gamma T_n$;
    \item\label{item:lem:5:3} $\|\hat{\psi}_1^{\Block_{j\ell}}\| > \Gamma \sqrt{\log(n)/n} \implies \|\psi_1^{\Block_{j\ell}}\| > \frac{1}{2}\Gamma \sqrt{\log(n)/n}$;
    \item\label{item:lem:5:4} $\|\hat{\psi}_1^{\Block_{j\ell}}\| \leq \Gamma \sqrt{\log(n)/n} \implies \|\psi_1^{\Block_{j\ell}}\| \leq \frac{3}{2}\Gamma \sqrt{\log(n)/n}$.
  \end{enumerate}
\end{lemma}
\begin{proof}
  Before proving the items, we first remark that we never have $\hat{\phi}_1^2 = 1$ nor $\hat{m}_1 = 0$ on the event $\Omega_n$ thanks to Lemma~\ref{lem:7}.

  We establish Item~\ref{item:lem:5:1}. Notice that
  \begin{align*}
    \|\hat{\B}_{\pm}^{\Block_{j\ell}}\| > \Gamma \hat{T}_n%
    &\iff \Big\|\frac{1 \mp \hat{\phi}_1}{1 \pm \hat{\phi}_1}\hat{\psi}_1^{\Block_{j\ell}} \mp \frac{\hat{g}(1 \mp \hat{\phi}_1)}{2\hat{m}_1}\hat{G}^{\Block_{j\ell}} \Big\| > \Gamma \hat{T}_n\\
    &\iff \Big\|\frac{1 \mp \tilde{s}\phi_1}{1 \pm \tilde{s}\phi_1}\hat{\psi}_1^{\Block_{j\ell}} \mp  \frac{1 \pm \hat{\phi}_1}{1 \pm \tilde{s}\phi_1}\frac{\hat{g}(1 \mp \tilde{s}\phi_1)}{2\hat{m}_1}\hat{G}^{\Block_{j\ell}} \Big\| > \frac{1 \mp \tilde{s}\phi_1}{1 \pm \tilde{s}\phi_1}\frac{1 \pm \hat{\phi}_1}{1 \mp \hat{\phi}_1} \Gamma \hat{T}_n\\
    &\implies \Big\|\frac{1 \mp \tilde{s}\phi_1}{1 \pm \tilde{s}\phi_1}\hat{\psi}_1^{\Block_{j\ell}} \mp  \frac{1 \pm \hat{\phi}_1}{1 \pm \tilde{s}\phi_1}\frac{\hat{g}(1 \mp \tilde{s}\phi_1)}{2\hat{m}_1}\hat{G}^{\Block_{j\ell}} \Big\| > \frac{1}{16} \Gamma T_n
  \end{align*}
  ie.
  \begin{multline*}
    \|\hat{\B}_{\pm}^{\Block_{j\ell}}\| > \Gamma \hat{T}_n%
    \implies\\
    \|\B_{\pm}^{\Block_{j\ell}}\| > \frac{1}{16} \Gamma T_n - \frac{1 \mp \tilde{s}\phi_1}{1\pm \tilde{s}\phi_1}
      \|\hat{\psi}_1^{\Block_{j\ell}} - \psi_1^{\Block_{j\ell}}\|%
       - \Big\|\frac{1 \pm \hat{\phi}_1}{1 \pm \tilde{s}\phi_1}\frac{\hat{g}}{\hat{m}_1}\hat{G}^{\Block_{j\ell}} - \frac{g}{m_1}G^{\Block_{j\ell}} \Big\|
  \end{multline*}
  where we have used Lemmas~\ref{lem:4} and~\ref{lem:6}. But on the event $E_{j\ell} \cap \Xi_n \cap \Omega_n$
  \begin{equation*}
    \frac{1 \mp \tilde{s}\phi_1}{1\pm \tilde{s}\phi_1}
    \|\hat{\psi}_1^{\Block_{j\ell}} - \psi_1^{\Block_{j\ell}}\| \leq \frac{(1 \mp \tilde{s\phi_1})^2}{1 - \phi_1^2} \cdot c_0\Gamma \sqrt{\log(n)/n}%
    \leq c_0 \Gamma T_n
  \end{equation*}
  and
  \begin{align*}
    &\Big\|\frac{1 \pm \hat{\phi}_1}{1 \pm \tilde{s}\phi_1}\frac{\hat{g}}{\hat{m}_1}\hat{G}^{\Block_{j\ell}} - \frac{g}{m_1}G^{\Block_{j\ell}} \Big\|\\%
    &\leq \frac{1 \pm \hat{\phi}_1}{1 \pm \tilde{s}\phi_1} \frac{\hat{g}}{|\hat{m}_1|}\|\hat{G}^{\Block_{j\ell}} - G^{\Block_{j\ell}}\|%
      + \Big|\frac{1 \pm \hat{\phi}_1}{1 \pm \tilde{s}\phi_1}\frac{\hat{g}}{|\hat{m}_1|} - \frac{g}{m_1} \Big|\|G^{\Block_{j\ell}}\|\\
    &\leq \frac{1 \pm \hat{\phi}_1}{1 \pm \tilde{s}\phi_1} \frac{\hat{g}}{|\hat{m}_1|}\|\hat{G}^{\Block_{j\ell}} - G^{\Block_{j\ell}}\|\\
    &\qquad%
      + \Bigg(\frac{1 \pm \hat{\phi}_1}{1 \pm \tilde{s}\phi_1}\frac{|\hat{g} - g|}{|\hat{m}_1|}%
      + \frac{1 \pm \hat{\phi}_1}{1 \pm \tilde{s}\phi_1}\frac{g|\hat{m}_1 - m_1|}{|\hat{m}_1m_1|}%
      + \frac{g}{|m_1|}\Big| \frac{1 \pm \hat{\phi}_1}{1\pm \tilde{s}\phi_1} - 1 \Big|
      \Bigg)\|G^{\Block_{j\ell}}\|\\
    &\leq \frac{8 g}{|m_1|}\|\hat{G}^{\Block_{j\ell}} - G^{\Block_{j\ell}}\|%
      + \Bigg( \frac{4|\hat{g} - g|}{|m_1|}%
      + \frac{4g|\hat{m}_1 - m_1|}{m_1^2}%
      + \frac{g|\hat{\phi}_1 - \tilde{s}\phi_1|}{(1 - \phi_1^2)|m_1|}
      \Bigg)\|G^{\Block_{j\ell}}\|
  \end{align*}
  where the last line holds true on $\Omega_n$ by Lemmas~\ref{pro:4} and~\ref{lem:4}. Therefore by Lemmas~\ref{lem:10AA} and~\ref{lem:3}, there is a universal constant $C > 0$ such that
  \begin{multline*}
    \Big\|\frac{1 \pm \hat{\phi}_1}{1 \pm \tilde{s}\phi_1}\frac{\hat{g}}{\hat{m}_1}\hat{G}^{\Block_{j\ell}} - \frac{g}{m_1}G^{\Block_{j\ell}} \Big\|\\%
    \leq \frac{8 g}{|m_1|}\|\hat{G}^{\Block_{j\ell}} - G^{\Block_{j\ell}}\|%
    + \frac{C\max(1,g)}{|m_1m_2|}\max_{j=1,2,3}|\hat{m}_j - m_j|\|G^{\Block_{j\ell}}\|%
    \leq (8c_1 + Cc_2) \Gamma T_n
  \end{multline*}
  on the event $E_{j\ell}\cap \Xi_n \cap \Omega_n$ by definitions of these events. Therefore by choosing $c_0,c_1,c_2$ small enough, the Item~\ref{item:lem:5:1} claim follows. The proof of the Item~\ref{item:lem:5:2} is nearly identical. Items~\ref{item:lem:5:3} and~\ref{item:lem:5:4} are immediate from the definition of $\Xi_n$ provided $c_0 \leq 1/2$.
\end{proof}

In the next we make use of the symbol $\lesssim$ to denote inequalities that are valid up to a universal multiplicative constant. Furthermore, since $\hat{m}_1 \ne 0$ and $\hat{\phi}_1^2 \ne 1$ on the event $\Omega_n$ thanks to Lemma~\ref{lem:7}, and since all the terms we wish to control are conditional on $\Omega_n$, we will assume throughout the rest of the proof that $\hat{m}_1 \ne 0$ and $\hat{\phi}_1^2 \ne 1$ without justification.

\subsubsection{Control of $R_1$}

This has already been done in Section~\ref{sec:control-r_1}. We recall the result:
\begin{align*}
  \sup_{\theta \in \SmoothClass \cap \RegulClass}R_1(\theta)%
  \leq \frac{B L^2}{\delta^2\epsilon^2\zeta^2}\frac{\log(n)}{n\sg}%
  + \frac{BL^3}{\delta^2\epsilon^4\zeta^4}\frac{1}{n\sg}%
  +\frac{B\max(\tau,L)^6}{\delta^2\epsilon^4\zeta^4}\frac{1}{(n\sg)^2}.
\end{align*}

\subsubsection{Control of $R_2$}

\begin{align*}
  &\|\hat{\A}_{\pm}^{\Block_{j\ell}}\1_{\|\hat{\psi}_1^{\Block_{j\ell}}\|>\Gamma \sqrt{\log(n)/n} }+ \hat{\B}_{\pm}^{\Block_{j\ell}}\1_{\|\hat{\B}_{\pm}^{\Block_{j\ell}}\| > \Gamma \hat{T}_n} - \A_{\pm}^{\Block_{j\ell}} - \B_{\pm}^{\Block_{j\ell}} \|\\
  &\qquad%
    = \|\hat{\A}_{\pm}^{\Block_{j\ell}} + \hat{\B}_{\pm}^{\Block_{j\ell}} - \A_{\pm}^{\Block_{j\ell}} - \B_{\pm}^{\Block_{j\ell}}%
    - \hat{\A}_{\pm}^{\Block_{j\ell}}\1_{\|\hat{\psi}_1^{\Block_{j\ell}}\|\leq \Gamma \sqrt{\log(n)/n} }%
    - \hat{\B}_{\pm}^{\Block_{j\ell}}\1_{\|\hat{\B}_{\pm}^{\Block_{j\ell}}\|  \leq \Gamma \hat{T}_n}\|\\
  &\qquad%
    \leq \|\hat{\A}_{\pm}^{\Block_{j\ell}} + \hat{\B}_{\pm}^{\Block_{j\ell}} - \A_{\pm}^{\Block_{j\ell}} - \B_{\pm}^{\Block_{j\ell}}\|%
    + \frac{2\Gamma \sqrt{\log(n)/n}}{1 \pm \hat{\phi}_1}%
    + \Gamma \hat{T}_n\\
  &\qquad%
    \leq \|\hat{\A}_{\pm}^{\Block_{j\ell}} + \hat{\B}_{\pm}^{\Block_{j\ell}} - \A_{\pm}^{\Block_{j\ell}} - \B_{\pm}^{\Block_{j\ell}}\|%
    + 8 \Gamma T_n
\end{align*}
on the event $\Omega_n$ by Lemmas~\ref{lem:4} and \ref{lem:6}. Furthermore, letting $\hat{f}_{\pm}^{\Block_{j\ell}}$ and  $f_{\pm}^{\Block_{j\ell}}$ as defined in Section~\ref{sec:proof-thm:1}, it is easily seen that
\begin{align*}
  \hat{\A}_{\pm}^{\Block_{j\ell}} + \hat{\B}_{\pm}^{\Block_{j\ell}} - \A_{\pm}^{\Block_{j\ell}} - \B_{\pm}^{\Block_{j\ell}}%
  &= \hat{f}_{\pm}^{\Block_{j\ell}} - f_{\pm}^{\Block_{j\ell}}.
\end{align*}
Hence by Lemma~\ref{lem:pro:7}, on the event $\Xi_n \cap \Omega_n$,
\begin{align}
  \notag
  \|\hat{\A}_{\pm}^{\Block_{j\ell}} + \hat{\B}_{\pm}^{\Block_{j\ell}} - \A_{\pm}^{\Block_{j\ell}} - \B_{\pm}^{\Block_{j\ell}}\|%
  &\notag\leq c_0\Gamma \sqrt{\log(n)/n}%
    + \frac{4g}{|m_1|}c_1 \Gamma \sqrt{\log(n)/n}%
    + \frac{1}{2}|\hat{\omega}_{\pm} - \omega_{\pm}|\|G^{\Block_{j\ell}}\|\\
  &\notag\leq (c_0 + 4c_1)\Gamma T_n + \frac{1}{2}|\hat{\omega}_{\pm} - \omega_{\pm}|\|G^{\Block_{j\ell}}\|\\
  \label{eq:full-diff}
  &\leq (c_0 + 4c_1)\Gamma T_n + \frac{41.5 \max(1,g)}{|m_1m_2|}\max_{j=1,2,3}|\hat{m}_j - m_j|\|G^{\Block_{j\ell}}\|
\end{align}
where the last line follows by Proposition~\ref{pro:9}. Deduce from the definition of $E_{j\ell}$ that on the event $E_{j\ell}^c \cap \Xi_n \cap \Omega_n$ we must have
\begin{multline*}
  \|\hat{\A}_{\pm}^{\Block_{j\ell}}\1_{\|\hat{\psi}_1^{\Block_{j\ell}}\|>\Gamma \sqrt{\log(n)/n} }+ \hat{\B}_{\pm}^{\Block_{j\ell}}\1_{\|\hat{\B}_{\pm}^{\Block_{j\ell}}\| > \Gamma \hat{T}_n} - \A_{\pm}^{\Block_{j\ell}} - \B_{\pm}^{\Block_{j\ell}} \|\\
  \leq \Big(\frac{8 + c_0 + 4c_1}{c_2}+ 41.5 \Big)\frac{\max(1,g)}{|m_1m_2|}\max_{j=1,2,3}|\hat{m}_j - m_j|\|G^{\Block_{j\ell}}\|.
\end{multline*}
From this we obtain the estimate
\begin{align*}
  R_2(\theta)%
  &\lesssim%
    \frac{\max(1,g)^2}{m_1^2m_2^2} \EE_{\theta}\Big(\max_{j=1,2,3}|\hat{m}_j - m_j|^2 \Big) \sum_{j\geq J_n}\sum_{\ell}\|G^{\Block_{j\ell}}\|^2\\
  &\lesssim \frac{\max(1,g)^2}{m_2^2}\Big( \frac{C^2L^3}{n \gamma^*} + \frac{C^2\max(\tau,L)^6}{(n\gamma^*)^2} \Big)
\end{align*}
where the last line follows from Proposition~\ref{pro:expec-maxm}. Therefore we deduce from Lemma~\ref{lem:basic-relations} that
\begin{equation*}
  \sup_{\theta \in \SmoothClass \cap \RegulClass}R_2(\theta)%
  \lesssim%
  \frac{1}{\delta^2\epsilon^4\zeta^4}\Big( \frac{L^3}{n \gamma^*} + \frac{\max(\tau,L)^6}{(n\gamma^*)^2} \Big).
\end{equation*}

\subsubsection{Control of $R_3$}
\label{sec:control-r_3-1}

By equation~\eqref{eq:full-diff} and the definition of $E_{j\ell}$, it is found that on the event $E_{j\ell} \cap \Xi_n \cap \Omega_n$,
\begin{equation*}
  \|\hat{\A}_{\pm}^{\Block_{j\ell}} + \hat{\B}_{\pm}^{\Block_{j\ell}} - \A_{\pm}^{\Block_{j\ell}} - \B_{\pm}^{\Block_{j\ell}}\|%
  \leq (c_0 + 2c_1 + 41.5 c_2)\Gamma T_n.
\end{equation*}
Then we deduce from Lemma~\ref{lem:5} that
\begin{equation*}
  R_3(\theta)%
  \lesssim \Gamma^2 T_n^2 \sum_{j=J_n}^{\tilde{\jmath}_n}\sum_{\ell}\1_{ \{ \|\B_{\pm}^{\Block_{j\ell}}\| > \frac{1}{32}\Gamma T_n \} }.
\end{equation*}
Noting $\B_{\pm} = - \frac{1 \mp \tilde{s}\phi_1}{1 \pm \tilde{s}\phi_1} f_{\mp}$ and mimicking the proof in Section~\ref{sec:control-r_5}, it is found that
  \begin{equation}
    \label{eq:2:aa}
  \sum_{\ell}\1_{ \{ \|\B_{\pm}^{\Block_{j\ell}}\| > \frac{1}{32}\Gamma T_n \} }%
  \leq \min\Bigg(\frac{2^j}{N},\, \Big(\frac{1 \mp \tilde{s}\phi_1}{1 \pm \tilde{s}\phi_1}\Big)^2\frac{R^22^{-2j s_{\mp}}}{\Gamma^2T_n^2} \Bigg)
\end{equation}%
Letting $A = \sup\Set{ 0 \leq j \leq \tilde{\jmath}_n \given 2^{-j(s_{\mp}+ 1/2)}> \frac{\Gamma T_n}{R \sqrt{N}}\frac{1 \pm \tilde{s\phi_1}}{1 \mp \tilde{s}\phi_1} }$ it is found that
\begin{align*}
  R_3(\theta)%
  &\lesssim \Gamma^2T_n^2 \sum_{j=0}^A \frac{2^j}{N}%
    + \Gamma^2T_n^2 \sum_{j> A}\Big(\frac{1 \mp \tilde{s}\phi_1}{1 \pm \tilde{s}\phi_1}\Big)^2\frac{R^22^{-2j s_{\mp}}}{\Gamma^2T_n^2}\\
  &\lesssim \frac{\Gamma^2T_n^2}{N}2^A%
    + \Big(\frac{1 \mp \tilde{s}\phi_1}{1 \pm \tilde{s}\phi_1}\Big)^2 R^2 \frac{2^{-2A s_{\mp}}}{2^{2s_{\mp}} - 1}\\
  &\lesssim%
    \frac{\Gamma^2T_n^2}{N}\Bigg( \Big( \frac{1 \mp \tilde{s}\phi_1}{1 \pm \tilde{s}\phi_1} \Big)^2 \frac{R^2N}{\Gamma^2T_n^2} \Bigg)^{1/(2s_{\mp} + 1)}\\%
  &\quad%
    + \Big(\frac{1 \mp \tilde{s}\phi_1}{1 \pm \tilde{s}\phi_1}\Big)^2 R^2 \frac{1}{2^{2s_{\mp}} - 1}\Bigg(\Big( \frac{1 \pm \tilde{s}\phi_1}{1 \mp \tilde{s}\phi_1} \Big)^2\frac{\Gamma^2T_n^2}{R^2N} \Bigg)^{2s_{\mp}/(2s_{\mp}+1)}\\
  &\lesssim \frac{R^2}{\min(1, s_{\mp})} \Big( \frac{1 \mp \tilde{s}\phi_1}{1\pm \tilde{s}\phi_1} \Big)^{2/(2s_{\mp} + 1)} \Big( \frac{\Gamma^2T_n^2}{R^2N} \Big)^{2s_{\mp}/(2s_{\mp} + 1)}.
\end{align*}
It follows using the definition of $T_n$ and $\SmoothClass$ together with Lemma~\ref{lem:basic-relations} (recall that $\zeta \leq 1$ by assumption) that 
\begin{equation*}
  \sup_{\theta \in \SmoothClass \cap \RegulClass}R_3(\theta)%
  \lesssim%
  \frac{R^2}{\min(1, s_{\mp})} \frac{1}{\delta^2} \Big( \frac{\Gamma^2 }{R^2 n \epsilon^2\zeta^2} \Big)^{2s_{\mp}/(2s_{\mp} + 1)}.
\end{equation*}

\subsubsection{Control of $R_{4}$}

When $\|\psi_1^{\Block_{j\ell}}\| > \frac{g(1 \pm \tilde{s}\phi_1)}{|m_1|}\|G^{\Block_{j\ell}}\|$
\begin{align*}
  \|\B_{\pm}^{\Block_{j\ell}}\|%
  &= \Big\|\frac{1 \mp \tilde{s}\phi_1}{1 \pm \tilde{s}\phi_1}\psi_1^{\Block_{j\ell}} \mp \frac{g(1 \mp \tilde{s}\phi_1)}{2m_1}G^{\Block_{j\ell}} \Big\|\\
  &\geq \frac{1 \mp \tilde{s}\phi_1}{1 \pm \tilde{s}\phi_1}\|\psi_1^{\Block_{j\ell}}\| - \frac{g(1 \mp \tilde{s}\phi_1)}{2|m_1|}\|G^{\Block_{j\ell}}\|\\
  &\geq \frac{1}{2}\frac{1 \mp \tilde{s}\phi_1}{1 \pm \tilde{s}\phi_1}\|\psi_1^{\Block_{j\ell}}\|.
\end{align*}
Consequently,
\begin{align*}
  \|\hat{\A}_{\pm}^{\Block_{j\ell}} - \A_{\pm}^{\Block_{j\ell}} - \B_{\pm}^{\Block_{j\ell}} \|%
  &= \Big\|\frac{2}{1 \pm \hat{\phi}_1}\hat{\psi}_1^{\Block_{j\ell}} - \frac{2}{1 \pm \tilde{s}\phi_1}\psi_1^{\Block_{j\ell}} + \Big(\frac{1 \mp \tilde{s}\phi_1}{1 \pm \tilde{s}\phi_1}\psi_1^{\Block_{j\ell}} \mp \frac{g(1 \mp \tilde{s}\phi_1)}{2m_1}G^{\Block_{j\ell}} \Big) \Big\|\\
  &= \Big\|\frac{2(\hat{\psi}_1^{\Block_{j\ell}} - \psi_1^{\Block_{j\ell}})}{1 \pm \hat{\phi}_1} + \Big( \frac{1 \mp \hat{\phi}_1}{1 \pm \hat{\phi}_1}\psi_1^{\Block_{j\ell}} \mp  \frac{g(1 \mp \tilde{s}\phi_1)}{2m_1}G^{\Block_{j\ell}} \Big) \Big\|\\
  &\leq \frac{2\|\hat{\psi}_1^{\Block_{j\ell}} - \psi_1^{\Block_{j\ell}}\|}{1 \pm \hat{\phi}_1}%
    + \frac{1 \mp \hat{\phi}_1}{1 \pm \hat{\phi}_1}\|\psi_1^{\Block_{j\ell}}\|%
    + \frac{g(1 \mp \tilde{s}\phi_1)}{2|m_1|}\|G^{\Block_{j\ell}}\|
\end{align*}
Then on the event $\Xi_n\cap \Omega_n$, by Lemma~\ref{lem:4}
\begin{align*}
  \|\hat{\A}_{\pm}^{\Block_{j\ell}} - \A_{\pm}^{\Block_{j\ell}} - \B_{\pm}^{\Block_{j\ell}} \|%
  &\leq \frac{4c_0\Gamma \sqrt{\log(n)/n}}{1 \pm \tilde{s}\phi_1}%
    + \frac{1 \mp \tilde{s}\phi_1}{1 \pm \tilde{s}\phi_1}\Big(4 \|\psi_1^{\Block_{j\ell}}\| + \frac{g(1 \pm \tilde{s}\phi_1)}{2|m_1|}\|G^{\Block_{j\ell}}\| \Big)\\
  &\leq \frac{4c_0\Gamma \sqrt{\log(n)/n}}{1 \pm \tilde{s}\phi_1}%
    + 5\frac{1 \mp \tilde{s}\phi_1}{1 \pm \tilde{s}\phi_1}\|\psi_1^{\Block_{j\ell}}\|\\
  &\leq \frac{4c_0\Gamma \sqrt{\log(n)/n}}{1 \pm \tilde{s}\phi_1}%
    + 10 \|\B_{\pm}^{\Block_{j\ell}}\|.
\end{align*}
Deduce from Lemma~\ref{lem:5} that
\begin{align*}
  R_4(\theta)%
  &\lesssim%
    \frac{\Gamma^2 \log(n)/n}{(1 \pm \tilde{s}\phi_1)^2}\sum_{j=J_n}^{\tilde{\jmath}_n}\sum_{\ell}\1_{ \{ \|\psi_1^{\Block_{j\ell}}\| > \frac{1}{2}\Gamma \sqrt{\log(n)/n} \} }%
    + \sum_{j=J_n}^{\tilde{\jmath}_n}\sum_{\ell}\|\B_{\pm}^{\Block_{j\ell}}\|^2\1_{ \{ \|\B_{\pm}^{\Block_{j\ell}}\| \leq 32 \Gamma T_n \} }.
\end{align*}
Observe that $2\psi_1 = (1 + \tilde{s}\phi_1)f_+ + (1 - \tilde{s}\phi_1)f_-$. Therefore, for all $j \geq J_n$
  \begin{align}
    \notag
    \sum_k|\psi_1^{\Psi_{jk}}|^2%
    &\leq \frac{(1 + \tilde{s}\phi_1)^2}{2}\sum_k|f_+^{\Psi_{jk}}|^2%
      + \frac{(1 - \tilde{s}\phi_1)^2}{2}\sum_k|f_-^{\Psi_{jk}}|^2\\
    \label{eq:rough:estimatepsi1}%
    &\leq R^2 \frac{(1 + \tilde{s}\phi_1)^22^{-2js_+}+  (1 - \tilde{s}\phi_1)^22^{-2js_-}}{2},
  \end{align}
  whenever $\theta \in \SmoothClass$ (recall equation~\eqref{eq:def:besov}). Deduce that (see also Section~\ref{sec:control-r_5})
  \begin{align*}
    \sum_{\ell}\1_{ \{ \|\psi_1^{\Block_{j\ell}}\| > \frac{1}{2}\Gamma \sqrt{\log(n)/n} \} }%
    &\leq \min\Bigg(\frac{2^j}{N},\, \frac{2n R^2\big((1 + \tilde{s}\phi_1)^22^{-2js_+} +  (1 - \tilde{s}\phi_1)^22^{-2js_-} \big)}{\Gamma^2 \log(n)} \Bigg)\\
    &\leq\frac{1}{2}\min\Big(\frac{2^j}{N},\, \frac{4n R^2(1 + \tilde{s}\phi_1)^2 2^{-2js_+}}{\Gamma^2 \log(n)} \Big)\\%
    &\quad%
      + \frac{1}{2}\min\Big(\frac{2^j}{N},\, \frac{4n R^2(1 - \tilde{s}\phi_1)^2 2^{-2js_-}}{\Gamma^2 \log(n)} \Big)
  \end{align*}
  by convexity of $x\mapsto \min(2^j/N,x)$. Deduce that,
  \begin{align*}
    &\frac{\Gamma^2 \log(n)/n}{(1 \pm \tilde{s}\phi_1)^2}\sum_{j=J_n}^{\tilde{\jmath}_n}\sum_{\ell}\1_{ \{ \|\psi_1^{\Block_{j\ell}}\| > \frac{1}{2}\Gamma \sqrt{\log(n)/n} \} }\\
    &\qquad\lesssim \frac{\Gamma^2}{n(1 \pm \tilde{s}\phi_1)^2} \Big( \frac{n R^2(1 + \tilde{s}\phi_1)^2}{\Gamma^2} \Big)^{1/(2s_+ + 1)}\\%
    &\qquad\quad
      + \frac{1}{2^{2s_+} - 1}\frac{R^2(1 + \tilde{s}\phi_1)^2}{(1 \pm \tilde{s}\phi_1)^2}\Big(\frac{\Gamma^2}{nR^2(1 + \tilde{s}\phi_1)^2} \Big)^{2s_+/(2s_+ + 1)}\\
    &\qquad\quad%
      + \frac{\Gamma^2}{n(1 \pm \tilde{s}\phi_1)^2} \Big( \frac{n R^2(1 - \tilde{s}\phi_1)^2}{\Gamma^2} \Big)^{1/(2s_- + 1)}\\%
    &\qquad\quad%
      + \frac{1}{2^{2s_-} - 1}\frac{R^2(1 - \tilde{s}\phi_1)^2}{(1 \pm \tilde{s}\phi_1)^2}\Big(\frac{\Gamma^2}{nR^2(1 - \tilde{s}\phi_1)^2} \Big)^{2s_-/(2s_- + 1)}.
  \end{align*}
  That is,
  \begin{align*}
    &\frac{\Gamma^2 \log(n)/n}{(1 \pm \tilde{s}\phi_1)^2}\sum_{j=J_n}^{\tilde{\jmath}_n}\sum_{\ell}\1_{ \{ \|\psi_1^{\Block_{j\ell}}\| > \frac{1}{2}\Gamma \sqrt{\log(n)/n} \} }\\
    &\qquad\lesssim%
      \frac{R^2}{\min(1,s_+)}\Big(\frac{1 + \tilde{s}\phi_1}{1 \pm \tilde{s}\phi_1} \Big)^2%
      \Big( \frac{\Gamma^2}{nR^2(1+\tilde{s}\phi_1)^2} \Big)^{2s_+/(2s_+ + 1)}\\%
    &\qquad\quad%
      + \frac{R^2}{\min(1,s_-)}\Big(\frac{1 - \tilde{s}\phi_1}{1 \pm \tilde{s}\phi_1} \Big)^2%
      \Big( \frac{\Gamma^2}{nR^2(1-\tilde{s}\phi_1)^2} \Big)^{2s_-/(2s_- + 1)}.
  \end{align*}
  Regarding the remaining term, recall that $\beta_{\pm} = - \frac{1 \mp \tilde{s}\phi_1}{1 \pm \tilde{s}\phi_1}f_{\mp}$ and observe that
  \begin{align}
    \notag
    \sum_{j=J_n}^{\tilde{\jmath}_n}\sum_{\ell}\|\B_{\pm}^{\Block_{j\ell}}\|^2\1_{ \{ \|\B_{\pm}^{\Block_{j\ell}}\|\leq 32\Gamma T_n \} }%
    &\lesssim%
      \sum_{j=J_n}^{\tilde{\jmath}_n}\min\Bigg(\sum_{\ell}\|\B_{\pm}^{\Block_{j\ell}}\|^2,\, \frac{2^j\Gamma^2T_n^2}{N}\Bigg)\\
    \notag
    &\lesssim \sum_{j=J_n}^{\tilde{\jmath}_n}\min\Bigg(\sum_{\ell}\|\B_{\pm}^{\Block_{j\ell}}\|^2,\, \frac{2^j\Gamma^2T_n^2}{N}\Bigg)\\
    \notag
    &\lesssim%
      \sum_{j=J_n}^{\tilde{\jmath}_n}\min\Bigg(R^2\Big(\frac{1 \mp \tilde{s}\phi_1}{1 \pm \tilde{s}\phi_1} \Big)^2 2^{-2js_{\mp}},\, \frac{2^j\Gamma^2T_n^2}{N}\Bigg)\\
    \label{eq:sumsmallbetas}
    &\lesssim  \frac{R^2}{\min(1, s_{\mp})} \Big( \frac{1 \mp \tilde{s}\phi_1}{1\pm \tilde{s}\phi_1} \Big)^{2/(2s_{\mp} + 1)} \Big( \frac{\Gamma^2T_n^2}{R^2N} \Big)^{2s_{\mp}/(2s_{\mp} + 1)}
  \end{align}
  where the last line follows from the estimate in \eqref{eq:2:aa} and subsequent iterates. In the end,
  \begin{align*}
    R_4(\theta)%
    &\lesssim \frac{R^2}{\min(1,s_+)}\Big(\frac{1 + \tilde{s}\phi_1}{1 \pm \tilde{s}\phi_1} \Big)^2%
      \Big( \frac{\Gamma^2}{nR^2(1+\tilde{s}\phi_1)^2} \Big)^{2s_+/(2s_+ + 1)}\\%
    &\quad%
      + \frac{R^2}{\min(1,s_-)}\Big(\frac{1 - \tilde{s}\phi_1}{1 \pm \tilde{s}\phi_1} \Big)^2%
      \Big( \frac{\Gamma^2}{nR^2(1-\tilde{s}\phi_1)^2} \Big)^{2s_-/(2s_- + 1)}\\
    &\quad%
      + \frac{R^2}{\min(1, s_{\mp})} \Big( \frac{1 \mp \tilde{s}\phi_1}{1\pm \tilde{s}\phi_1} \Big)^{2/(2s_{\mp} + 1)} \Big( \frac{\Gamma^2T_n^2}{R^2N} \Big)^{2s_{\mp}/(2s_{\mp} + 1)}.
  \end{align*}
  Taking the suprema of each terms, with the help of Lemma~\ref{lem:basic-relations} it is found that
  \begin{align*}
    \sup_{\theta \in \SmoothClass \cap \RegulClass}R_4(\theta)%
    &\lesssim%
    \frac{R^2}{\min(1,s_{\pm})}\Big(\frac{\Gamma^2}{nR^2\delta^2} \Big)^{2s_{\pm}/(2s_{\pm} + 1)}\\%
    &\quad+ \frac{R^2}{\min(1,s_{\mp})}\frac{1}{\delta^2}\Big(\frac{\Gamma^2}{nR^2} \Big)^{2s_{\mp}/(2s_{\mp} + 1)}\\
    &\quad%
      +
      \frac{R^2}{\min(1, s_{\mp})} \frac{1}{\delta^2} \Big( \frac{\Gamma^2 }{R^2 n \epsilon^2\zeta^2} \Big)^{2s_{\mp}/(2s_{\mp} + 1)}.
  \end{align*}
  Namely,
  \begin{align*}
    \sup_{\theta \in \SmoothClass \cap \RegulClass}R_4(\theta)%
    &\lesssim%
    \frac{R^2}{\min(1,s_{\pm})}\Big(\frac{\Gamma^2}{nR^2\delta^2} \Big)^{2s_{\pm}/(2s_{\pm} + 1)}\\%
     &\quad +
      \frac{R^2}{\min(1, s_{\mp})} \frac{1}{\delta^2} \Big( \frac{\Gamma^2 }{R^2 n \epsilon^2\zeta^2} \Big)^{2s_{\mp}/(2s_{\mp} + 1)}.
  \end{align*}

\subsubsection{Control of $R_{5}$}

When $\|\psi_1^{\Block_{j\ell}}\| \leq \frac{g(1 \pm \tilde{s}\phi_1)}{|m_1|}\|G^{\Block_{j\ell}}\|$,
\begin{align*}
  \|\hat{\A}_{\pm}^{\Block_{j\ell}} - \A_{\pm}^{\Block_{j\ell}} - \B_{\pm}^{\Block_{j\ell}}\|
  &\leq \|\hat{\A}_{\pm}^{\Block_{j\ell}} - \A_{\pm}^{\Block_{j\ell}}\| + \|\B_{\pm}^{\Block_{j\ell}}\|\\
  &= \Big\|\frac{2}{1 \pm \hat{\phi}_1}\hat{\psi}_1^{\Block_{j\ell}} - \frac{2}{1 \pm \tilde{s}\phi_1}\psi_1^{\Block_{j\ell}} \Big\| + \|\B_{\pm}^{\Block_{j\ell}}\|\\
  &\leq \frac{2}{1 \pm \hat{\phi}_1}\|\hat{\psi}_1^{\Block_{j\ell}} - \psi_1^{\Block_{j\ell}}\| + 2\|\psi_1^{\Block_{j\ell}}\|\Big| \frac{1}{1 \pm \hat{\phi}_1} - \frac{1}{1 \pm \tilde{s}\phi_1}\Big| + \|\B_{\pm}^{\Block_{j\ell}}\|\\
  &\leq \frac{2}{1 \pm \hat{\phi}_1}\|\hat{\psi}_1^{\Block_{j\ell}} - \psi_1^{\Block_{j\ell}}\| + 2\|\psi_1^{\Block_{j\ell}}\| \frac{|\hat{\phi}_1 - \tilde{s}\phi_1|}{(1 \pm \hat{\phi}_1)(1 \pm \tilde{s}\phi_1)} + \|\B_{\pm}^{\Block_{j\ell}}\|
\end{align*}
So by Lemmas~\ref{lem:10AA} and~\ref{lem:4}, it holds on the event $E_{j\ell} \cap \Xi_n \cap \Omega_n$
\begin{align*}
  &\|\hat{\A}_{\pm}^{\Block_{j\ell}} - \A_{\pm}^{\Block_{j\ell}} - \B_{\pm}^{\Block_{j\ell}}\|\\%
  &\qquad\leq \frac{4c_0\Gamma \sqrt{\log(n)/n}}{1 \pm \tilde{s}\phi_1l
    }%
    + \frac{800 \max(1,g)}{\phi_2^2\phi_3^2\tilde{\mathcal{I}}^2 g}\frac{\max_{j=1,2,3}|\hat{m}_j - m_j|}{(1 \pm \tilde{s}\phi_1)^2}\|\psi_1^{\Block_{j\ell}}\| + \|\B_{\pm}^{\Block_{j\ell}}\|\\
  &\qquad\leq \frac{4c_0\Gamma \sqrt{\log(n)/n}}{1 \pm \tilde{s}\phi_1}%
    + \frac{800 \max(1,g)}{\phi_2^2\phi_3^2\tilde{\mathcal{I}}^2}\frac{\max_{j=1,2,3}|\hat{m}_j - m_j|}{|m_1|(1 \pm \tilde{s}\phi_1)}\|G^{\Block_{j\ell}}\| + \|\B_{\pm}^{\Block_{j\ell}}\|\\
  &\qquad\leq \frac{4c_0\Gamma \sqrt{\log(n)/n}}{1 \pm \tilde{s}\phi_1}%
    + \frac{800 \max(1,g)}{|m_1m_2|}\max_{j=1,2,3}|\hat{m}_j - m_j|\|G^{\Block_{j\ell}}\| + \|\B_{\pm}^{\Block_{j\ell}}\|.
\end{align*}
From here, it is seen that an upper bound on the supremum of $R_5$ is obtained by adding the bounds obtained on $R_2$ together with the bound on $R_4$, eventually up to a universal multiplicative constant.

\subsubsection{Control of $R_6$}

\begin{align*}
  &\|\hat{\B}_{\pm}^{\Block_{j\ell}} - \A_{\pm}^{\Block_{j\ell}} - \B_{\pm}^{\Block_{j\ell}} \|\\%
  &\qquad\leq \|\hat{\B}_{\pm}^{\Block_{j\ell}} - \B_{\pm}^{\Block_{j\ell}} \|%
    + \|\A_{\pm}^{\Block_{j\ell}}\|\\
  &\qquad= \Big\|\frac{1 \mp \hat{\phi}_1}{1\pm \hat{\phi}_1}\hat{\psi}_1^{\Block_{j\ell}} \mp \frac{\hat{g}(1 \mp \hat{\phi}_1)}{2\hat{m}_1}\hat{G}^{\Block_{j\ell}} - \Big(\frac{1 \mp \tilde{s}\phi_1}{1 \pm \tilde{s}\phi_1}\psi_1^{\Block_{j\ell}} \mp \frac{g(1 \mp \tilde{s}\phi_1)}{2m_1}G^{\Block_{j\ell}} \Big) \Big\|\\%
  &\qquad\quad%
    + \frac{2}{1 \pm \tilde{s}\phi_1}\|\psi_1^{\Block_{j\ell}}\|\\
  &\qquad\leq \frac{3}{1\pm \tilde{s}\phi_1}\|\psi_1^{\Block_{j\ell}}\|%
    + \frac{1}{1\pm \hat{\phi}_1}\|\hat{\psi}_1^{\Block_{j\ell}}\|%
    + \Big\|\frac{\hat{g}(1 \mp \hat{\phi}_1)}{2\hat{m}_1}\hat{G}^{\Block_{j\ell}} - \frac{g(1 \mp \tilde{s}\phi_1)}{2m_1}G^{\Block_{j\ell}} \Big\|
\end{align*}
but by Proposition~\ref{pro:9} on the event $\Omega_n$ we have
\begin{align*}
  \Big\|\frac{\hat{g}(1 \mp \hat{\phi}_1)}{2\hat{m}_1}\hat{G}^{\Block_{j\ell}} - \frac{g(1 \mp \tilde{s}\phi_1)}{2m_1}G^{\Block_{j\ell}} \Big\|%
  &= |\hat{\omega}_{\mp}|\|\hat{G}^{\Block_{j\ell}} - G^{\Block_{j\ell}}\| + |\hat{\omega}_{\mp} - \omega_{\mp}|\|G^{\Block_{j\ell}}\|\\
  &\lesssim \frac{g}{|m_1|}\|\hat{G}^{\Block_{j\ell}} - G^{\Block_{j\ell}}\|\\%
  &\quad
    + \frac{\max(1,g)}{|m_1m_2|}\max_{j=1,2,3}|\hat{m}_j - m_j|\|G^{\Block_{j\ell}}\|.
\end{align*}
Therefore on the event $E_{j\ell}\cap \Xi_n\cap \Omega_n$
\begin{align*}
  \|\hat{\B}_{\pm}^{\Block_{j\ell}} - \A_{\pm}^{\Block_{j\ell}} - \B_{\pm}^{\Block_{j\ell}} \|%
  &\lesssim \frac{\|\psi_1^{\Block_{j\ell}}\| + \|\hat{\psi}_1^{\Block_{j\ell}}\|}{1 \pm \tilde{s}\phi_1}%
    + \frac{g}{|m_1|} c_1\Gamma \sqrt{\log(n)/n} + c_2\Gamma T_n\\
  &\leq  \frac{\|\psi_1^{\Block_{j\ell}}\| + \|\hat{\psi}_1^{\Block_{j\ell}}\|}{1 \pm \tilde{s}\phi_1}%
    + (c_1 + c_2)\Gamma T_n.
\end{align*}
Deduce by Lemma~\ref{lem:5} that
\begin{equation*}
  R_6(\theta)%
  \lesssim \Gamma^2T_n^2 \sum_{j=J_n}^{\tilde{\jmath}_n}\sum_{\ell}\1_{\|\B_{\pm}^{\Block_{j\ell}}\| > \frac{1}{32}\Gamma T_n}.
\end{equation*}
Therefore, $R_6(\theta)$ admits the same upper bound as $R_3(\theta)$, eventually up to a universal multiplicative factor.

\subsubsection{Control of $R_7$}

\begin{align*}
  \|\A_{\pm}^{\Block_{j\ell}} + \B_{\pm}^{\Block_{j\ell}}\|%
  &\leq \|\A_{\pm}^{\Block_{j\ell}}\| + \|\B_{\pm}^{\Block_{j\ell}}\|
  = \frac{2}{1 \pm \tilde{s}\phi_1}\|\psi_1^{\Block_{j\ell}}\|%
    + \|\B_{\pm}^{\Block_{j\ell}}\|
\end{align*}
Therefore, we obtain from Lemma~\ref{lem:5} that
\begin{align*}
  R_7(\theta)%
  &\leq \frac{2}{(1\pm\tilde{s}\phi_1)^2}\sum_{j=J_n}^{\tilde{\jmath}_n}\sum_{\ell}\|\psi_1^{\Block_{j\ell}}\|^2\1_{\|\psi_1^{\Block_{j\ell}}\| \leq \frac{3}{2}\Gamma\sqrt{\log(n)/n} }\\%
  &\quad
    + 2\sum_{j=J_n}^{\tilde{\jmath}_n}\sum_{\ell}\|\B_{\pm}^{\Block_{j\ell}}\|^2\1_{\|\B_{\pm}^{\Block_{j\ell}}\| \leq 32 \Gamma T_n }
\end{align*}
From equation~\eqref{eq:rough:estimatepsi1},
\begin{align*}
  &\frac{2}{(1\pm\tilde{s}\phi_1)^2}\sum_{j=J_n}^{\tilde{\jmath}_n}\sum_{\ell}\|\psi_1^{\Block_{j\ell}}\|^2\1_{\|\psi_1^{\Block_{j\ell}}\| \leq \frac{3}{2}\Gamma\sqrt{\log(n)/n} }\\
  &\qquad\leq%
    \frac{2}{(1\pm\tilde{s}\phi_1)^2}\sum_{j=J_n}^{\tilde{\jmath}_n}\min\Bigg(
    \frac{9\Gamma^2 \log(n)}{4n} \frac{2^j}{N},\, \sum_{\ell}\|\psi_1^{\Block_{j\ell}}\|^2\Bigg)\\
  &\qquad%
    \lesssim \frac{1}{(1\pm \tilde{s}\phi_1)^2}\frac{\Gamma^2 \log(n)}{n}\sum_{j=J_n}^{\tilde{\jmath}_n}\min\Bigg(\frac{2^j}{N},\, nR^2\frac{(1+\tilde{s}\phi_1)^22^{-2js_+} +(1 - \tilde{s}\phi_1)^22^{-2js_-}}{\Gamma^2 \log(n)} \Bigg)
\end{align*}%
Then deduce from the series of estimates after \eqref{eq:rough:estimatepsi1} that
\begin{multline*}
  \frac{2}{(1\pm\tilde{s}\phi_1)^2}\sum_{j=J_n}^{\tilde{\jmath}_n}\sum_{\ell}\|\psi_1^{\Block_{j\ell}}\|^2\1_{\|\psi_1^{\Block_{j\ell}}\| \leq \frac{3}{2}\Gamma\sqrt{\log(n)/n} }\\
  \lesssim%
  \frac{R^2}{\min(1,s_+)}\Big(\frac{1 + \tilde{s}\phi_1}{1 \pm \tilde{s}\phi_1} \Big)^2%
  \Big( \frac{\Gamma^2}{nR^2(1+\tilde{s}\phi_1)^2} \Big)^{2s_+/(2s_+ + 1)}\\%
  + \frac{R^2}{\min(1,s_-)}\Big(\frac{1 - \tilde{s}\phi_1}{1 \pm \tilde{s}\phi_1} \Big)^2%
  \Big( \frac{\Gamma^2}{nR^2(1-\tilde{s}\phi_1)^2} \Big)^{2s_-/(2s_- + 1)}.
\end{multline*}
Next, it has been already established in \eqref{eq:sumsmallbetas} that
\begin{equation*}
  \sum_{j=J_n}^{\tilde{\jmath}_n}\sum_{\ell}\|\B_{\pm}^{\Block_{j\ell}}\|^2\1_{\|\B_{\pm}^{\Block_{j\ell}}\| \leq 32 \Gamma T_n }%
  \lesssim%
  \frac{R^2}{\min(1, s_{\mp})} \Big( \frac{1 \mp \tilde{s}\phi_1}{1\pm \tilde{s}\phi_1} \Big)^{2/(2s_{\mp} + 1)} \Big( \frac{\Gamma^2T_n^2}{R^2N} \Big)^{2s_{\mp}/(2s_{\mp} + 1)}.
\end{equation*}
Consequently, when passing to the supremum, $R_7$ will obey the same upper bound as $R_4$, eventually up to a universal multiplicative constant.

\subsubsection{Control of $R_8$}

This has already been done in Section~\ref{sec:control-r_6}. We recall the result:
\begin{equation*}
  \sup_{\theta\in \SmoothClass \cap \RegulClass}R_8(\theta)%
  \leq \frac{BR^2}{\min(1,s_{\pm})}\Big( \frac{\tau^2 \log(n)}{n} \Big)^{2s_{\pm}}.
\end{equation*}

\subsection{Proof of Theorem~\ref{thm:psitilde2}}
\label{sec:proof-theorem-psitilde2}

Recall $\tilde{V}$ is the leading eigenvector of the empirical Gram matrix $\tilde{\mathcal{G}}$ and $V_{\theta}$ the leading eigenvector of the Gram matrix $\mathcal{G}$ normalized such that $\|\tilde{V}\| = \|V_{\theta}\| = 1$. We use a Davis-Kahan argument to bound the norm $\|\tilde{V} - \sign(\Inner{\tilde{V},V_{\theta}})V_{\theta}\|$. In particular using the version of Davis-Kahan's theorem given in the Corollary~1 of \citep{MR3371006}, we know that
\begin{equation*}
  \|\tilde{V} - \sign(\Inner{\tilde{V},V_{\theta}})V_{\theta}\|%
  \leq \frac{2 \sqrt{2}\|\tilde{\mathcal{G}} - \mathcal{G}\|_{\textrm{op}}}{|\lambda|}
\end{equation*}
where $\lambda$ is the unique non-zero eigenvalue of $\mathcal{G}$, and $\|\cdot\|_{\textrm{op}}$ stands for the operator norm. It is rapidly seen that
\begin{equation*}
  \lambda = r(\phi) \sum_{\lambda\in \Lambda(M)}\Inner{\psi_2,e_{\lambda}}^2%
  = r(\phi)\Bigg(\sum_{k=0}^{2^J-1}\Inner{\psi_2,\Phi_{Jk}}^2%
  + \sum_{j=J}^M\sum_{k=0}^{2^j-1}\Inner{\psi_2,\Psi_{jk}}^2\Bigg).
\end{equation*}
We now bound $\|\tilde{\mathcal{G}} - \mathcal{G}\|_{\textrm{op}}$. By definition of the operator norm and then by a duality argument [here $U$ denotes the unit ball of $\Reals^{\Lambda(M)}$]
\begin{align*}
  \|\tilde{\mathcal{G}} - \mathcal{G}\|_{\textrm{op}}%
  &= \sup_{u\in U}\|\tilde{\mathcal{G}}u - \mathcal{G}u\|\\%
  &= \sup_{u\in U}\sup_{v\in U} v^T(\tilde{\mathcal{G}} - \mathcal{G})u\\
  &= \sup_{u\in U}\sup_{v\in U} \Big[\Big(\frac{u+v}{2}\Big)^T(\tilde{\mathcal{G}} - \mathcal{G})\frac{u+v}{2} - \Big(\frac{u-v}{2}\Big)^T(\tilde{\mathcal{G}} - \mathcal{G})\frac{u-v}{2} \Big]\\
  &\leq \sup_{u\in U}\sup_{v\in U} \Big[u^T(\tilde{\mathcal{G}} - \mathcal{G})u - v^T(\tilde{\mathcal{G}} - \mathcal{G})v \Big]\\
  &\leq 2\sup_{u\in U} u^T(\tilde{\mathcal{G}} - \mathcal{G})u.
\end{align*}

Then, let $\mathcal{N}$ be a $(1/8)$-net over $U$ in the euclidean norm, and let $\pi : U \to \mathcal{N}$ denote the map that projects elements of $U$ onto their closest element in $\mathcal{N}$. Then,
\begin{align*}
  \sup_{u\in U} u^T(\tilde{\mathcal{G}} - \mathcal{G})u%
  &= \sup_{u\in U}\Big[\pi(u)^T(\tilde{\mathcal{G}} - \mathcal{G})\pi(u)%
    + 2\pi(u)^T(\tilde{\mathcal{G}} - \mathcal{G})(u - \pi(u))\\%
  &\qquad\qquad%
    + (u - \pi(u))^T(\tilde{\mathcal{G}} - \mathcal{G})(u-\pi(u))
    \Big]\\
  &\leq \max_{u \in \mathcal{N}}u^T(\tilde{\mathcal{G}} - \mathcal{G})u%
    + \frac{3}{8}\|\tilde{\mathcal{G}} - \mathcal{G}\|_{\textrm{op}}
\end{align*}
and thus
\begin{equation*}
  \|\tilde{\mathcal{G}} - \mathcal{G}\|_{\textrm{op}}%
  \leq 8\max_{u \in \mathcal{N}}u^T(\tilde{\mathcal{G}} - \mathcal{G})u.
\end{equation*}
Next, we decompose $\tilde{\mathcal{G}} - \mathcal{G} = \Delta^{(1)} + \Delta^{(2)} + \Delta^{(3)} + \Delta^{(4)}$ with 
\begin{align*}
  \Delta^{(1)}_{\lambda\lambda'}
  &\coloneqq \frac{1}{2}\Big(\tilde{\PP}_n^{(1)}(e_{\lambda}\otimes e_{\lambda'} + e_{\lambda'}\otimes e_{\lambda}) - \EE_{\theta}(e_{\lambda}\otimes e_{\lambda'} + e_{\lambda'}\otimes e_{\lambda})  \Big)\\
  \Delta^{(2)}_{\lambda\lambda'}
  &\coloneqq - \EE_{\theta}(e_{\lambda'})\Big( \tilde{\PP}_n^{(1)}(e_{\lambda}) - \EE_{\theta}(e_{\lambda}) \Big)\\
  \Delta^{(3)}_{\lambda\lambda'}%
  &\coloneqq - \EE_{\theta}(e_{\lambda})\Big(\tilde{\PP}_n^{(1)}(e_{\lambda'}) - \EE_{\theta}(e_{\lambda'}) \Big)\\
  \Delta^{(4)}_{\lambda\lambda'}%
  &\coloneqq - \Big(\tilde{\PP}_n^{(1)}(e_{\lambda}) - \EE_{\theta}(e_{\lambda}) \Big)\Big(\tilde{\PP}_n^{(1)}(e_{\lambda'}) - \EE_{\theta}(e_{\lambda'}) \Big)
\end{align*}
Using Lemma~\ref{lem:paulin} applied to the function $h(y_1,y_2) = \frac{1}{2}\sum_{\lambda,\lambda'\in \Lambda(M)}u_{\lambda}u_{\lambda'}\big(e_{\lambda}(y_1)e_{\lambda'}(y_2) + e_{\lambda'}(y_1)e_{\lambda}(y_2) \big)$ we find that
\begin{align*}
  \PP_{\theta}\Big( \max_{u\in \mathcal{N}}|u^T\Delta^{(1)}u| \geq x \Big)%
  &\leq |\mathcal{N}| \max_{u\in |\mathcal{N}|}\PP_{\theta}\Big(|u^T\Delta^{(1)}u| \geq x \Big)\\
  &\leq 24^{2^M} \exp\Bigg(- \frac{Cn\sg x^2}{L^2 + 2^M x} \Bigg)
\end{align*}
because $\mathcal{N}$ can always be chosen to have cardinality no more than $24^{2^M}$ \citep[e.g.][Theorem 4.3.34]{GN16}, because $\EE_{\theta}(h^2) \leq L^2\norm{h}_{L^2}^2 = L^2$ for all $\theta \in \RegulClass$ by Lemma~\ref{lem:2}, and because
\begin{align*}
  \|h\|_{\infty}%
  &\leq \sup_{y_1,y_2}\Big|\sum_{\lambda\in \Lambda(M)}u_{\lambda}e_{\lambda}(y_1) \sum_{\lambda'\in \Lambda(M)}u_{\lambda'}e_{\lambda'}(y_2) \Big|\\
  &\leq  \Big( \sup_y\sum_{\lambda \in \Lambda(M)}\big|e_{\lambda}(y)\big|\Big)^2\\
  &\leq c 2^M
\end{align*}
for a constant $c>0$ depending only on the wavelet basis by a standard localization properties of wavelets  \citep[Theorem~4.2.10 or Definition~4.2.14]{GN16}. Next, note that
\begin{align*}
  u^T\Delta^{(2)}u = u^T\Delta^{(3)}u%
  = -\EE_{\theta}\Bigg(\sum_{\lambda\in \Lambda(M)}u_{\lambda} e_{\lambda} \Bigg)\Bigg(\sum_{\lambda \in \Lambda(M)}u_{\lambda}\Big(\tilde{\PP}_n^{(1)}(e_{\lambda}) - \EE_{\theta}(e_{\lambda}) \Big) \Bigg)
\end{align*}
and,
\begin{equation*}
  u^T\Delta^{(4)}u = - \Bigg(\sum_{\lambda\in \Lambda(M)}u_{\lambda}\Big(\tilde{\PP}_n^{(1)}(e_{\lambda}) - \EE_{\theta}(e_{\lambda}) \Big) \Bigg)^2.
\end{equation*}
Again using Lemma~\ref{lem:paulin}, this time applied to the function $h(y) = \sum_{\lambda\in \Lambda(M)}u_{\lambda} e_{\lambda}(y)$ which satisfies $\EE_{\theta}(h^2) \leq L$ for all $\theta \in \RegulClass$ and $\|h\|_{\infty} \leq c2^{M/2}$ for a universal constant $c>0$, we deduce that
\begin{equation*}
  \PP_{\theta}\Bigg(\max_{u\in \mathcal{N}}\Big|\sum_{\lambda\in \Lambda(M)}u_{\lambda}\Big(\tilde{\PP}_n^{(1)}(e_{\lambda}) - \EE_{\theta}(e_{\lambda}) \Big)\Big| \geq x \Bigg)%
  \leq 24^{2^{M}}\exp\Bigg(- \frac{Cn \sg x^2}{L + 2^{M/2}x} \Bigg).
\end{equation*}
Since $\abs{\EE_\theta h } \leq [\EE_\theta h^2]^{1/2}\leq \sqrt{L}$, 
using that $L,2^{M/2}\geq 1$,  we deduce that 
\begin{equation*}
  \PP_{\theta}\Big( \tfrac{1}{8}\norm{\tilde{\mathcal{G}} - \mathcal{G}}_{\textrm{op}} \geq (2\sqrt{L}+1)x + x^2 \Big)%
  \leq 2 \cdot 24^{2^M}\exp\Bigg(- \frac{Cn \sg x^2}{L^2 + 2^M x} \Bigg)
\end{equation*}
for a constant $C > 0$. This entails that
\begin{equation*}
  \PP_{\theta}\Bigg( \|\tilde{V} - \sign(\Inner{\tilde{V},V_{\theta}})V_{\theta}\|%
  \geq \frac{16\sqrt{2}\big( (2\sqrt{L}+1)x + x^2 \big)}{|r(\phi)| \sum_{\lambda\in\Lambda(M)}\Inner{\psi_2,e_{\lambda}}^2} \Bigg)%
  \leq 2 \cdot 24^{2^M}\exp\Bigg(- \frac{Cn\sg x^2}{L^2 + 2^M x} \Bigg)
\end{equation*}

Let us remark that the wavelets coefficients of $\psi_2$ are those of $(f_0 - f_1)/\phi_3$. Hence, whenever $\theta \in \SmoothClass$, from the definition of $\SmoothClass$ and of the Besov norm in equation~\eqref{eq:def:besov} it must be that 
\begin{equation}\label{eqn:BesovBound}
	\sup_{j\geq J}2^{2js_{*}}\sum_{k=0}^{2^j-1} \abs{\Inner{\psi_2,\Psi_{jk}}}^2 \leq \frac{4R^2}{\phi_3^2},
\end{equation}
Consequently since $\norm{\psi_2}_{L^2} = 1$:
\begin{align*}
  1
  &= \sum_{k=0}^{2^J-1}\Inner{\psi_2,\Phi_{Jk}}^2 + \sum_{j\geq J}\sum_{k=0}^{2^j-1}\Inner{\psi_2,\Psi_{jk}}^2\\
  &\leq \sum_{k=0}^{2^J-1}\Inner{\psi_2,\Phi_{Jk}}^2 +
    \sum_{j=J}^M\sum_{k=0}^{2^j-1}\Inner{\psi_2,\Psi_{jk}}^2%
    + \frac{4R^2}{\phi_3^2}\sum_{j>M}2^{-2j s_{*}}\\
  &= \sum_{\lambda\in \Lambda(M)}\Inner{\psi_2,e_{\lambda}}^2 + \frac{4R^2}{\phi_3^2}\frac{2^{-2M s_{*}}}{2^{2s_{*}}-1}.
\end{align*}
and hence $\sum_{\lambda\in \Lambda(M)}\Inner{\psi_2,e_{\lambda}}^2 \geq 3/4$ under the assumptions of the theorem. Observe that $\abs{r(\phi)} \leq \phi_3^2/4 \leq L/2$ by Lemmas~\ref{lem:ub-phi3} and~\ref{lem:bounds-parameters}. Then taking $x = \kappa \abs{r(\phi)}/\sqrt{L}$ for a small enough constant $\kappa$, we find that for some $C>0$
\begin{equation*}
  \PP_{\theta}\Bigg( \|\tilde{V} - \sign(\Inner{\tilde{V},V_{\theta}})V_{\theta}\| \geq \frac{1}{5}\Bigg)%
  \leq 2 \cdot 24^{2^M}\exp\Bigg(- \frac{Cn\sg r(\phi)^2}{L^3 + 2^M\sqrt{L}\abs{r(\phi)}} \Bigg).
\end{equation*}
Next, let define $t \coloneqq \sum_{\lambda \in \Lambda(M)}\tilde{V}_{\lambda}e_{\lambda}$ and $f(x) \coloneqq \max(-\tau,\min(\tau,x))$. Observe that
\begin{equation*}
  \|\psi_2\|_{\infty}%
  =\frac{\|f_0 - f_1\|_{\infty}}{\phi_3}%
  \leq \frac{L}{\zeta}
\end{equation*}
since $0 \leq f_0,f_1\leq L$ and $\phi_3 \geq \zeta$ when $\theta \in \SmoothClass\cap \RegulClass$. Then by assumption $\abs{\psi_2(x)} \leq \tau$ for all $x$, and thus $\psi_2(x) = f(\psi_2(x))$. Also $f$ is $1$-Lipschitz, and thus
\begin{equation*}
  \norm{f\circ t - \tilde{s}\psi_2}_{L^2} = \norm{f\circ t - f\circ (\tilde{s}\psi_2)}_{L^2} \leq \norm{t - \tilde{s} \psi_2}_{L^2}%
  = \norm{\tilde{V} - \sign\brackets[\big]{\Inner{\tilde{V},V_{\theta}}}V_{\theta} }.
\end{equation*}
Since $\tilde{\psi}_2 = f\circ t/\norm{f\circ t}_{L^2}$, we use that for any norm $\norm{a/\norm{a}-b/\norm{b}}\leq 2\norm{a-b}/(1-\norm{a-b})$ if $\norm{b}=1$, $\norm{a-b}<1$ to deduce that
\begin{align*}
  \norm{\tilde{\psi}_2 - \tilde{s}\psi_2}_{L^2}%
  \leq \frac{2\norm{\tilde{V} - \sign\brackets[\big]{\Inner{\tilde{V},V_{\theta}}}V_{\theta}} }{1 - \norm{\tilde{V} - \sign\brackets[\big]{\Inner{\tilde{V},V_{\theta}}} V_{\theta} }}.
\end{align*}
The conclusion follows since $\norm{\tilde{\psi}_2 - \tilde{s}\psi_2}_{L^2}^2 = 2 - 2\abs{\Inner{\tilde{\psi}_2,\psi_2}}$, and hence $\abs{\Inner{\tilde{\psi}_2,\psi_2}} \geq 1 - \frac{\norm{\tilde{\psi}_2 - \tilde{s}\psi_2}_{L^2}^2}{2}$.

\subsection{Proof of Corollary~\ref{cor:smooth-ub}}
\label{sec:proof-cor:smooth-ub}

Suppose $2^M = O(1)$, then $\frac{n\gamma^2\delta^2\epsilon^2\zeta^4}{L^3 + 2^M\sqrt{L}\delta\epsilon\zeta^2} \gtrsim n^{1-2a-2b-2c}$ so that the first exponential in the bound of Theorem~\ref{thm:1} is smaller than $\exp(-Kn^{1-2a-2b-2c})$ for some $K >0$, which is negligible. If $2^M$ is not $O(1)$, then in the considered regime $\frac{n\gamma^2\delta^2\epsilon^2\zeta^4}{L^3 + 2^M\sqrt{L}\delta\epsilon\zeta^2} \gtrsim n 2^{-M}\delta \epsilon \zeta^2 \gg 2^M$ so that the first exponential in the bound of Theorem~\ref{thm:1} is smaller than $\exp(-Kn^{(1-a-b-2c)/2})$ for some $K > 0$, which is negligible.

Also $n\delta^2\epsilon^4\zeta^6 \geq n^{1-2a-4b-6c}$ while $L^3 + \max(\tau,\sqrt{L})^3\delta\epsilon^2\zeta^3 \leq L^3 + \max(\tau,\sqrt{L})^3$ since $\delta\epsilon^2\zeta^3 \leq 1$. Hence, the second exponential term in the bound of Theorem~\ref{thm:1} is smaller than $\exp(-K n^{1-2a-4b-6c})$ for some $K > 0$ and is negligible.

We claim that the term $\frac{1}{\delta^2\epsilon^2\zeta^2}\frac{\log(n)}{n}$ never dominates. Indeed, for this term to dominate, it is necessary that $\epsilon^2\zeta^2 \gg \frac{1}{\log(n)}$ to dominate the term $\frac{1}{\delta^2\epsilon^4\zeta^4n}$ and that $\delta^2\epsilon^2\zeta^2n = O(\log(n)^{2s_i+1})$ to dominate the term $(\delta^2\epsilon^2\zeta^2n)^{-2s_i/(2s_i+1)}$, \textit{ie}. $\epsilon^2\zeta^2 = O(\frac{\log(n)^{2s_i+1}}{n\delta^2} ) = O(\frac{\log(n)^{2s_i+1}}{n^{1-2a}})$. Since $1-2a > 0$, the two requirements cannot be fulfilled simultaneously for $n$ large.

Finally, the term $\frac{1}{\delta^2\epsilon^4\zeta^6n^2}$ is clearly dominated by the term $\frac{1}{\delta^2\epsilon^4\zeta^6n}$ and the remaining term is clearly dominated by the term $(\delta^2\epsilon^2\zeta^2n)^{-2s_i/(2s_i+1)}$.

\subsection{Proof of Corollary~\ref{cor:rough-ub}}
\label{sec:proof-cor:rough-ub}

As for the proof of Corollary~\ref{cor:smooth-ub} the two first exponential terms in the bound of Theorem~\ref{thm:rough} cannot dominate in the considered regime. It has been shown in Corollary~\ref{cor:smooth-ub} that the term $\frac{\log(n)}{\delta^2\epsilon^2\zeta^2n}$ cannot simultaneously dominate the terms $\frac{1}{\delta^2\epsilon^4\zeta^4n}$ and $\delta^{-2}(n\epsilon^2\zeta^2)^{-2s_1/(2s_1+1)}$ [observe that $\delta^{-2}(n\epsilon^2\zeta^2)^{-2s_1/(2s_1+1)} \geq (n\delta^2\epsilon^2\zeta^2)^{-2s_1/(2s_1+1)}$]. Also using the arguments in the proof of Corollary~\ref{cor:smooth-ub} it is trivial that the terms $\frac{1}{\delta^2\epsilon^4\zeta^4n^2}$ and $(\log(n)/n)^{2s_0}$ cannot dominate.

To finish the proof, it is enough to show that the term $\delta^{-2}(n\epsilon^2\zeta^2)^{-2s_1/(2s_1+1)}$ is dominated by the term $(n\delta^2)^{-2s_0/(2s_0+1)}$. But in the considered regime $\delta^{-2}(n\epsilon^2\zeta^2)^{-2s_1/(2s_1+1)} = n^{-2s_1/(2s_1+1) + o(1)}$ and $(n\delta^2)^{-2s_0/(2s_0+1)} = n^{-2s_0/(2s_0+1) + o(1)}$. The conclusion follows since $s_1 > s_0$ by assumption.

\bibliography{bibliography}

\end{document}